\documentclass[12pt]{amsart}

\usepackage[top=30mm, bottom=30mm, left=30mm, right=30mm]{geometry}
\usepackage{subfig}
\usepackage{caption}

\usepackage{wrapfig}
\usepackage{pinlabel, color}
\usepackage{multirow}

\usepackage{amsthm} 
\usepackage{amsmath} 
\usepackage{amssymb} 

\usepackage{microtype} 
\usepackage{pinlabel} 

\usepackage[
hidelinks,
pagebackref]{hyperref}

\renewcommand*{\backref}[1]{}
\renewcommand*{\backrefalt}[4]{
  \ifcase #1 
  [No citations.]
  \or [#2]
  \else [#2]
  \fi }

\let\originalleft\left
\let\originalright\right
\renewcommand{\left}{\mathopen{}\mathclose\bgroup\originalleft}
\renewcommand{\right}{\aftergroup\egroup\originalright}

\newcommand{\thsup}{{\rm th}}

\newcommand{\calC}{\mathcal{C}}
\newcommand{\calD}{\mathcal{D}}

\newcommand{\calG}{\mathcal{G}}

\newcommand{\calT}{\mathcal{T}}

\newcommand{\CC}{\mathbb{C}}

\newcommand{\HH}{\mathbb{H}}
\newcommand{\II}{\mathbb{I}}

\newcommand{\RR}{\mathbb{R}}

\newcommand{\ZZ}{\mathbb{Z}}

\newcommand{\Id}{\operatorname{Id}}

\newcommand{\st}{\mathbin{\mid}} 
\newcommand{\from}{\colon} 

\newcommand{\isom}{\cong} 

\newcommand{\bdy}{\partial} 

\newcommand{\subgp}[1]{{\langle #1 \rangle}}

\newcommand{\Isom}{\operatorname{Isom}} 

\newcommand{\SO}{\operatorname{SO}} 

\numberwithin{equation}{section} 

\makeatletter
\let\c@figure\c@equation
\makeatother
\numberwithin{figure}{section} 

\theoremstyle{plain}
\newtheorem{theorem}[equation]{Theorem}
\newtheorem{corollary}[equation]{Corollary}
\newtheorem{lemma}[equation]{Lemma}

\newtheorem{proposition}[equation]{Proposition}

\theoremstyle{definition}
\newtheorem{definition}[equation]{Definition}

\newtheorem{exercise}[equation]{Exercise}

\newtheorem*{question*}{Question}
\newtheorem*{answer*}{Answer}
\newtheorem*{application*}{Application}

\theoremstyle{remark}
\newtheorem{remark}[equation]{Remark}
\newtheorem*{remark*}{Remark}

\newtheorem*{case*}{Case}
\newtheorem*{step*}{Step}

\newtheorem*{claim*}{Claim}

\newcommand{\refsec}[1]{Section~\ref{Sec:#1}}
\newcommand{\refapp}[1]{Appendix~\ref{App:#1}}

\newcommand{\refcor}[1]{Corollary~\ref{Cor:#1}}
\newcommand{\reflem}[1]{Lemma~\ref{Lem:#1}}
\newcommand{\refprop}[1]{Proposition~\ref{Prop:#1}}

\newcommand{\refrem}[1]{Remark~\ref{Rem:#1}}

\newcommand{\reffig}[1]{Figure~\ref{Fig:#1}}
\newcommand{\refdef}[1]{Definition~\ref{Def:#1}}

\newcommand{\refeqn}[1]{Equation~\ref{Eqn:#1}}

\newcommand{\fakeenv}{} 

{ 
 \renewcommand{\fakeenv}{#2} 
 \theoremstyle{plain} 
 \newtheorem*{\fakeenv}{#1~\ref{#2}} 
 \begin{\fakeenv}
}
{
 \end{\fakeenv}
}

\newenvironment{restated}[2]  
{ 
 \renewcommand{\fakeenv}{#2} 
 \theoremstyle{definition} 
 \newtheorem*{\fakeenv}{#1~\ref{#2}} 
 \begin{\fakeenv}
}
{
 \end{\fakeenv}
}

\usepackage{mathtools}
\usepackage{booktabs}

\newcommand{\Sym}{\operatorname{Sym}}
\newcommand{\Cong}[1]{{\overline{#1}}}
\newcommand{\basis}[1]{{\langle {#1} \rangle}}
\newcommand{\centerpt}{\operatorname{center}}
\newcommand{\eq}{{\operatorname{eq}}}
\newcommand{\inn}{{\operatorname{in}}}
\newcommand{\out}{{\operatorname{out}}}

\newcommand{\hull}{\operatorname{hull}}
\newcommand{\Vor}{{\operatorname{Vor}}}
\newcommand{\Sphere}{{\operatorname{Sph}}}

\begin{document}

\title[$120$--cell]{Puzzling the $120$--cell}

\author[Schleimer]{Saul Schleimer}
\address{\hskip-\parindent
        Department of Mathematics\\
        University of Warwick\\
        Coventry, UK}
\email{s.schleimer@warwick.ac.uk}

\author[Segerman]{Henry Segerman}
\address{\hskip-\parindent
        Department of Mathematics \\
        Oklahoma State University \\
        Stillwater, OK USA}
\email{segerman@math.okstate.edu}

\thanks{This work is in the public domain.}

\date{\today}

\begin{abstract}
We introduce \emph{Quintessence}: a family of burr puzzles based on
the geometry and combinatorics of the $120$--cell.  We discuss the
regular polytopes, their symmetries, the dodecahedron as an important
special case, the three-sphere, and the quaternions.  We then
construct the $120$--cell, giving an illustrated survey of its
geometry and combinatorics.  This done, we describe the pieces out of
which Quintessence is made.  The design of our puzzle pieces uses a
drawing technique of Leonardo da Vinci; the paper ends with a catalogue
of new puzzles.
\end{abstract}


\maketitle

\begin{figure}[htbp]
\centering 
\includegraphics[width = 0.75\textwidth]{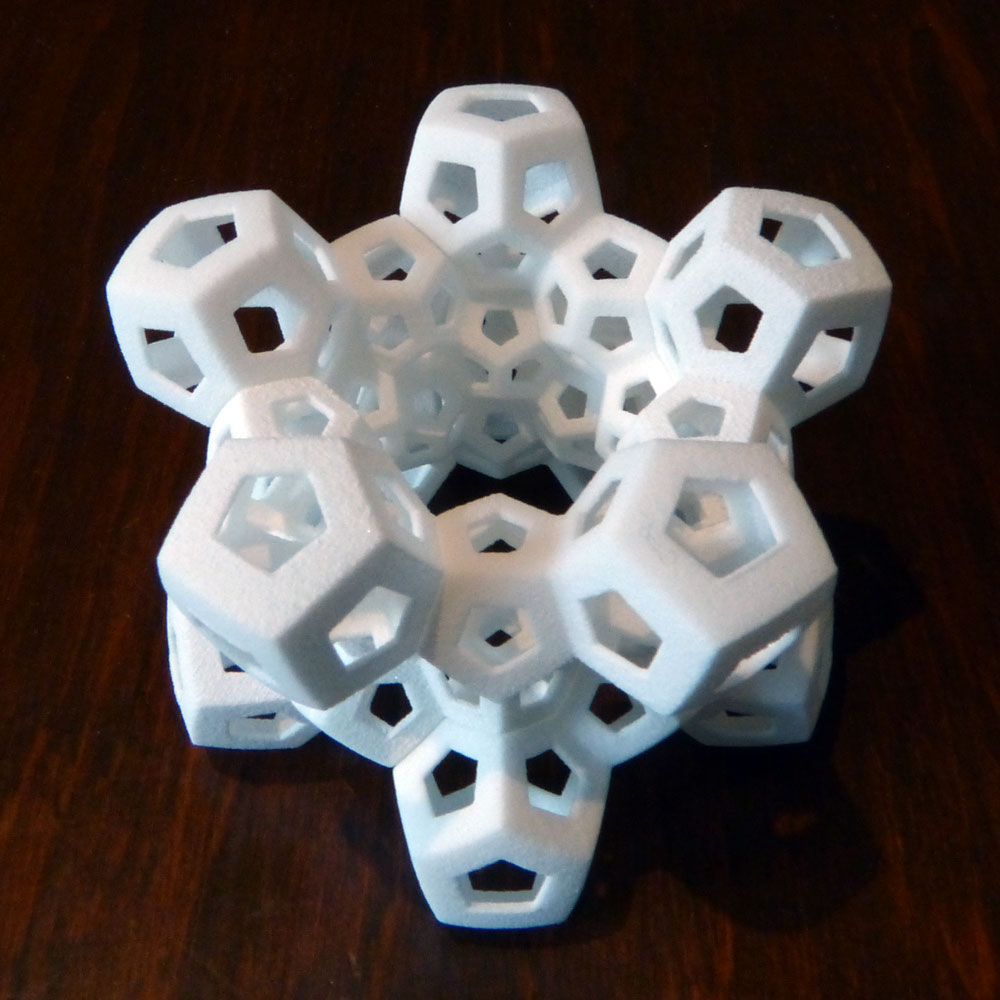}
\caption{The Dc30 Ring, one of the simpler puzzles in Quintessence.}
\label{Fig:Ring}
\end{figure}

\section{Introduction}
\label{Sec:Intro}

\begin{wrapfigure}[10]{r}{0.42\textwidth}
\vspace{-10pt}
\centering 
\subfloat[]
{
\includegraphics[width=0.16\textwidth]{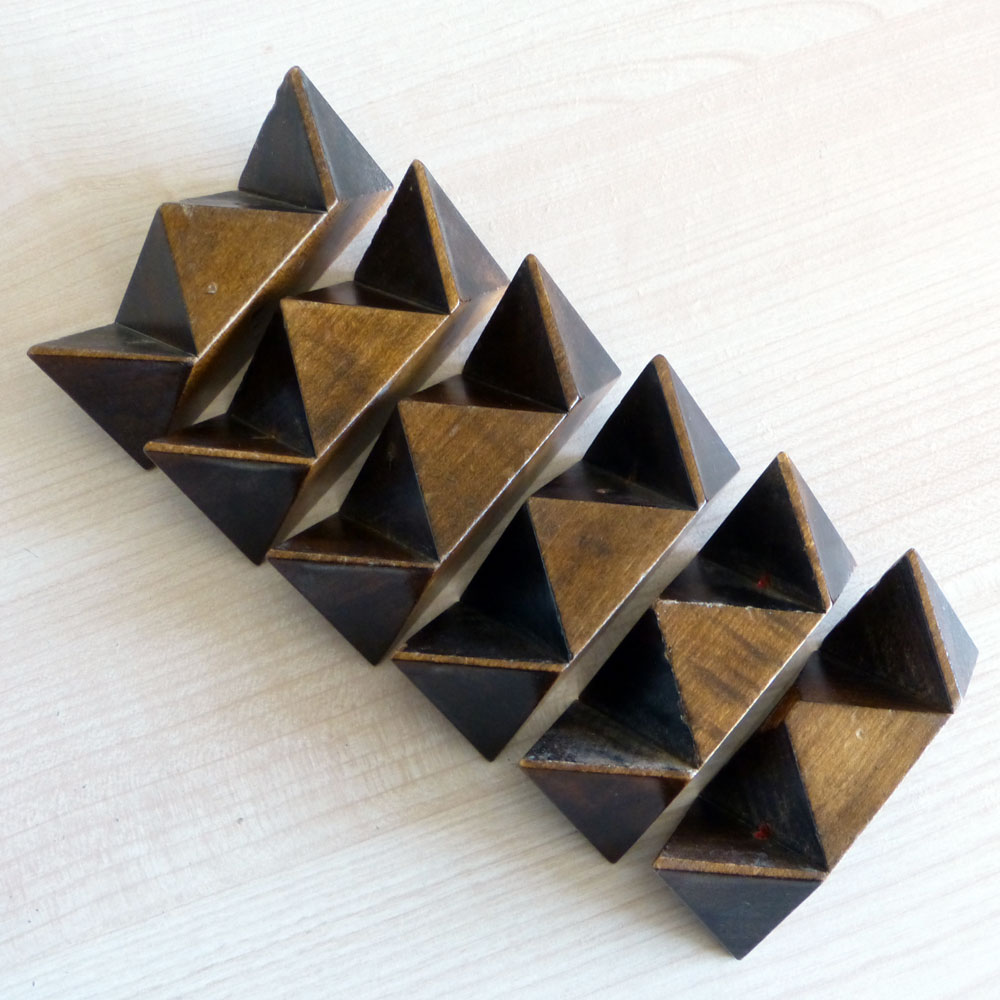}
\label{Fig:StarBurrApart}
} \quad 
\subfloat[]
{
\includegraphics[width=0.16\textwidth]{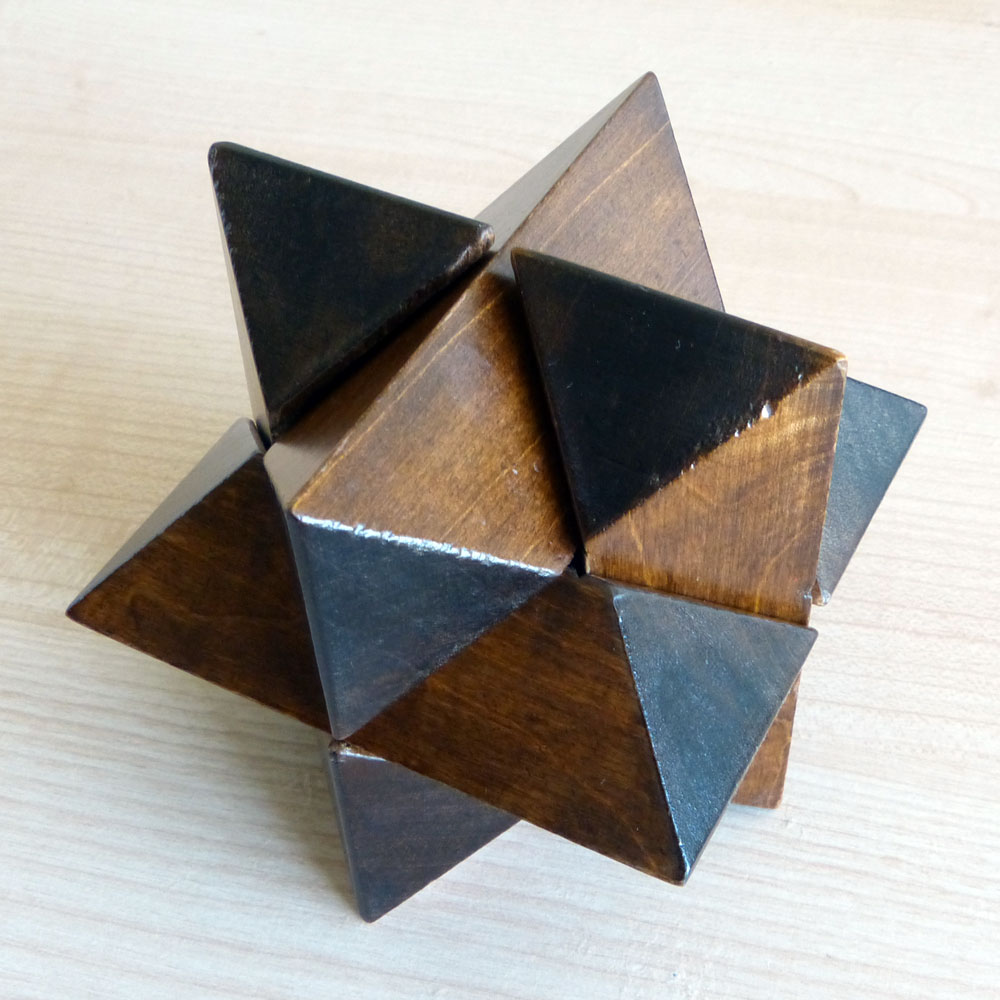}
\label{Fig:StarBurrTogether}
}
\caption{The star burr.} 
\label{Fig:StarBurr}
\end{wrapfigure}

A \emph{burr puzzle} is a collection of notched wooden
sticks~\cite[page~xi]{Coffin07} that fit together to form a highly
symmetric design, often based on one of the Platonic solids.  The
assembled puzzle may have zero, one, or more internal voids; it may
also have multiple solutions.  Ideally, no force is required.
Of course, a puzzle may violate these rules in various ways and still
be called a burr.

The best known, and certainly largest, family of burr puzzles are
collectively called the $6$--piece burrs~\cite{Cutler94}.
Another well-known burr, the star burr, is more closely related to our
work.
Unlike the $6$-piece burrs, the six sticks of the star burr are all
identical, as shown in~\reffig{StarBurrApart}.  The solution is unique
and, once solved, the star burr has no internal voids.  The solved
puzzle is a copy of the first stellation of the rhombic dodecahedron;
see~\reffig{StarBurrTogether}.


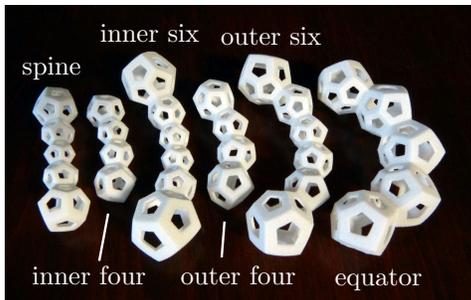
\begin{wrapfigure}[12]{l}{0.45\textwidth}
\vspace{5pt}
\centering 
\labellist
\footnotesize
\pinlabel \textcolor{white}{spine} at 100 504
\pinlabel \textcolor{white}{inner six} at 310 580
\pinlabel \textcolor{white}{outer six} at 570 573
\pinlabel \textcolor{white}{inner four} at 180 60
\pinlabel \textcolor{white}{outer four} at 500 60
\pinlabel \textcolor{white}{equator} at 800 50
\endlabellist
\includegraphics[width = 0.4\textwidth]{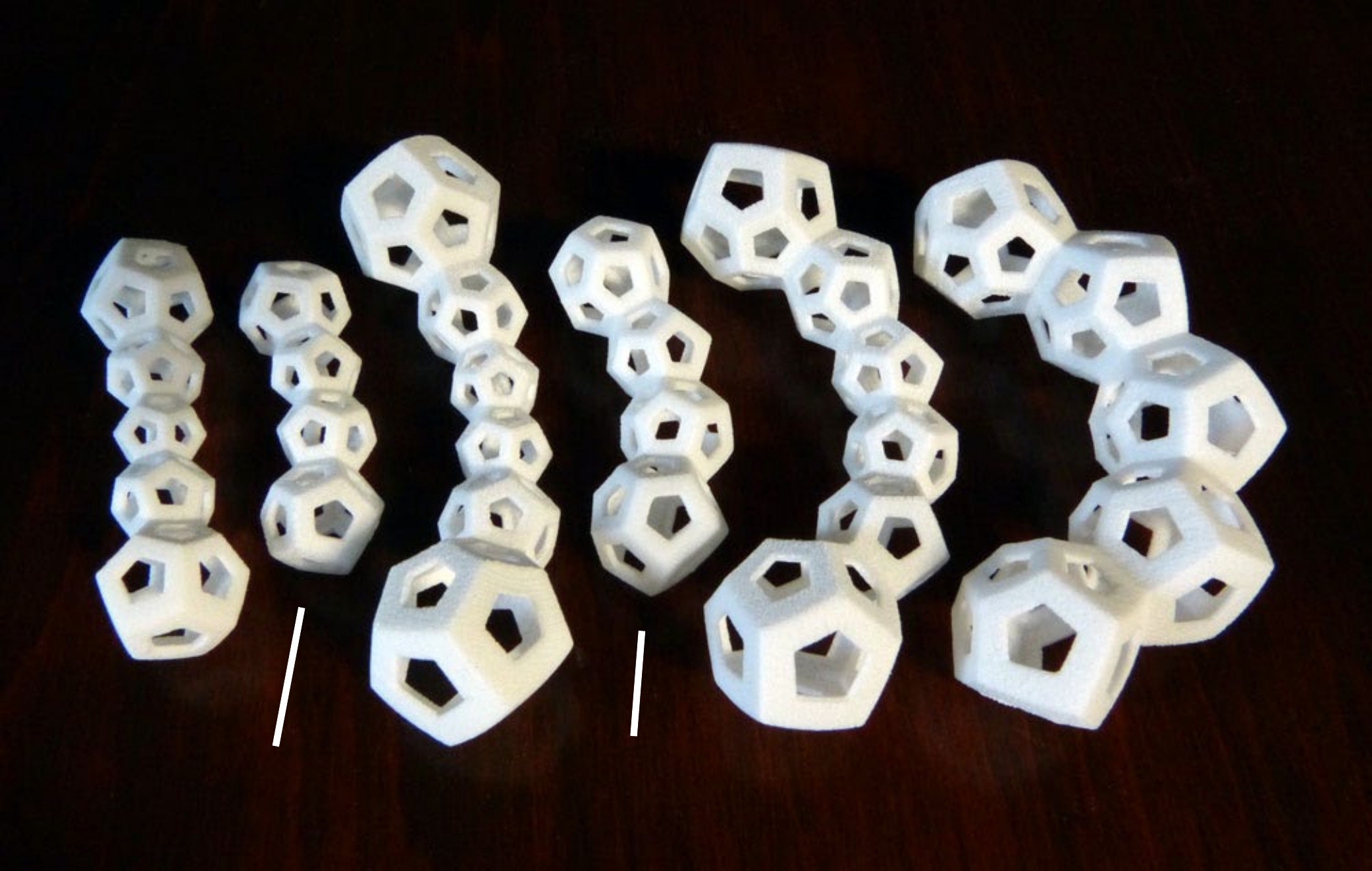}
\caption{The six rib types.} 
\label{Fig:Ribs}
\end{wrapfigure}


The goal of this paper is to describe \emph{Quintessence}: a new
family of burr puzzles based on the $120$--cell, a regular
four-dimensional polytope.  The puzzles are built from collections of
six kinds of sticks, shown in \reffig{Ribs}; we call these \emph{ribs}
as they are gently curving chains of distorted dodecahedra.

In \refsec{Poly} we review the basic concepts of regular polytopes in
low dimensions; in \refsec{Dodeca} we construct the dodecahedron and
derive several trigonometric facts.  In \refsec{Quatern} we briefly
review the three-sphere, the quaternions and stereographic projection.
As discussed in our previous paper~\cite{SchleimerSegerman12},
stereographic projection allows us to translate objects from the
three-sphere into our usual three-dimensional space.

Using the binary dodecahedral group, as it lies inside of the
quaternions, in \refsec{120} we construct the $120$--cell.  In
\refsec{Comb} we investigate the combinatorics of the $120$--cell,
focusing on how it decomposes into spheres and rings of dodecahedra.
In \refsec{Ribs} we lay out our choice of ribs, as influenced by the
cell-centred stereographic projection.  We use this to give a basic
combinatorial restriction on the possible burr puzzles in
Quintessence.  \refsec{Leonardo} briefly recalls Leonardo da Vinci's
technique for drawing polytopes; we use his method and stereographic
projection to produce our puzzle pieces.  One of the completed
puzzles, the Dc30 Ring, serves as our frontispiece (\reffig{Ring}).
We end with \refapp{Catalog}, a catalogue of some of the burr puzzles in
Quintessence.  The connection between the classic burrs and ours is
left as a final exercise for the intrigued reader.



\subsection*{Acknowledgements} 
We thank Robert Tang and Stuart Young for their insights into the
combinatorics of the $120$--cell.

\section{Polytopes}
\label{Sec:Poly}

We refer to~\cite{Ziegler95} for an in-depth discussion of polytopes.
Here we concentrate on the ideas needed to understand regular
polytopes.

\begin{wrapfigure}[10]{r}{0.25\textwidth}
\vspace{-25pt}
\centering 
\includegraphics[width=0.20\textwidth]{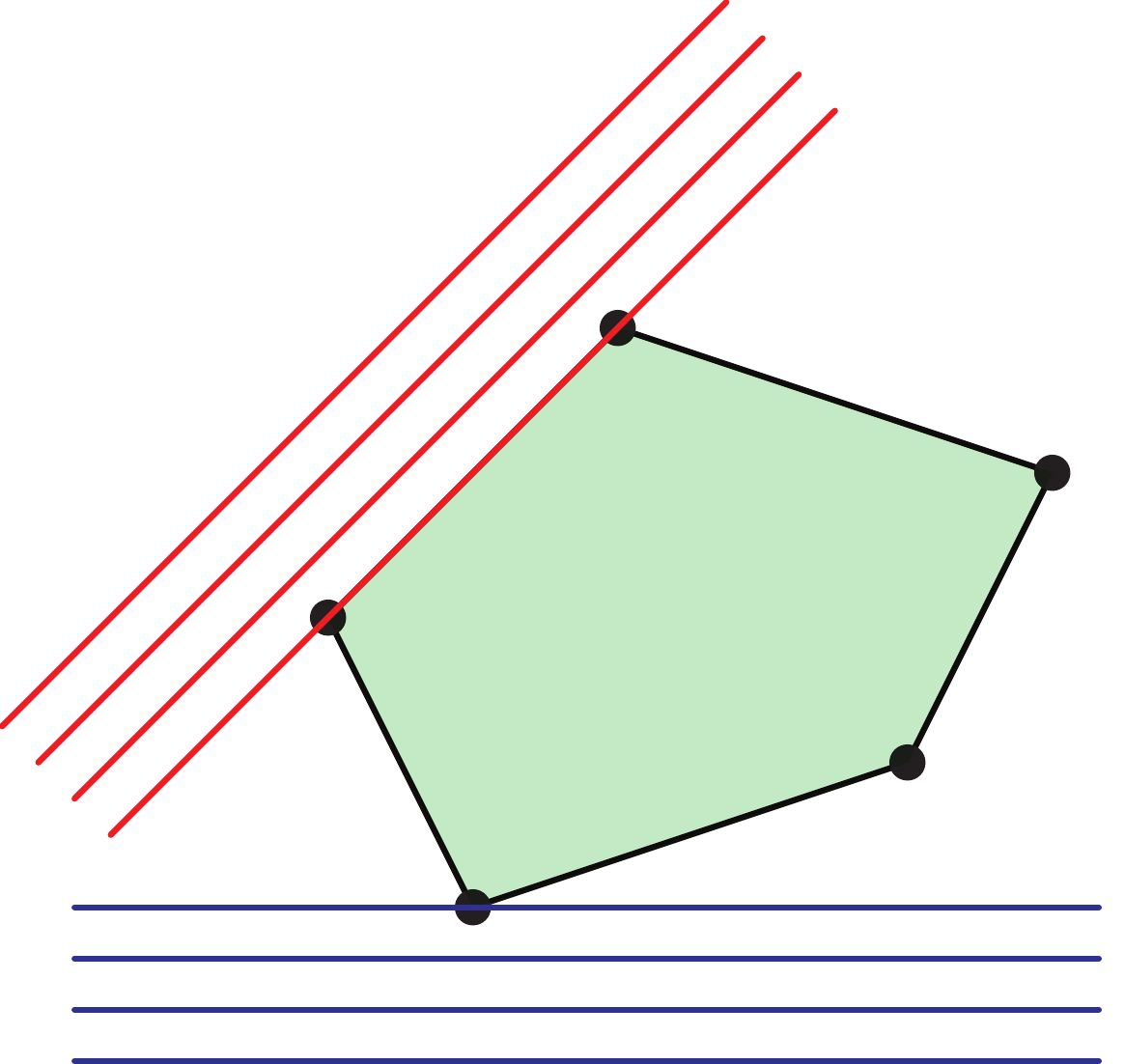}
\caption{Convex hull of five points in the plane.} 
\label{Fig:Convex}
\end{wrapfigure}

\subsection{Convexity}

We use $\RR^n$ to denote the usual $n$--dimensional space; we use
$S^{n-1}$ to denote the sphere of radius one in $\RR^n$.  A set $C
\subset \RR^n$ is \emph{convex} if for any points $x$ and $y$ in $C$
the line segment $[x,y]$ is also contained in $C$.  As a consequence
$C$ cannot have any internal voids.  Convexity also rules out dents on
the boundary of $C$.

For any subset $V \subset \RR^n$ the \emph{convex hull} of $V$,
denoted by $\hull(V)$, is the smallest convex set containing $V$.  For
example, the convex hull of two distinct points is a line segment.
The convex hull of three points, not all in a line, is a triangle.  In
general, if $V$ is a collection of $k+1$ points, not all in a
$k$--dimensional hyperplane, then $\hull(V)$ is called a
\emph{$k$--simplex}.  


\begin{wrapfigure}[20]{l}{0.25\textwidth}
\vspace{5pt}
\centering 
\includegraphics[width=0.20\textwidth]{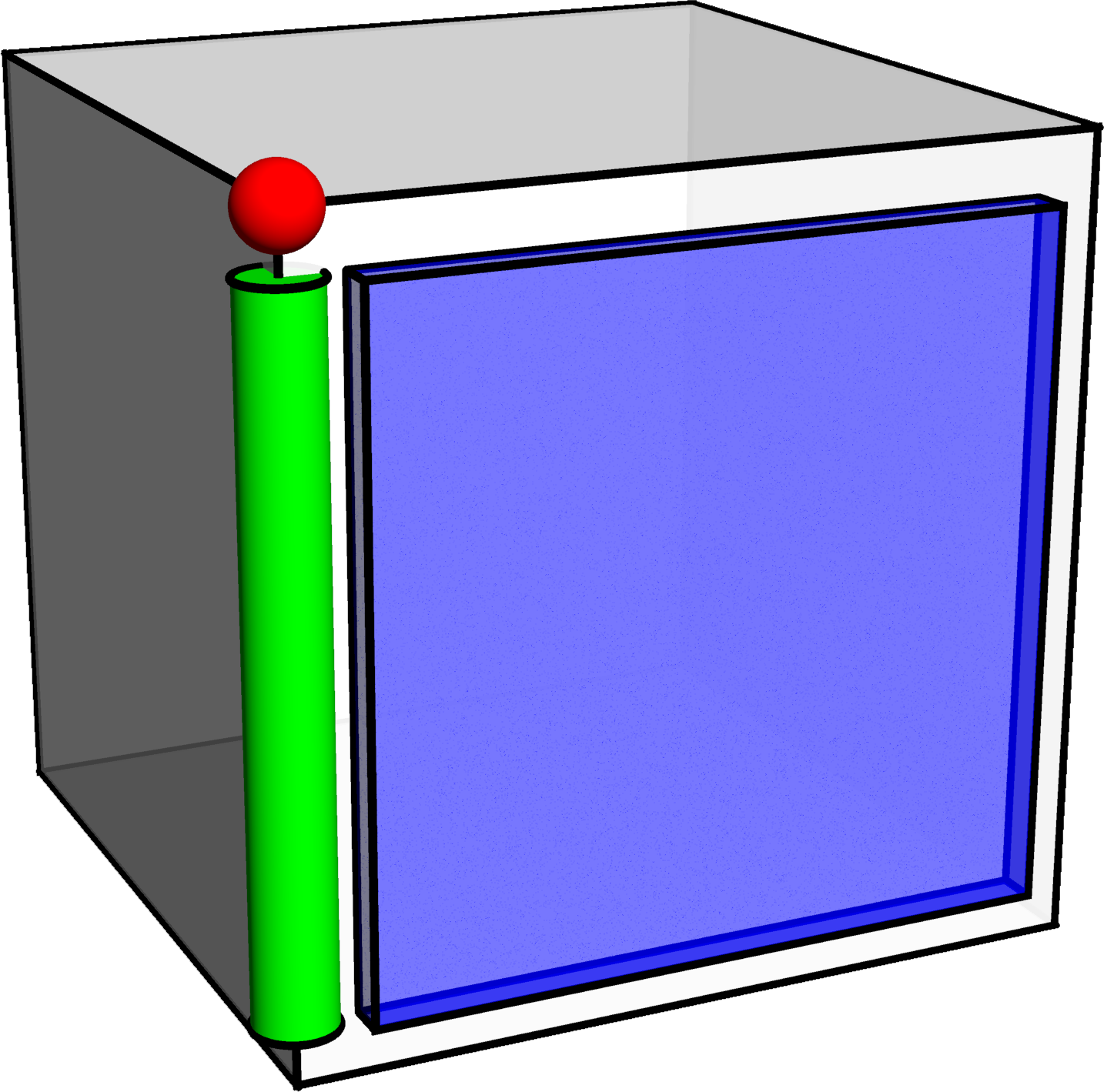} \\
\mbox{}\\
\includegraphics[width=0.20\textwidth]{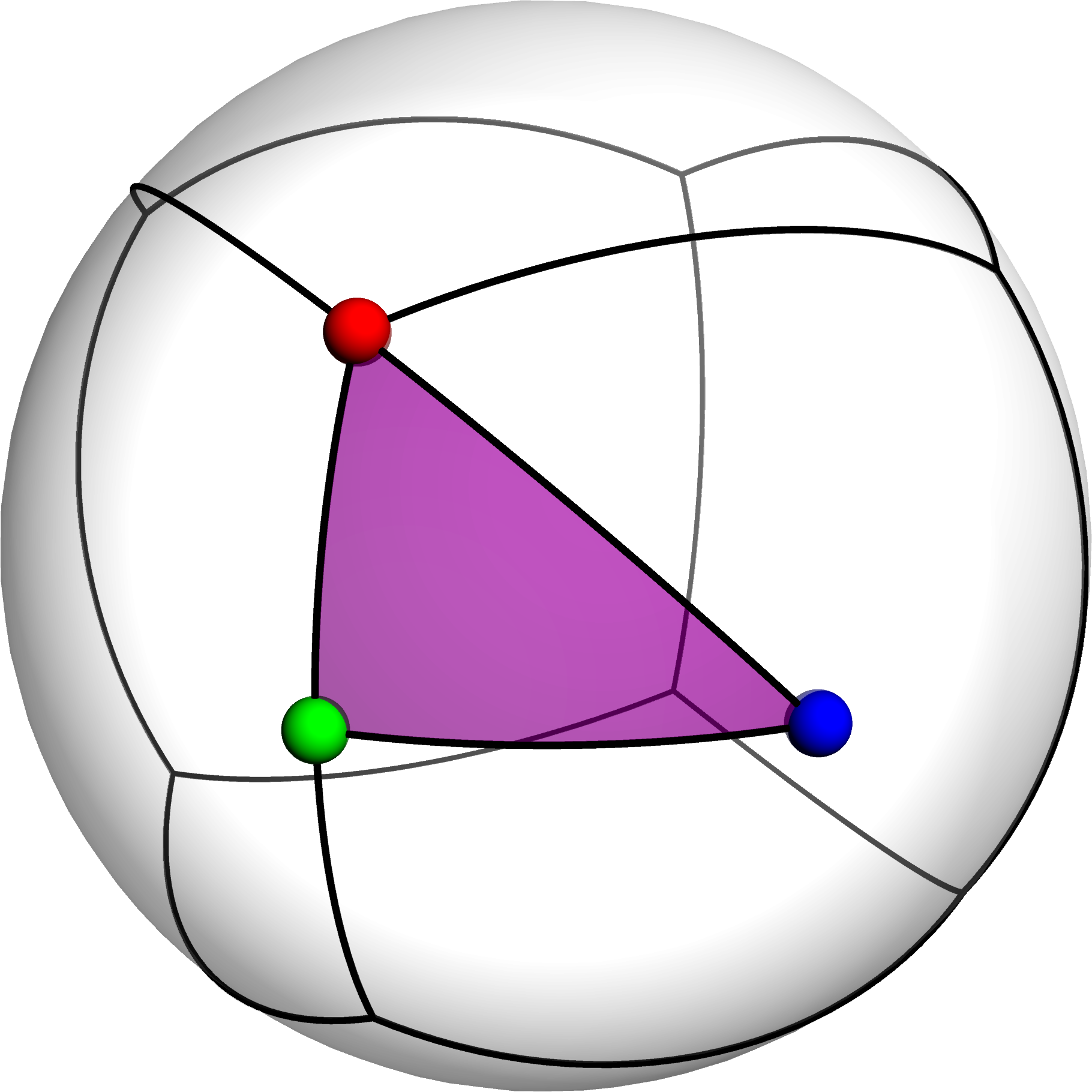}
\caption{Flag for cube and corresponding spherical flag triangle.} 
\label{Fig:Flag}
\end{wrapfigure}

Here we will always restrict $V \subset \RR^n$ to be finite; thus $P =
\hull(V)$ is a \emph{polytope}.
The dimension of $P$ is the dimension of the smallest affine subspace
$H \subset \RR^n$ containing $P$.  We call $H$ the \emph{affine span}
of $P$.  In the examples above the interval has dimension one, the
triangle two, and the tetrahedron three.  


Choose $K$, a hyperplane in $H$, that is disjoint from the polytope
$P$.  We move $K$, always staying parallel to itself, towards $P$
until they first touch.  See \reffig{Convex}.  The resulting
intersection $Q = P \cap K$ is again a polytope; we call $Q$ a
\emph{face} of $P$.

The \emph{vertices} of $P$ are exactly the zero-dimensional faces.  If
the dimension of $Q$ is exactly one less than that of $P$ then we call
$Q$ a \emph{facet} of $P$.  For example, any tetrahedron has four
facets, all triangles; this gives the tetrahedron its name.  We define
$\bdy P$, the \emph{boundary} of $P$, to be the union of the facets of
$P$.

\subsection{Regular polytopes}

Suppose that $P$ is a $k$--dimensional polytope, with affine span $H$.
A collection of faces $Q_0 \subset Q_1 \subset \ldots \subset Q_{k-1}
\subset Q_k = P$ is called a \emph{flag} of $P$ if $Q_\ell$ has
dimension $\ell$.  See \reffig{Flag} (top) for a picture of one of the
$48$ flags of the cube.

Let $\Sym(P)$ be the group of rigid motions (and reflections) of $H$
that preserve $P$ setwise.  We call elements of $\Sym(P)$ the
\emph{symmetries} of $P$.

\begin{definition}
\label{Def:Regular}
A polytope $P$ is \emph{regular} if for any pair of its flags, $F$ and
$G$, there is a symmetry $\phi \in \Sym(P)$ with $\phi(F) = G$.
\end{definition}

It follows that all facets of a regular polytope are congruent and
are themselves regular.
As an example, consider the octahedron $O \subset \RR^3$: the convex
hull of the six points
\[
(\pm 1, 0, 0), \; (0, \pm 1, 0), \; (0, 0, \pm 1).
\]
The octahedron, like the cube,  has 
$48$ flags.  Any one can be sent to any other by reflections in the
coordinate planes and rotations about the coordinate axes.  Note that
the facets of $O$ are all congruent equilateral triangles, so are
themselves regular two-polytopes.  


So, suppose $P$ is regular.  Define $p = \centerpt(P)$ to be the
average of the vertices of $P$.
Since $\Sym(P)$ permutes the vertices of $P$, it fixes $p$.  Since
$\Sym(P)$ sends any flag to any other, the same is true of the vertices. So the vertices are all the same
distance from $p$.  Thus $p$ is a \emph{circumcentre}: $P$ is
circumscribed by the sphere $S_P$ centred at $p$ and running through
the vertices of $P$.  If we project $\bdy P$ from $p$ outwards to
$S_P$ we obtain a 
spherical tiling $\calT_P$.


Conversely, when we are constructing an $n$--dimensional regular
polytope $P$ our first move is to build a spherical tiling $\calT_P$
on $S^{n-1}$.  The tiling $\calT_P$ is often more tractable, and is
certainly easier to visualise. 

\begin{definition}
\label{Def:FlagPoly}
Suppose that $P$ is regular and $F = \{Q_i\}$ is a flag in $P$.  Then
the \emph{flag polytope} $Q_F$ is the convex hull of the centres of
the $Q_i$.  The \emph{spherical flag polytope} is the radial
projection of $Q_F - p$ to $S_P$.  See \reffig{Flag} (bottom).
\end{definition}

If $P$ is regular, then all of its spherical flag polytopes are
congruent.

\begin{definition}
\label{Def:Dual}
Suppose $P$ is a regular polytope.  We form the \emph{dual} polytope
$P'$ by taking the convex hull of the centres of the facets of $P$ and
then rescaling so all vertices of $P'$ lie on $S_P$.
\end{definition}

For example, the cube and octahedron are dual; this explains why they
have the same number of flags.


\subsection{Constructions}

There are four infinite families of regular polytopes; each family is
associated with a topological operation.
We begin in dimension two, with the regular polygons.  Let $\rho_n
\from \CC \to \CC$ be the map $\rho_n(\omega) = \omega^n$.  Restricted
to $S^1$ this becomes an $n$--fold covering map of the circle.

\begin{definition}
\label{Def:Polygon}
The \emph{regular $n$--gon} $P_n$ is the convex hull of
$\rho_n^{-1}(1)$: that is, of the $n^\thsup$ roots of unity.
\end{definition}

\begin{figure}[htbp]
\centering 
\includegraphics[width = 0.9\textwidth]{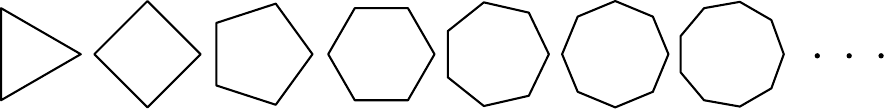}
\caption{The regular polygons.} 
\label{Fig:Polygons}
\end{figure}

The exterior angle at the vertex of $P_n$ is $\pi(1 - \frac{2}{n})$;
also $P_n$ is self-dual.  Already in this first example we see an
important principle: a regular polytope $P$ should be understood via
its circumscribing sphere, here the unit circle.

We now turn to the three families that exist in all dimensions:
simplices, cubes, and cross-polytopes.  Each family is defined in
terms of convex hulls and also given by its topological operation.  We
take $e^k_i = (0, \ldots, 0, 1, 0, \ldots, 0) \in \RR^k$ to be the
point with a single $1$ in the $i^\thsup$ coordinate and all other
coordinates $0$.


\begin{definition}
\label{Def:Simplex}
The \emph{$k$--simplex} is the convex hull of the $k+1$ points $\{ e_i
\}$ in $\RR^{k+1}$.  Thus it is a (right) cone
with base the $(k-1)$--simplex and with height
$\sqrt{1+k^{-1}}$.
\end{definition}


\begin{definition}
\label{Def:Cube}
The \emph{$k$--cube} is the convex hull of the $2^k$ points $\{ \pm
e_1, \pm e_2, \ldots \pm e_k \}$ in $\RR^k$.  Thus it is a product between
the $(k-1)$--cube and the unit interval.
\end{definition}

\begin{definition}
\label{Def:CrossPoly}
The \emph{$k$--cross-polytope} is the convex hull of the $2k$ points
$\{ \pm e_i \}$, taken in $\RR^k$.  Thus it is a suspension with base
the $(k-1)$--cross-polytope and of height one.  Here a
\emph{suspension} is a double (right) cone to points lying
symmetrically above and below the centre of the base.
\end{definition}

\begin{figure}[htbp]
\centering 
\includegraphics[width = 0.7\textwidth]{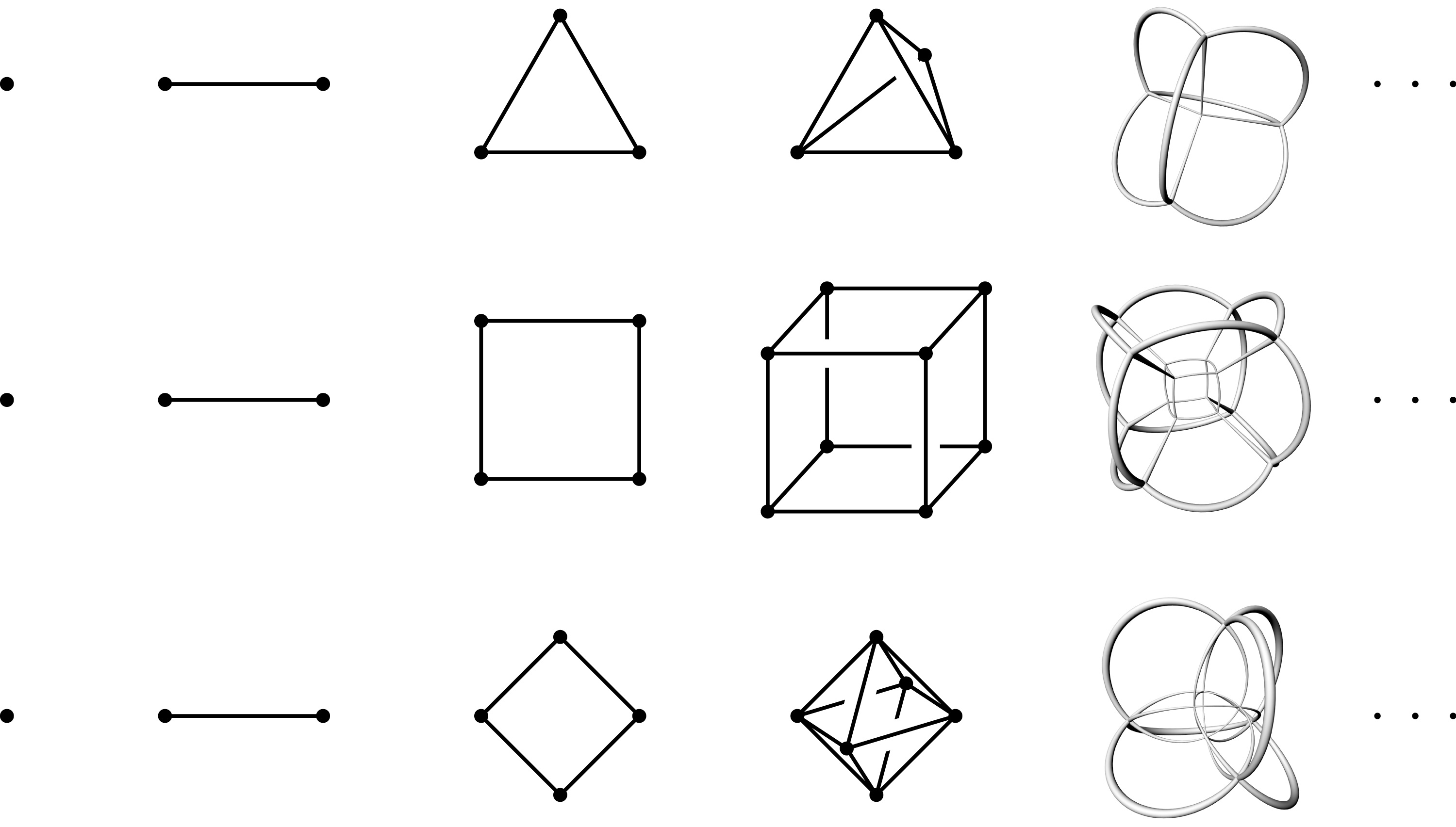}
\caption{The first five simplices, cubes, and cross-polytopes.} 
\label{Fig:ThreeFamilies}
\end{figure}

The first several examples of each are shown in
\reffig{ThreeFamilies}.  The one-dimensional versions are all
intervals.  In dimension two they are the triangle, square, and
diamond, respectively.  In dimension three the simplex is the
tetrahedron and the cross-polytope is the octahedron.  The fifth
column shows the \emph{stereographic projections} of the spherical
tilings for the four-dimensional members of each family.  These cannot
be drawn in three-dimensional space so we instead radially project
their boundaries to $S^3$ and then stereographically project to
$\RR^3$.  This technique was the subject of our paper
\cite{SchleimerSegerman12} and we use it again here.  For the
convenience of the reader, we repeat the definition of stereographic
projection in \refsec{StereoProj}.

We now collect several useful statements which we will not prove here.
Instead see~\cite[page~143]{Toth64}.

\begin{lemma}
\label{Lem:Poly}
The simplex, cube, and cross-polytope are regular.  The cube and the
cross-polytope are dual; the simplex is self-dual.  In dimensions
three and higher, these three polytopes are distinct. \qed
\end{lemma}

\begin{theorem}
\label{Thm:Poly}
There are exactly five regular polytopes not in one of the four
families.  These are, in dimension three, the dodecahedron and
icosahedron (dual) and, in dimension four, the $24$--cell (self-dual),
and the $120$--cell and $600$--cell (dual). \qed
\end{theorem}

We construct the dodecahedron and the $120$--cell in
Sections~\ref{Sec:Dodeca} and~\ref{Sec:120}.

\section{Dodecahedron}
\label{Sec:Dodeca}

\subsection{Construction} 
\label{Sec:DodecaCon}

The dodecahedron exists for reasons more subtle than those, given
above, for the four families.  As such it has many constructions; the
earliest seems to be Proposition 17 in Book 13 of Euclid's
Elements~\cite{EuclidHeath56}.  See~\cite{Waterhouse72} for one
historical account of the five Platonic solids.





We give an indirect construction of the dodecahedron $D$ that has two
advantages.  The argument finds the symmetry group $\Sym(D)$ along the
way.  It also generalises to all other regular tessellations of the
sphere, the Euclidean plane, and hyperbolic plane.  

By continuity, for any angle $\theta \in \left( 3\pi/5, 7\pi/5 \right)$
there is a regular spherical pentagon $P \subset S^2$ with all angles
equal to $\theta$.  See \reffig{PentagonEDT} (left).
Thus we may take $\theta$ equal to $2\pi/3$.

\begin{figure}[htbp]
\centering 
\includegraphics[width=0.30\textwidth]{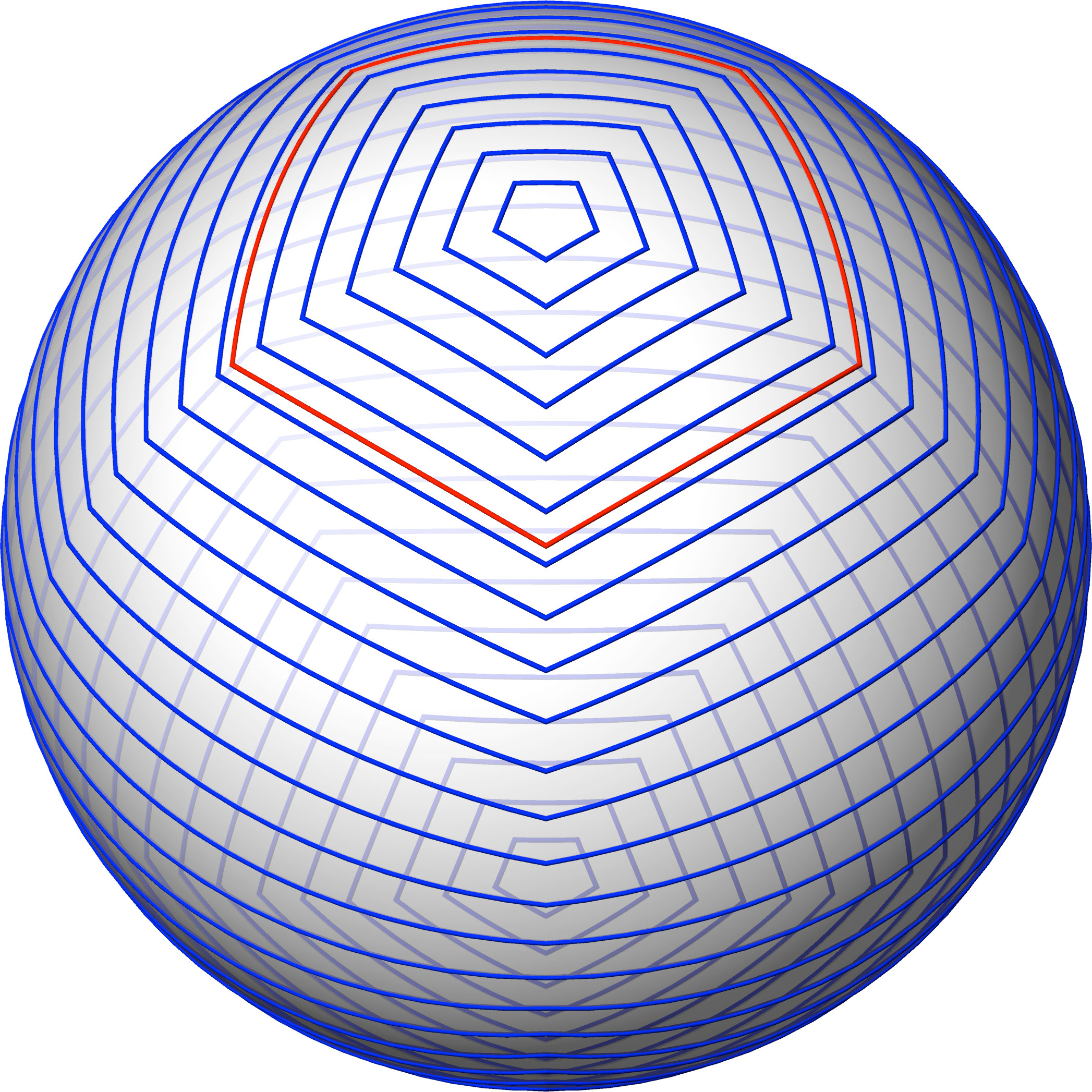}
\quad
\includegraphics[width=0.30\textwidth]{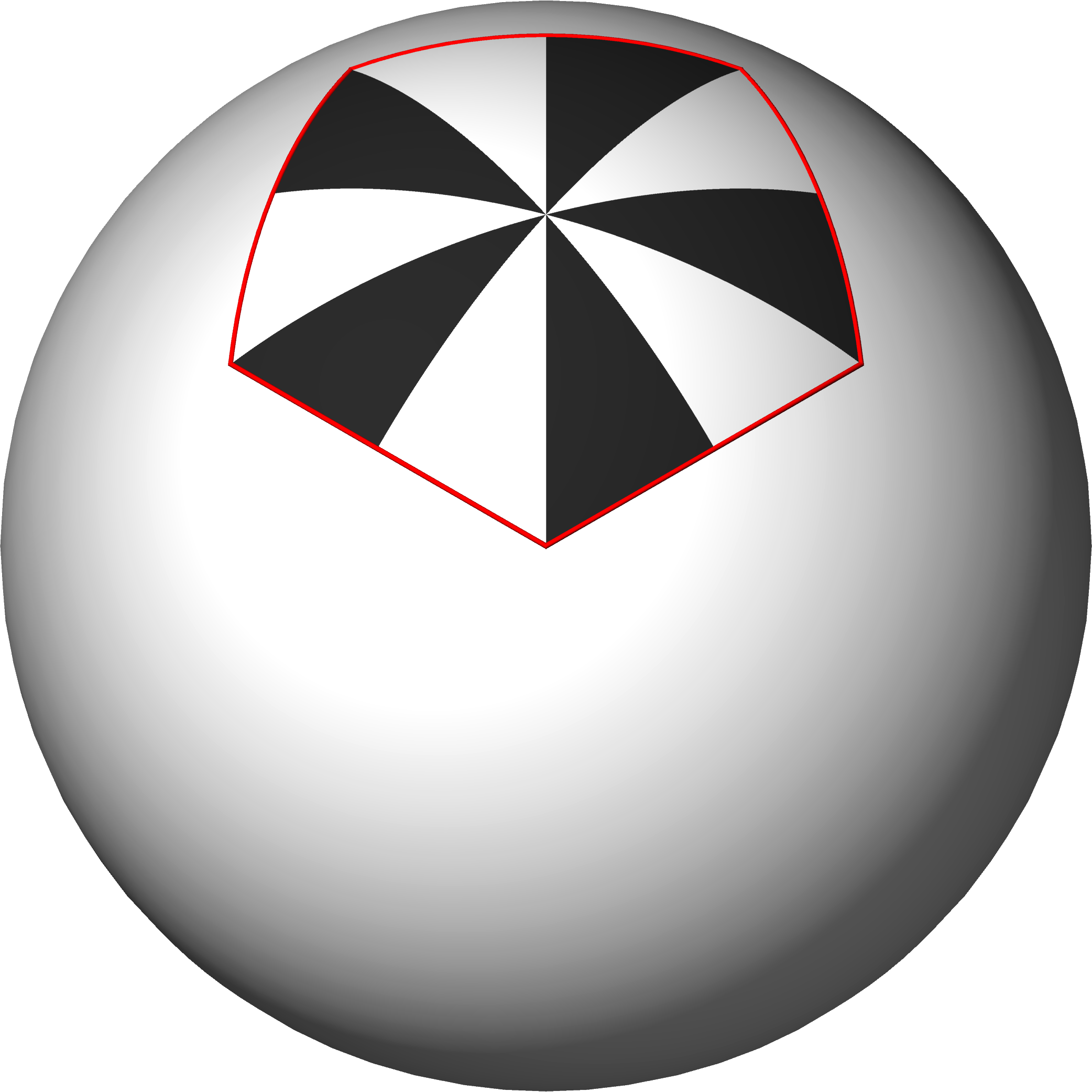}
\quad
\includegraphics[width=0.30\textwidth]{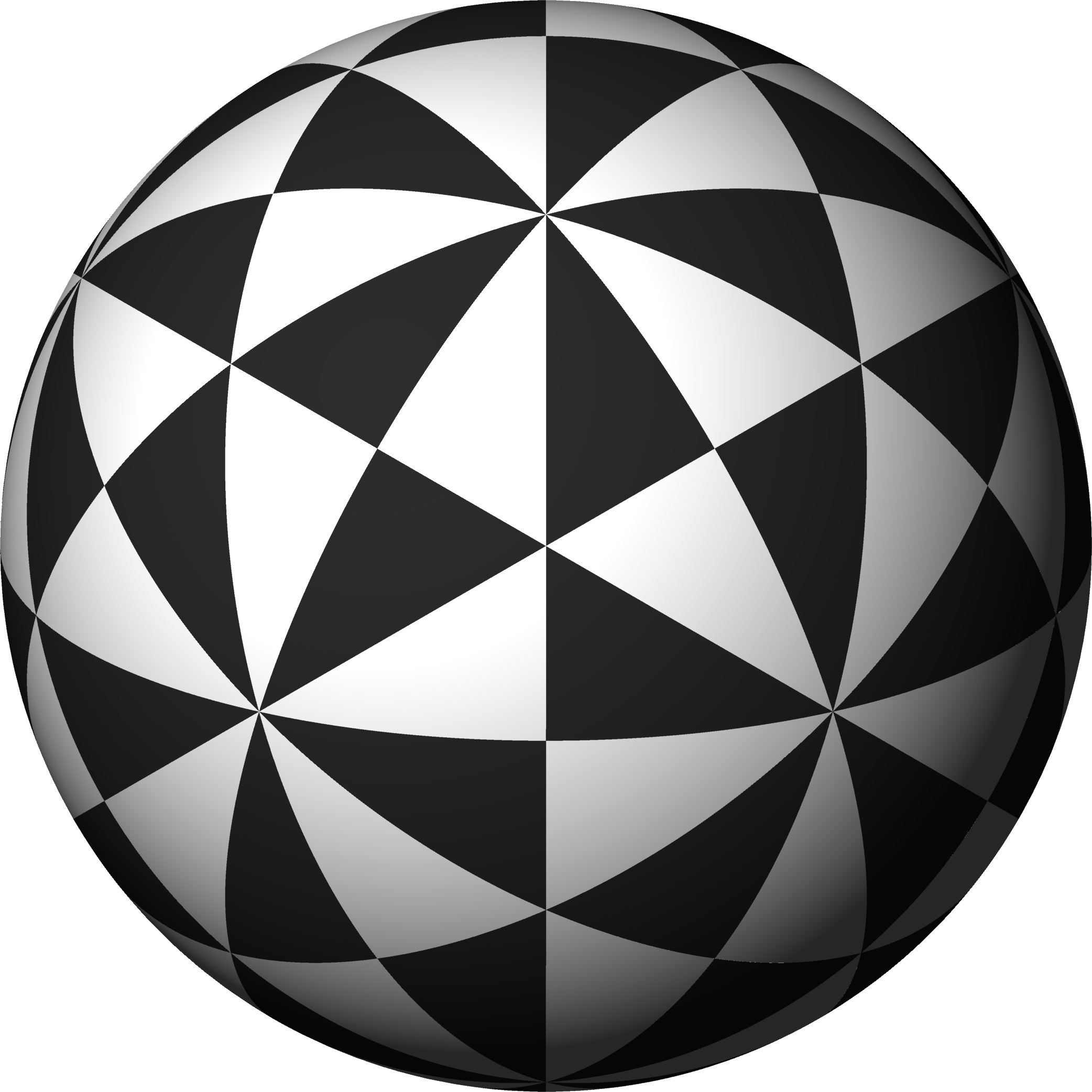}
\caption{Left: The continuity argument.  Centre: Dividing the pentagon
  into five right-handed spherical flags (in black) and five
  left-handed ones.  Right: The tiling $\calT$.}
\label{Fig:PentagonEDT} 
\end{figure}

Adding a vertex at the centre and at the midpoints of the edges, we
divide $P$ into ten spherical flag triangles.  These alternate between being
\emph{right-} and \emph{left-handed}; all have internal angles
$(\pi/2, \pi/3, \pi/5)$.  See \reffig{PentagonEDT} (centre).  These
three angles appear at the edge, vertex, and centre of $P$.  Let $T_R$
and $T_L$ be copies of the right and left handed spherical flag triangles, and
note that there are rotations of $S^2$ matching the edges of $T_R$ and
$T_L$ in pairs.

The celebrated Poincar\'e polygon
theorem~\cite[Theorem~4.14]{EpsteinPetronio94} now implies that copies
of $T_R$ and $T_L$ give a tiling $\calT$ of $S^2$, shown in
\reffig{PentagonEDT} (right).
Poincar\'e's theorem also implies that $\Sym(\calT)$ is transitive on
the triangles of $\calT$ and that any local symmetry extends to give
an element of $\Sym(\calT)$.


We now appeal to Girard's formula for the area of a triangle in
$S^2$~\cite[Equation~2.11]{Coxeter74}.  

\begin{lemma}
\label{Lem:Area}
A spherical triangle with interior angles $A$, $B$, $C$ has area $A +
B + C - \pi$. \qed
\end{lemma}


\begin{figure}[htbp]
\centering 
\subfloat[]
{
\includegraphics[width=0.24\textwidth]{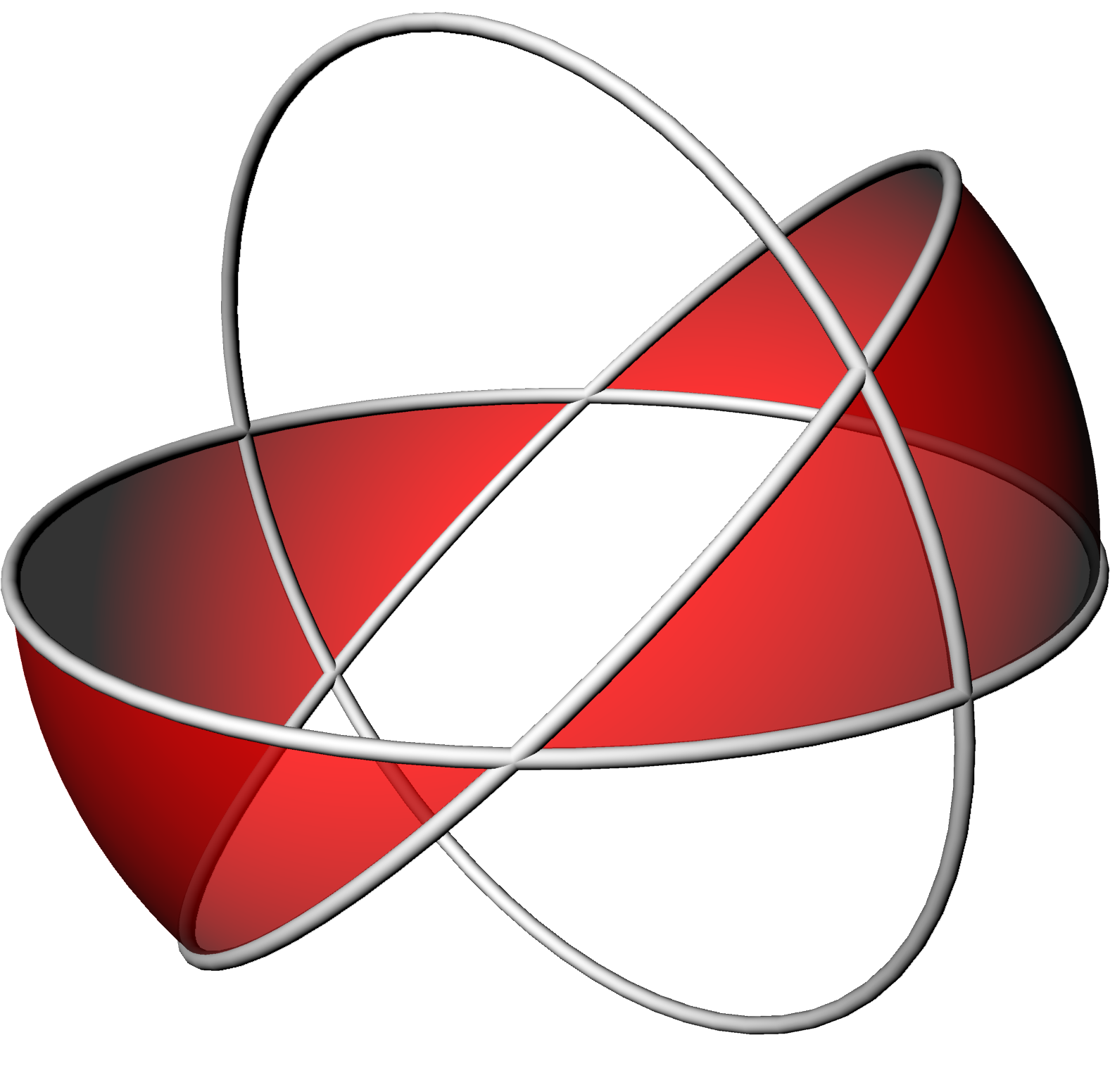}
\label{Fig:SphericalTriAreaBigon1}
}
\subfloat[]
{
\includegraphics[width=0.24\textwidth]{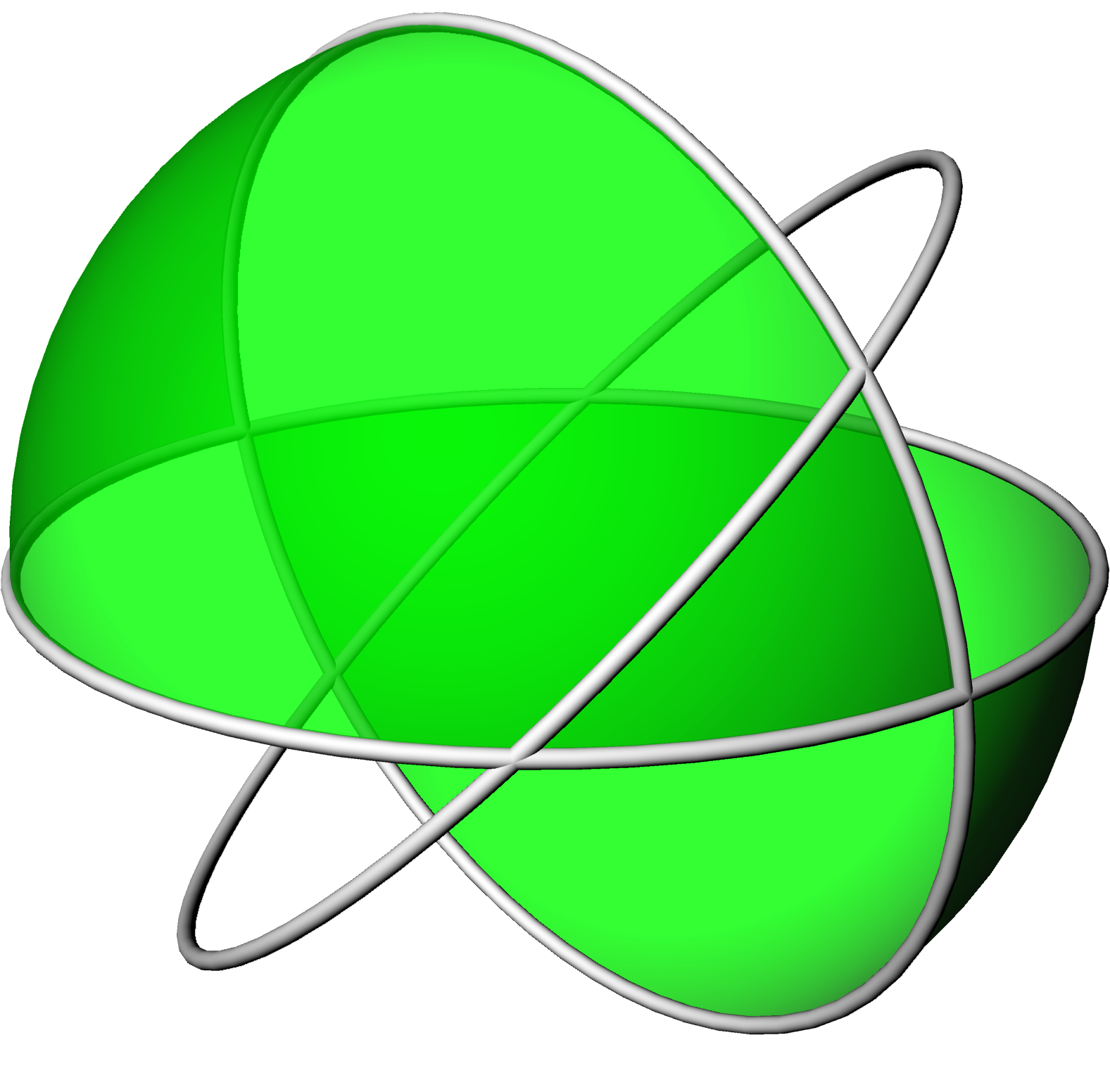}
\label{Fig:SphericalTriAreaBigon2}
}
\subfloat[]
{
\includegraphics[width=0.24\textwidth]{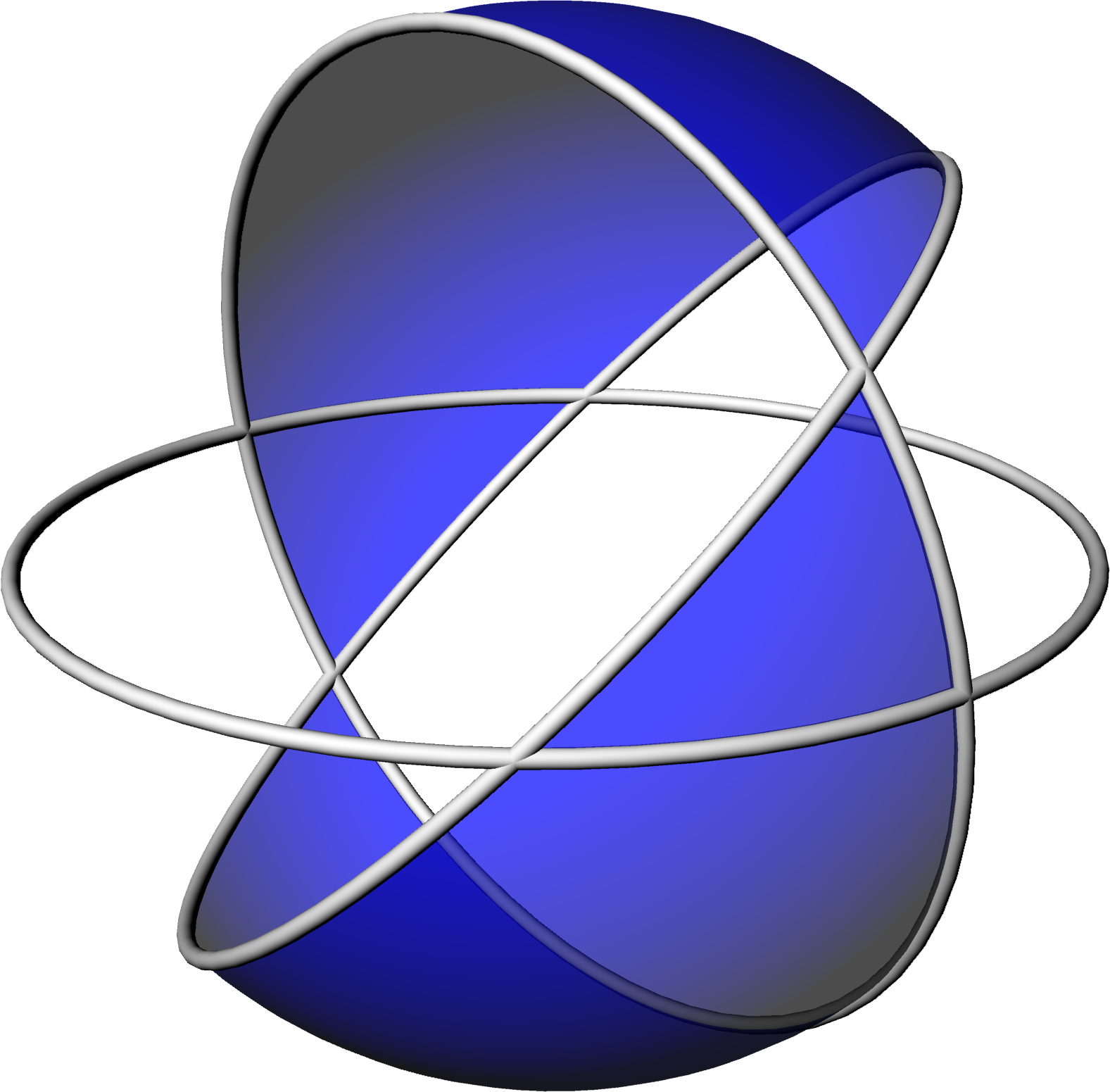}
\label{Fig:SphericalTriAreaBigon3}
}
\subfloat[]
{
\includegraphics[width=0.24\textwidth]{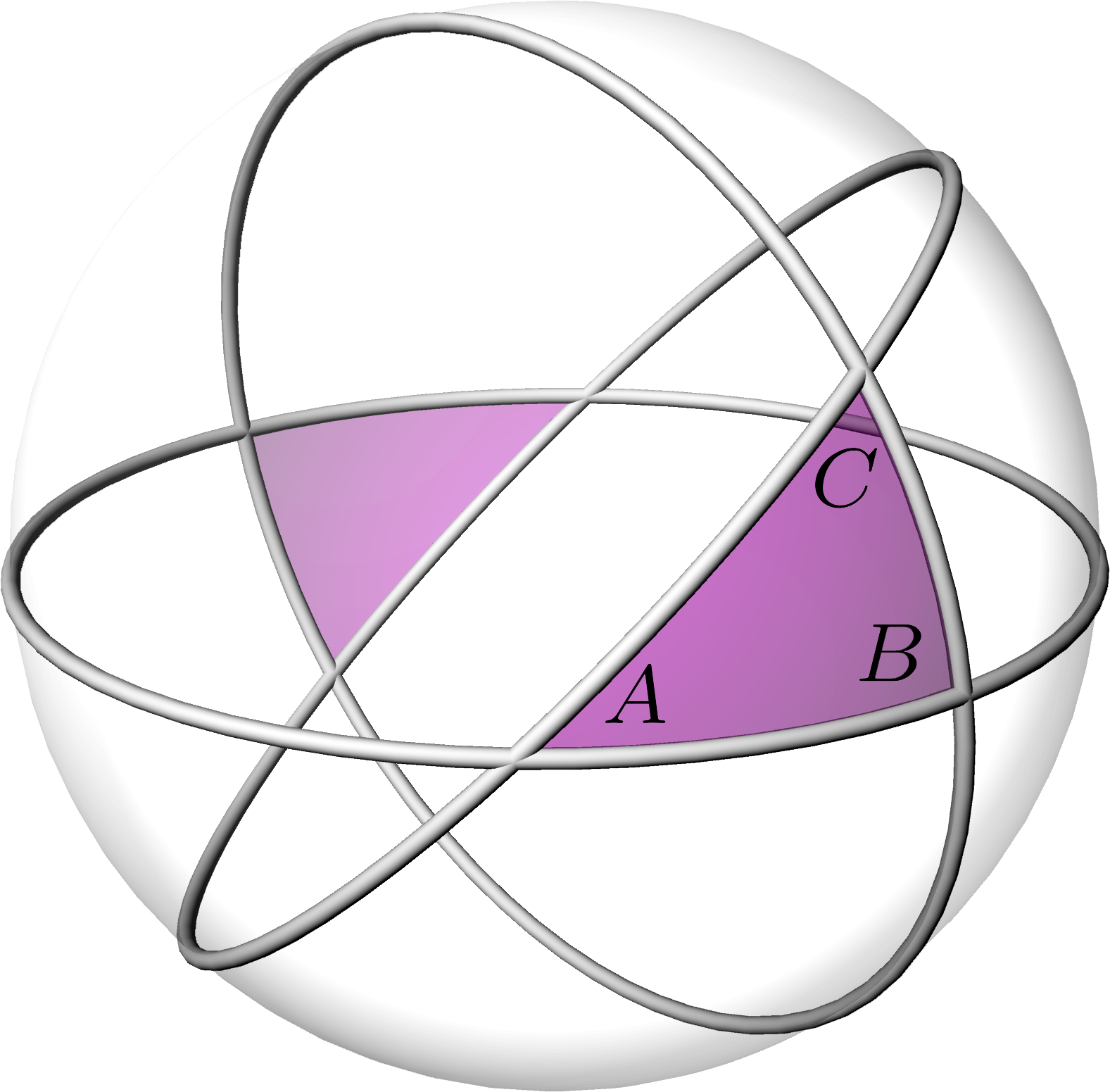}
\label{Fig:SphericalTriAreaTri}
} 
\caption{Proof of \reflem{Area}.}
\label{Fig:SphericalTriArea}
\end{figure}

A ``proof by picture'' of \reflem{Area} is given in
\reffig{SphericalTriArea}.  Thus the area of $T_R$ is
\[
\pi \cdot (1/2 + 1/3 + 1/5) - \pi 
= \pi/30.
\]
Since the area of $S^2$ is $4\pi$ deduce that the tiling $\calT$
contains $120$  triangles.  

\begin{definition}
\label{Def:Dodecahedron}
We partition $\calT$ into copies of $P$ to obtain the tiling
$\calT_D$; this has $12$ pentagonal faces, $12 \cdot 5/2 = 30$ edges,
and $12 \cdot 5/3 = 20$ vertices.  We take the convex hull (in
$\RR^3$) of the vertices of $\calT_D$ (in $S^2$) to obtain $D$, the
dodecahedron.
\end{definition}

We use $\SO(n)$ to denote the group of $n$-by-$n$ orthogonal matrices
with determinant one.  This is also the group of rigid motions of
$\RR^n$ fixing the origin. 
When $n = 3$ we have Euler's rotation theorem~\cite{Euler76}, as
follows.  Any $A \in SO(3)$ is a rotation about some \emph{axis}: a
line through the origin fixed pointwise by $A$.  When $A$ is not the
identity, this axis is unique.  See~\cite{PalaisEtAl09} for several
proofs and a historical discussion.

We end this section by examining the symmetries of $\calT$.

\begin{lemma}
\label{Lem:Tiling}
The group $\Sym(\calT)$ has order $120$; the orientation-preserving
subgroup $\calD = \Sym^+(\calT)$ has order $60$.  Also, the tiling
$\calT$ is invariant under the antipodal map.
\end{lemma}

\begin{proof}
Note $\calD$ is a subgroup of $\SO(3)$.  Fix a non-trivial element $F
\in \calD$.  So $F$ is a symmetry of $\calT$.  By Euler's rotation
theorem $F$ fixes, and rotates about, antipodal points $p, q \in S^2$.
If $p$ lies in the interior of a triangle $T$, then $F$ non-trivially
permutes the vertices of $T$, contradicting the fact that all of their
internal angles are distinct.  Suppose instead that $p$ lies in the
interior of an edge of $T$.  Then $F$ swaps the endpoints of the edge,
another contradiction.  The last possibility is that $p$ is a vertex
of $T$, say of degree $2d$.  In this case $F$ is one of the $d-1$
possible rotations.

We deduce that the orientation-preserving symmetries of $\calT$ are in
one-to-one correspondence with (say) the right-handed flag triangles.
This counts the elements of $\calD = \Sym^+(\calT)$ and thus of
$\Sym(\calT)$.

It remains to prove that $\calT$ is invariant under the antipodal map.
Suppose that $p$ is a vertex of degree $2d$ of $\calT$.  There is a
local symmetry $f$ of $\calT$ that rotates about $p$, with order $d$.
Thus $f$ extends to a global symmetry $F \in \SO(3)$.  Since $F$ is a
non-trivial rotation, Euler again gives us a pair of antipodal fixed
points for $F$ on the unit sphere $S^2$.  One of these is $p$; call
the antipode $q$.  Restricting $F$ to a small neighbourhood of $q$
yields a rotation of order $d$ (of the opposite handedness).  It follows
that $q$ is another vertex of $\calT$, also of degree $2d$.
\end{proof}

\pagebreak

\begin{wrapfigure}[15]{l}{0.40\textwidth}
\vspace{10pt}
\centering 
\includegraphics[width=0.35\textwidth]{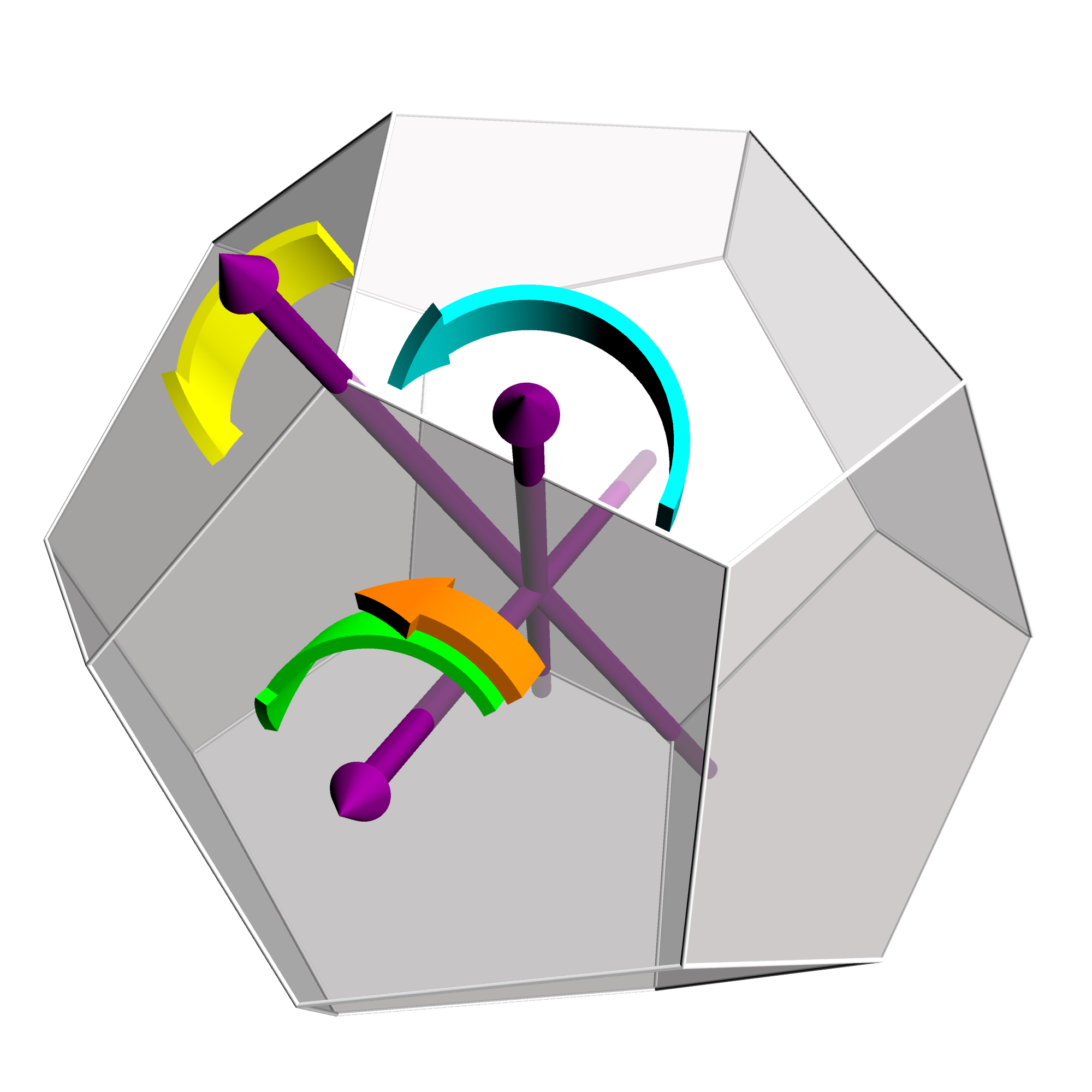}
\caption{Rotational symmetries of the dodecahedron.}
\label{Fig:DodecWithRotationArrows}
\end{wrapfigure}

\mbox{}
\vspace{-22pt}
\vspace{-\lineskip}

\begin{corollary}
\label{Cor:ElementsOfD}
The group $\calD$ contains
\begin{itemize}
\item
the identity,
\item
$12$ face rotations through angle $2\pi/5$,
\item
$20$ vertex rotations through angle $2\pi/3$,
\item
$12$ face rotations through angle $4\pi/5$, and
\item
$15$ edge rotations through angle $\pi$.
\end{itemize}
\end{corollary}

\begin{proof}
For any vertex $p$ of $\calT$ of degree $2d$ we obtain a cyclic
subgroup $\ZZ/d\ZZ$ in $\calD$.  By the second part of \reflem{Tiling}
the vertex $p$ and its antipode $q$ give rise to the same subgroup.
Thus we may count elements of $\calD$ by always restricting to those
rotations through an angle of $\pi$ or less.  Counting the symmetries
obtained this way gives $60$; by the first part of \reflem{Tiling}
there are no others.
\end{proof}

\subsection{Trigonometry}
\label{Sec:DodecaGeom}

For the construction of the $120$--cell, in \refsec{120}, we require
some trigonometric information about $\calT_D$.  Recall that $P$ is a
regular spherical pentagon with all angles equal to $2\pi/3$.

\begin{wrapfigure}[16]{r}{0.37\textwidth}
\vspace{15pt}
\centering 
\labellist
\small\hair 2pt
\pinlabel $f$ at 105 105
\pinlabel $v$ at 5 59
\pinlabel $\frac{\pi}{2}$ at 38 104
\pinlabel $\frac{\pi}{3}$ at 20 82
\pinlabel $\frac{\pi}{5}$ at 72 103
\pinlabel $a$ at 58 84
\pinlabel $a$ at 93 32
\endlabellist
\includegraphics[width=0.32\textwidth]{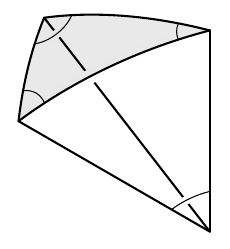}
\caption{Spherical flag triangle, coned to the origin.}
\label{Fig:SphericalAngles}
\end{wrapfigure}
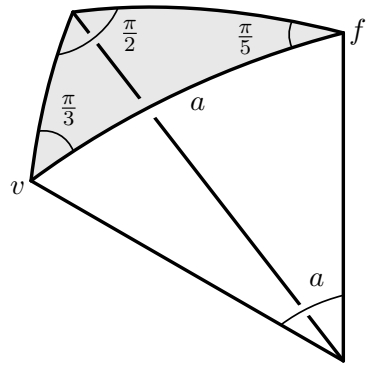

\vspace{-12pt}
\mbox{}

\begin{lemma}
\label{Lem:Edge}
The spherical distance between the face centre $f$ and the vertex $v$
of $P$ is
\[
\displaywidth=\parshapelength\numexpr\prevgraf+2\relax
\arccos \left(\frac{1}{\sqrt{3}} \cot \pi/5 \right).
\] 
\end{lemma}

\begin{proof}
Any spherical triangle with angles $A, B, C$ and opposite edge lengths
$a, b, c$ satisfies the dual spherical law of
cosines~\cite[pages~74--76]{Thurston97}:
\[
\displaywidth=\parshapelength\numexpr\prevgraf+2\relax
\cos A = - \cos B \cos C + \sin B \sin C \cos a.
\]
Recall the pentagon $P$ is a union of 10 triangles; any one of
these is a spherical triangle $T$ with angles $A = \pi/2$, $B =
\pi/3$, and $C = \pi/5$.  Using the law of cosines we find
$\cos a = \frac{1}{\sqrt{3}} \cot \pi/5$, as desired. 
\end{proof}

\begin{corollary}
\label{Cor:Edge}
The square of the Euclidean distance between the face centre $f$ and
the vertex $v$ of $P$ is $2 - \frac{2}{\sqrt{3}} \cot \pi/5$. \qed
\end{corollary}


We gather together several trigonometric facts needed to construct the
$120$--cell.  For an elementary and enlightening discussion, see
Langlands' lectures~\cite[Part~3,~pages~1-9]{LanglandsDuke}.


\begin{center}
\begin{tabular}{cccc}
  $\theta$  &$\cos \theta$ &$\sin \theta$ &$\cot \theta$ \\
\midrule
$\pi/5$     &$\frac{1}{4}\left(1 + \sqrt{5}\right)$  
            &$\frac{1}{4}\sqrt{10 - 2\sqrt{5}}$ 
            &$\sqrt{1 + \frac{2}{\sqrt{5}}}$ \\
$2\pi/5$    &$\frac{1}{4}\left(-1 + \sqrt{5}\right)$  
            &$\frac{1}{4}\sqrt{10 + 2\sqrt{5}}$ 
            &$\sqrt{1 - \frac{2}{\sqrt{5}}}$ 
\end{tabular}
\end{center}
Deduce the following identities. 
\begin{align}
\label{Eqn:Trig1}
\cot^2 \pi/5 + \cot^2 2\pi/5 &= 2 \\
\label{Eqn:Trig2}
4 \cos^2 \pi/5  - 2 \cos \pi/5 - 1 &= 0
\end{align}

\section{Four-space and quaternions}
\label{Sec:Quatern}

In this section we review the quaternions, the three-sphere, and
stereographic projection.  See also~\cite[Chapter~6]{Coxeter74},
\cite[Section~2.7]{Thurston97}, or~\cite[Part~II]{ConwaySmith03}. The
quaternions bridge the gap between the algebra of certain groups and
the geometry of four-dimensional space.  The three-sphere is the
natural home of the spherical $120$--cell.

Due to the physiology of the human eye, we only ever see
two-dimensional images.  The brain instinctively interprets some of
these as representing three-dimensional objects, but is not equipped
to deal with higher dimensions.  Hence we do not attempt to draw any
native pictures of four-dimensional objects.
Instead, we use stereographic projection to transport objects from the
three-sphere into three-dimensional space, where they can be seen with
human eyes.


\subsection{The quaternions}

The real numbers $\RR$, being one-dimensional, can be augmented by
adding $i = \sqrt{-1}$ to obtain the two-dimensional complex numbers
$\CC$.  In very similar fashion Hamilton augmented $\CC$ to obtain the
quaternions $\HH$.  Let $\basis{1, i, j, k}$ be the usual orthonormal
basis for $\RR^4$.  We take $\HH = \RR \oplus \II$, where $\II = i\RR
\oplus j\RR \oplus k\RR$ is the subspace of \emph{purely imaginary}
quaternions.  Following Hamilton we endow $\HH$ with the relations
\[
i^2 = j^2 = k^2 = ijk = -1. 
\]
These relations, $\RR$--linearity, associativity, and distributivity
allow us to compute any product in $\HH$.  

If $p = a + bi + cj + dk \in \HH$ then we call $a$ the \emph{real
  part} of $p$ and $bi + cj + dk$ the \emph{imaginary part} of $p$.
We call $\Cong{p} = a - bi - cj - dk$ the \emph{conjugate} of $p$.
Since $ij = -ji$ and so on, we deduce that $\Cong{p \cdot q} =
\Cong{q} \cdot \Cong{p}$ for any $p, q \in \HH$.  

It is impossible to separate the algebra of the quaternions from their
geometry.  For example, the usual norm and Euclidean distance on $\HH$
are given by
\[
|p| = \sqrt{p \Cong{p}} = \sqrt{a^2 + b^2 + c^2 + d^2} \quad \mbox{and}
\quad d_\HH(p,q) = |p - q|.
\]
Thus $|pq|^2 = pq \Cong{pq} = pq \cdot \Cong{q} \cdot \Cong{p} = p
|q|^2 \Cong{p} = |p|^2 |q|^2$, and so $|pq| = |p||q|$.

Since $\HH$ is identical to $\RR^4$ as a real vector space, there is a
copy of the three-sphere inside the quaternions: namely, $S^3 = \{ q
\in \HH : |q| = 1\}$.  The metric on $\HH$ induces the round metric on
the sphere, namely
\[
d_S(p,q) = \arccos(\langle p, q \rangle), 
\]
where $\langle p, q \rangle = \sum p_i q_i$ is the usual inner
product.  
The function from $S^3$ to itself taking $p$ to $-p$ is called the
\emph{antipodal map}.  When $L \subset \HH$ is a linear subspace of
dimension one, two, or three the intersection $L \cap S^3$ is a pair
of antipodal points, a \emph{great circle}, or a \emph{great sphere},
respectively.  We call the antipodal points $1$ and $-1$, as they lie
in $S^3$, the \emph{south} and \emph{north} poles, respectively.  We
call $S^2_\II = S^3 \cap \II$ the \emph{equatorial} great sphere.  See
\reffig{Axes} for a depiction of how several great circles among
$1,i,j,k$ lie inside of $S^3$.

\subsection{The unit quaternions}

The points of the three-sphere, the \emph{unit quaternions}, form a
group under quaternionic multiplication.  The point $1 \in S^3$ serves
as the identity, associativity follows from the associativity of
$\HH$, and inverses are given by $q^{-1} = \Cong{q}$.
Again, we see how the group structure and geometry of $S^3$ are
tightly intertwined, as follows. 

\begin{lemma}
\label{Lem:Isom}
The left and right actions of $S^3$ on $\HH$ are via
orientation-preserving isometries.  The same holds for the
three-sphere's action on itself.
\end{lemma}

\begin{proof}
Fix $p \in S^3$ and $q, r \in \HH$.  We compute $d_\HH(pq, pr) = |pq -
pr| = |p(q-r)| = |p||q-r| = |q - r| = d_\HH(q,r)$, verifying the left
action is via isometry.  Since $S^3$ is connected, and since $1$ acts
trivially, the action is orientation preserving.  Also, the action
preserves the three-sphere, and so preserves the induced metric.
\end{proof}

The group elements $\pm 1$ are very special; they are the only
elements that are their own inverses.  The sphere $S^2_\II$ of pure
imaginaries is much more homogeneous, as follows.

\begin{lemma}
\label{Lem:Orthogonal}
We have 
\[
u^2 = v^2 = w^2 = uvw = -1
\]
when $\basis{u,v,w}$ is a right-handed orthonormal basis for
$\II$. \qed
\end{lemma}


We can now parametrise great circles in $S^3$ through the identity.
For any $u \in S^2_\II$ define $L_u = \basis{1, u}$ to be the
corresponding plane in $\HH$.  The intersection $L_u \cup S^3$ is thus
a great circle $C_u$.  We parametrise $C_u$ by sending $\alpha \in
\RR$ to the point
\begin{align}
\label{Eqn:Param}
e^{u\alpha} &= \cos \alpha + u \cdot \sin \alpha.
\end{align}

\begin{wrapfigure}[16]{r}{0.40\textwidth}
\vspace{-5pt}
\centering 
\labellist
\pinlabel $\mathbb{I}$ at 97 263
\small\hair 2pt
\pinlabel $u$ at 117 246
\pinlabel $-u$ at 120 17
\pinlabel $1$ at 218 133
\pinlabel $-1$ at -13 132
\pinlabel {$q = e^{u\alpha}$} [Bl] at 163 227
\pinlabel $\alpha$ at 126 142
\pinlabel $\rho(q)$ at 87 196
\endlabellist
\includegraphics[width=0.30\textwidth]{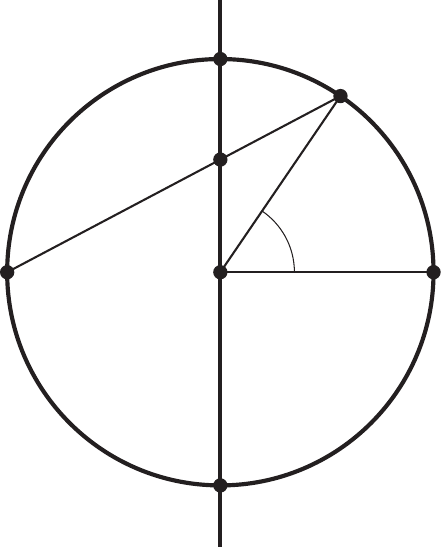}
\caption{Stereographic projection from $S^1 - \{-1\}$ to
  $\mathbb{I}$.}
\label{Fig:StereoProj}
\end{wrapfigure}
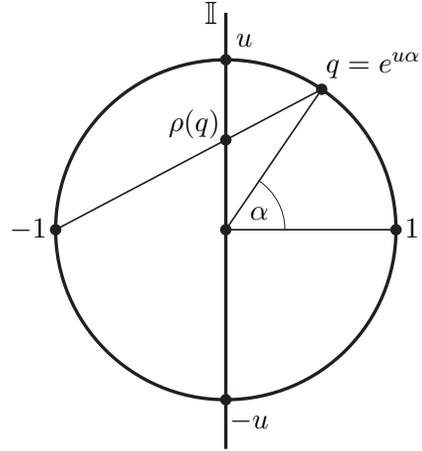

\mbox{}
\vspace{-17pt}

\begin{lemma}
\label{Lem:TrigProd}
For any pure imaginary $u \in S^2_\II$ and for any $\alpha, \beta \in
\RR$ we have $e^{u\alpha} e^{u\beta} = e^{u(\alpha + \beta)}$.  Thus
$\{e^{u\alpha}\}$ is a one-parameter subgroup of $S^3$.  Also, $d_S(1,
e^{u\alpha}) = \alpha$ for $\alpha \in [0,\pi]$. \qed
\end{lemma}



This gives a parametrisation of $S^3$, as follows.

\begin{lemma}
\label{Lem:TrigExistsUnique}
For any $q \in S^3 - \{\pm 1\}$ there is a unique $u \in S^2_\II$ and
a unique $\alpha \in (0, \pi)$ so that $q = e^{u\alpha}$. \qed
\end{lemma}

\subsection{Stereographic projection}
\label{Sec:StereoProj}

Throughout the paper we use stereographic projection to visualise
objects in, and motions of, the three-sphere.  Recall that $\II$ is a
copy of $\RR^3$.  We define stereographic projection $\rho \from S^3 -
\{-1\} \to \II$ by
\[
\displaywidth=\parshapelength\numexpr\prevgraf+2\relax
\rho(q) = \frac{\sin(\alpha)}{1 + \cos(\alpha)} \cdot u
\]
with $q = e^{u\alpha}$ as in \reflem{TrigExistsUnique}.  See
\reffig{StereoProj} for a cross-sectional view.  Note that $\rho$
sends the south pole to the origin, fixes the equatorial sphere
$S^2_\II$ pointwise, and sends the north pole to ``infinity''.  The
one-parameter subgroup $e^{u\theta}$ is sent to the straight line in
the direction of $u$.  \reffig{Axes} shows the result of applying
stereographic projection to various great circles connecting $1, i, j,
k$ inside of $S^3$.

\begin{wrapfigure}[16]{r}{0.50\textwidth}
\centering
\labellist
\scriptsize\hair 2pt
\pinlabel $-j$ at 100 183
\pinlabel $i$ at 174 114
\pinlabel $1$ at 207 166
\pinlabel $k$ at 212 257
\pinlabel $-k$ at 214 47
\pinlabel $-i$ at 257 186
\pinlabel $j$ at 307 140
\endlabellist
\includegraphics[width=0.45\textwidth]{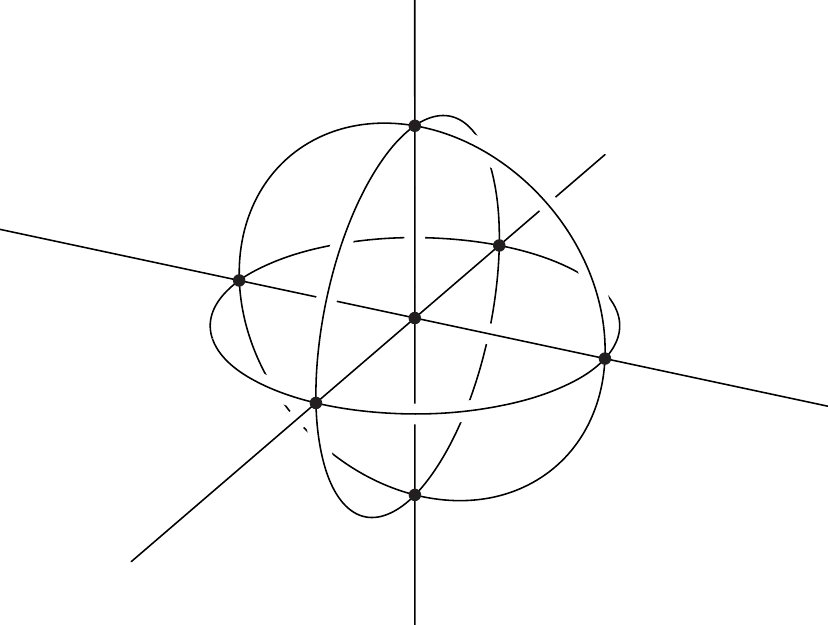}
\caption{Several great circles connecting $1, i, j, k$, shown after
  stereographic projection to $\mathbb{R}^3$.}
\label{Fig:Axes}
\end{wrapfigure}
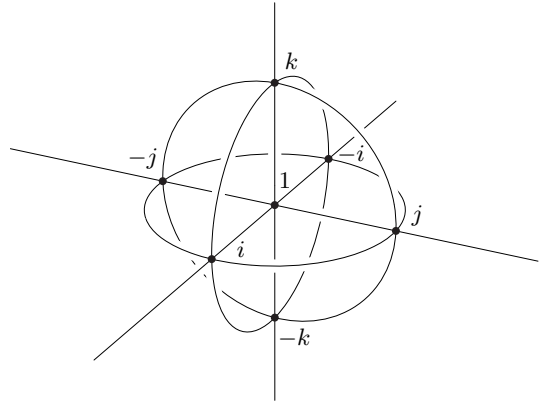

\subsection{Mapping to $\SO(3)$}
\label{Sec:SO3}


Recall that $\SO(3)$ is the group of three-by-three orthogonal
matrices with determinant one.
Taking $\basis{i,j,k}$ as a basis for $\II$, we identify $\SO(3)$ with
$\Isom^+_0(\II)$, the group of orientation-preserving isometries of
$\II$ fixing the origin.

In \reflem{Isom} we discussed the left and right actions of $S^3$ on
$\HH$.  We combine these to obtain the \emph{twisted action}: for $q
\in S^3$ define $\phi_q \from \HH \to \HH$ by $\phi_q(p) = qpq^{-1}$.
The twisted action is again via isometries.  Note that the action
preserves $\RR \subset \HH$ pointwise.  Thus it preserves $\II \subset
\HH$ setwise.  We define $\psi_q \from \II \to \II$ by $\psi_q =
\phi_q|\II$ and deduce the following.

\begin{lemma}
\label{Lem:Homo}
The map $\psi_q$ is an element of $\SO(3)$.  The induced map $\psi
\from S^3 \to \SO(3)$ is a group homomorphism.
\end{lemma}

\begin{proof}
As remarked above, $\psi_q$ is an isometry of $\II$ that fixes the
origin.  Since $S^3$ is connected, the isometries $\psi_q$ and $\psi_1
= \Id$ have the same handedness.  Thus $\psi_q$ lies in $\SO(3)$.  The
equality $\psi_{qr} = \psi_q \psi_r$ follows from the associativity of
$\HH$.
\end{proof}

We need an explicit form of $\psi$, discovered independently by Gauss,
Rodrigues, Cayley, and Hamilton~\cite[page~21]{Stillwell01}.

\begin{lemma}
\label{Lem:Cover}
For $q = \pm e^{u\alpha}$ the isometry $\psi_q$ is a rotation of $\II$
about the direction $u$ through angle $2\alpha$.  Thus $\psi \from S^3
\to \SO(3)$ is a double cover.
\end{lemma}

\begin{proof}
As a convenient piece of notation, we write $q = a + bu$ where $a =
\cos(\alpha)$ and $b = \sin(\alpha)$.  So $q^{-1} = a - bu$.  We check
that $\psi_q(u) = u$.
\begin{align*}
\psi_q(u) &= quq^{-1} = (a+bu)u(a-bu) \\
          &= (au-b)(a-bu) \\
          &= a^2u + ab - ab + b^2u \\
          &= u
\end{align*}
By Euler's rotation theorem, the line through $u$ is an axis for
$\psi_q$.
Now suppose that $v$ is orthogonal to $u$.  Let $w = uv$.  Thus
$\basis{u,v,w}$ is a right-handed orthonormal basis of $\II$.  We
compute $\psi_q(v)$.
\begin{align*}
\psi_q(v) &= (a + bu) v (a - bu) \\ 
          &= a^2 v - ab vu + ab uv - b^2 uvu \\
          &= a^2 v + 2ab w - b^2 uvu \\
          &= (a^2 - b^2)v + 2ab w \\
          &= \cos(2\alpha) v + \sin(2\alpha) w
\end{align*}
Thus $\psi_q$ rotates by the desired amount.  It follows from the
rotation theorem that $\psi$ is surjective.  Note that $\psi_q = \Id$
if and only if $\cos(2\alpha) = 1$ if and only if $\alpha \in
\{0,\pi\}$.  Thus $\psi$ is two-to-one.  We leave the proof that
$\psi$ is a covering map as a topological exercise.
\end{proof}

\begin{definition}
\label{Def:Binary}
If $\calG \subset \SO(3)$ is a group, then we call $\calG^* =
\psi^{-1}(\calG)$ the \emph{binary} group corresponding to $\calG$.
\end{definition}

\section{The $120$--cell}
\label{Sec:120}

It is time to construct the $120$--cell.  We could use a continuity
argument, as in \refsec{DodecaCon}, to build a spherical dodecahedron
in $S^3$ with all dihedral angles equal to $2\pi/3$.  The Poincar\'e
polyhedron theorem would then produce a tiling of $S^3$; regularity of
the tile leads to regularity of the tiling.  Taking the convex hull of
the vertices would give the $120$--cell.  However, computing the
number of cells would require computing the volume of the spherical
flag polytope, a highly non-trivial task.  Also, it is crucial for us
to see how the binary dodecahedral group $\calD^*$ lies inside of the
symmetry group of the $120$--cell.  Thus we give a more explicit
construction.  We refer to~\cite{Cheritat12, Stillwell01, Sullivan91}
as very useful commentaries on the $120$--cell.

\subsection{Outline of the construction}

\begin{wrapfigure}[15]{r}{0.40\textwidth}
\vspace{-22pt}
\centering 
\labellist
\scriptsize\hair 2pt
\pinlabel $v$ at 194 195
\pinlabel $i$ at 74 125
\pinlabel $j$ at 298 140
\pinlabel $k$ at 181 297
\pinlabel $f$ at 157 110
\endlabellist
\includegraphics[width=0.36\textwidth]{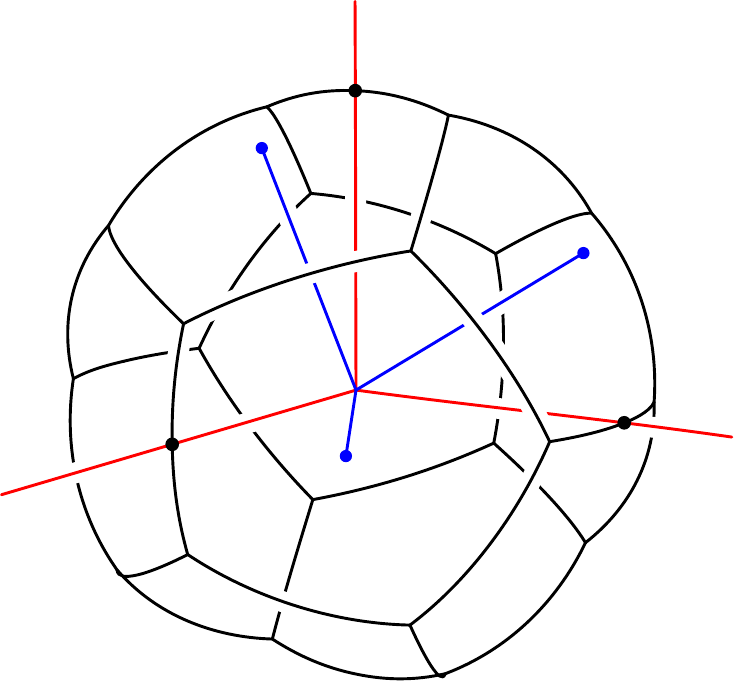}
\caption{The tiling $\calT_D$ can be positioned with one vertex at $v
  = \frac{1}{\sqrt{3}}(i + j + k)$ and with one face centre $f$ in the
  $ij$--plane.}
\label{Fig:DodecaPosition}
\end{wrapfigure}
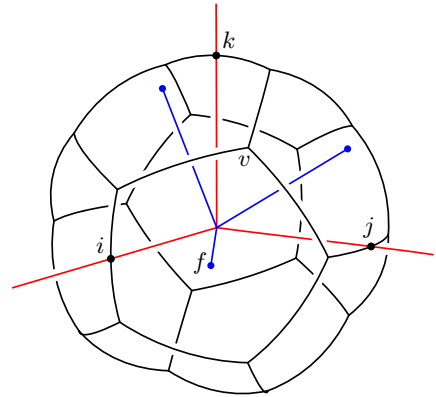

Let $\calT_D \subset S^2_\II$ be the tiling constructed in
\refdef{Dodecahedron}.  Let $\calD \subset \SO(3)$ be its group of
orientation-preserving symmetries.  As in \refdef{Binary}, let
$\calD^* \subset S^3$ be the binary dodecahedral group.  From
\reflem{Tiling} deduce that $\calD^*$ has $120$ elements.  Let
$\calT_{120}$ be the tiling of $S^3$ by Voronoi domains about the
points of $\calD^*$.  We show that each domain is a regular spherical
dodecahedron.  Taking the convex hull of the vertices of $\calT_{120}$
yields the $120$--cell.  We now give the details.

\subsection{Positioning the dodecahedron}
\label{Sec:Pos}

As in \refdef{Dodecahedron}, let $\calT_D \subset \II \isom \RR^3$ be
the tiling of the unit sphere $S^2_\II$ by twelve spherical pentagons.
See \reffig{DodecaPosition} for a picture of the edges.  We
rotate $\calT_D$ to have one vertex at the point $v =
\frac{1}{\sqrt{3}}(i + j + k)$.  This done, the vertex
rotation about $v$ permutes the coordinate planes.  Pick $f \in
\calT_D$ to be one of the three face centres closest to $v$.  We wish
to rotate $\calT_D$, about the line through $0$ and $v$, to bring $f$
into the $ij$--plane.  To show that this is possible, and to find the
resulting coordinates of $f$, suppose $f = xi + yj$, where $x^2 + y^2
= 1$.  We now compute.
\begin{align*}
|v - f|^2 &= 1 - \frac{2}{\sqrt{3}}(x+y) + x^2 + y^2 \\
          &= 2 - \frac{2}{\sqrt{3}}(x+y).
\end{align*}
From \refcor{Edge} deduce that $x + y = \cot \pi/5$.  Solving the resulting
quadratic in $x$, and applying \refeqn{Trig1}, yields
\[
\{x, y\} = \left\{ \frac{ \cot \pi/5 \pm \cot 2\pi/5 }{2} \right\}.
\]
We choose the solution where $x > y$.  The resulting position of
$\calT_D$ is shown in \reffig{DodecaPosition}

Using the vertex rotation about $v$ deduce $f' = xj + yk$ and $f'' =
yi + xk$ are the other face centres of $\calT_D$ that are closest to
$v$.  We finish by noting, as indicated in \reffig{DodecaPosition},
there are three edges centres of $\calT_D$ at the points $i$, $j$, and
$k$.  As with \reflem{Edge}, verifying this is an exercise in
spherical trigonometry.

\subsection{Voronoi cells}

\begin{wrapfigure}[8]{r}{0.25\textwidth}
\vspace{-20pt}
\centering 
\includegraphics[width=0.20\textwidth]{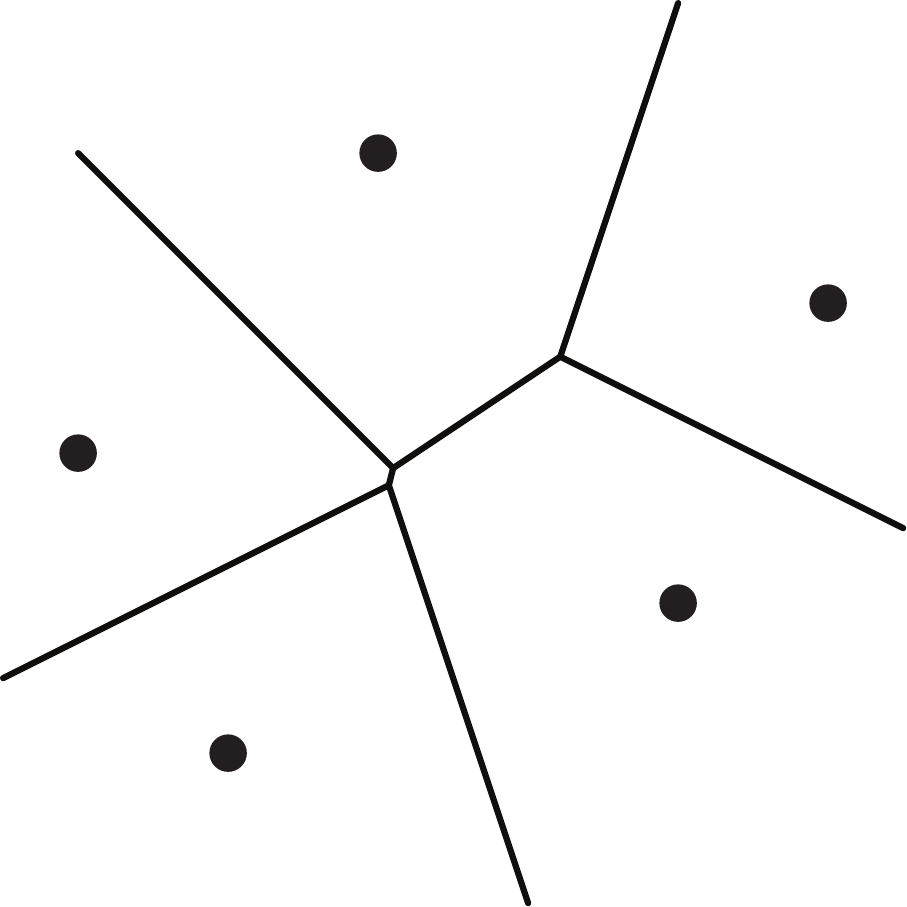}
\label{Fig:Voronoi}
\end{wrapfigure}

Suppose $V$ is a finite set of points in a metric space $X$.  The
\emph{Voronoi cell} about a point $q \in V$ is the set
\[
\displaywidth=\parshapelength\numexpr\prevgraf+2\relax
\Vor(q) = \{ r \in S^3 \st \mbox{for all $p \in V$, $d_X(q,r) \leq
  d_X(r,p)$} \}.
\]
An example with five points, in $\RR^2$, is shown to the right.  Let
$\calD \subset \SO(3)$ be the group of orientation-preserving
symmetries of the dodecahedron $D$, as given in \refsec{Pos}.  Let
$\calD^* \subset S^3$ be the corresponding binary dodecahedral group,
of $120$ elements.  Let $\calT_{120}$ be the tiling of the
three-sphere by the cells $\{\Vor(q) \st q \in \calD^*\}$.  Define
$\calC = \Sym(\calT_{120})$.

\begin{lemma}
\label{Lem:Sym120}
The left action of $\calD^*$ on $\calT_{120}$ is transitive on the
three-cells.  The twisted action of $\calD^*$ fixes $\Vor(1)$ setwise.
Both actions give homomorphisms of $\calD^*$ to $\calC$.  \qed
\end{lemma}

\begin{lemma}
\label{Lem:Voronoi}
Each cell $\Vor(q)$ is a regular spherical dodecahedron. 
\end{lemma}

\begin{proof}
Let $1$ be the identity of $S^3$.  By \reflem{Sym120} it suffices to
prove the lemma for $\Vor(1)$.  For any $q \in \calD^*$, not equal to
$1$, we define $\Sphere(q) \subset S^3$ to be the great sphere
equidistant from $1$ and $q$.  Note that $\Vor(1)$ is obtained by
cutting $S^3$ along $\Sphere(q)$, for all $q\neq 1$, and taking the closure
of the component that contains $1$.

By \refcor{ElementsOfD} and by Lemmas~\ref{Lem:Cover}
and~\ref{Lem:TrigProd} there are twelve quaternions
$\{q_i\}_{i=1}^{12}$ in $\calD^*$ that are distance $\pi/5$ from $1$.
Define $U$ by cutting $S^3$ along the spheres $\Sphere(q_i)$ only, and
then taking the closure of the component containing $1$.  By
\reflem{Sym120} the twisted action of $\calD^*$ preserves $\{q_i\}$
setwise; we deduce $U$ is a regular spherical dodecahedron.  Also, $U$
contains $\Vor(1)$.

\begin{claim*}
$U = \Vor(1)$.
\end{claim*}

\begin{proof}
We must show, for every $p \in \calD^* - \{q_i\}$, that the sphere
$\Sphere(p)$ misses $U$.  We will only do this for a single lift of a
vertex rotation of $D$, leaving the other cases as exercises.


\begin{figure}[htbp]
\centering 

\includegraphics[width=0.8\textwidth]{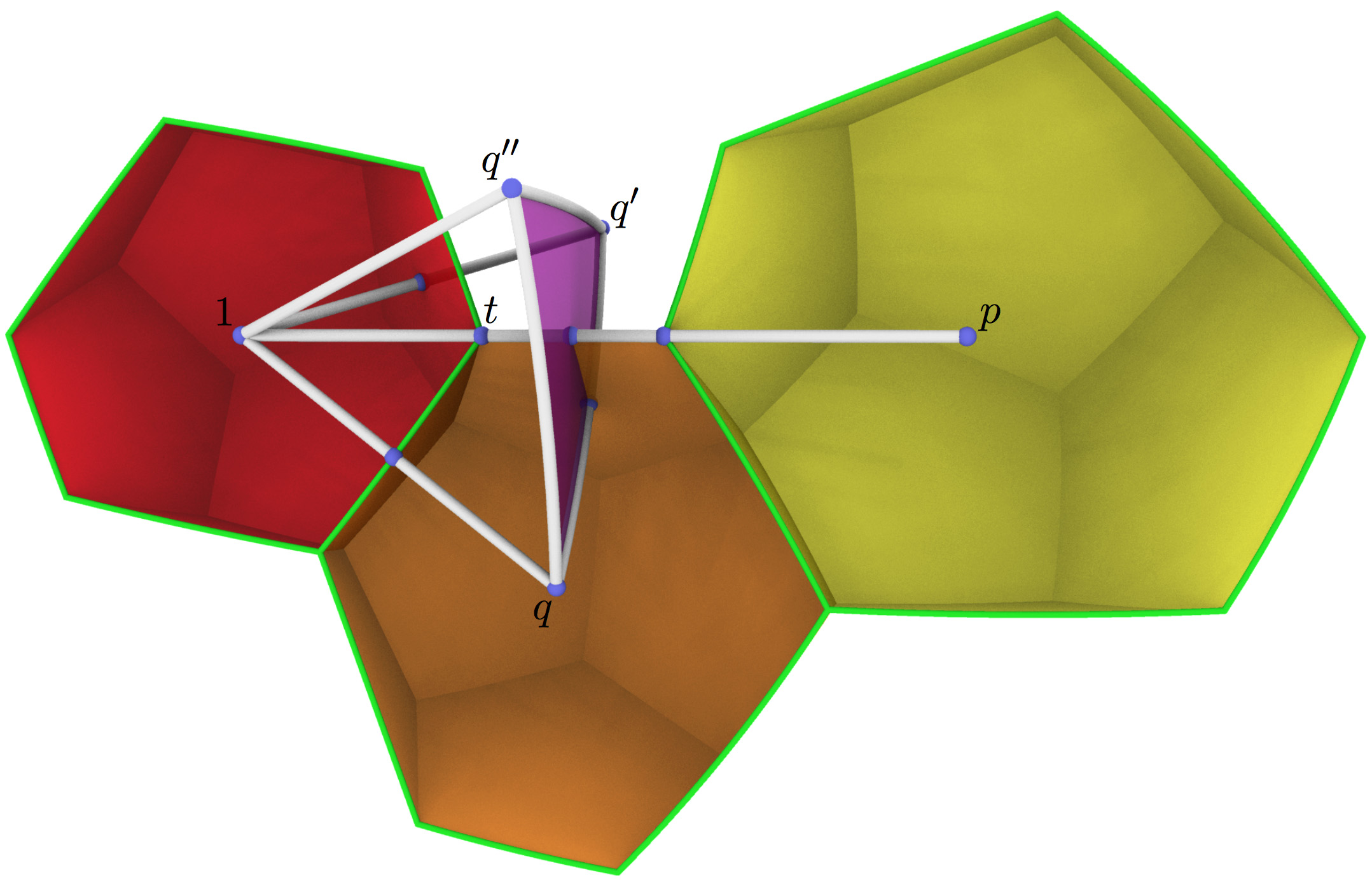}
\caption{Three dodecahedral cells of the tiling $\calT_{120}$, each
  chopped in half.  The triangle meeting the central cell has vertices
  at $q$, $q'$, and $q''$; also it bisects the geodesic connecting $1$
  and $p$.  The point $t$ is equidistant from $1$, $q$, $q'$, and
  $q''$.}
\label{Fig:ProveVoroni}
\end{figure}

Take $v$, $f$, $f'$, and $f''$ as defined in \refsec{Pos}.  Fix the
following quaternions in $\calD^*$
\begin{align*}
p &= \cos \pi/3 + v \cdot \sin \pi/3, \\
q &= \cos \pi/5 + f \cdot \sin \pi/5 
\end{align*}
and define $q'$ and $q''$ similarly with respect to $f'$ and $f''$.
Thus $p$ is the desired lift of the vertex rotation about $v$.  Note
$q$, $q'$, and $q''$ are lifts of face rotations.  By
\reflem{TrigProd} the elements $q$, $q'$, and $q''$ are all distance
$\pi/5$ from $1$ in $S^3$.  We compute
\begin{align*}
(q^{-1}) \cdot q' 
  &= (\cos \pi/5 - f \cdot \sin \pi/5) (\cos \pi/5 + f' \cdot \sin \pi/5) \\
  &= \cos^2 \pi/5 + (- f + f') \cos \pi/5 \sin \pi/5 - f f' \cdot \sin^2 \pi/5.
\end{align*}
Expanding the product $f f'$ and applying \refeqn{Trig2}, we find the
real part of $(q^{-1}) \cdot q'$ is also equal to $\cos \pi/5$.  Since
the twisted action by $p$ permutes $q$, $q'$, and $q''$ cyclically,
deduce that $1$, $q$, $q'$, and $q''$ are the vertices of a regular
spherical tetrahedron, $T$.  Let $t = \centerpt(T)$ be the spherical
centre of $T$ -- the radial projection of the Euclidean centre of $T$.
It follows that $t$ is a vertex of $U$.  We claim $t$ is the point of
$U$ closest to $p$.
Note the real part of $t$ is
$\frac{1}{2} \sqrt{1 + 3 \cos\pi/5}$.  Since this is greater than $\cos \pi/6$ 
deduce that $\Sphere(p)$ does not cut $t$ off of $U$.  Thus
$\Sphere(p)$ misses $U$, as desired.
\end{proof}
This completes the proof of \reflem{Voronoi}.
\end{proof}

\begin{definition}
\label{Def:120}
The \emph{$120$--cell} $C$ is the convex hull, taken in $\HH$, of the
vertices of $\calT_{120}$.
\end{definition}

This completes the construction of the $120$--cell.  

\begin{figure}[htbp]
\centering 
\includegraphics[width = 0.70\textwidth]{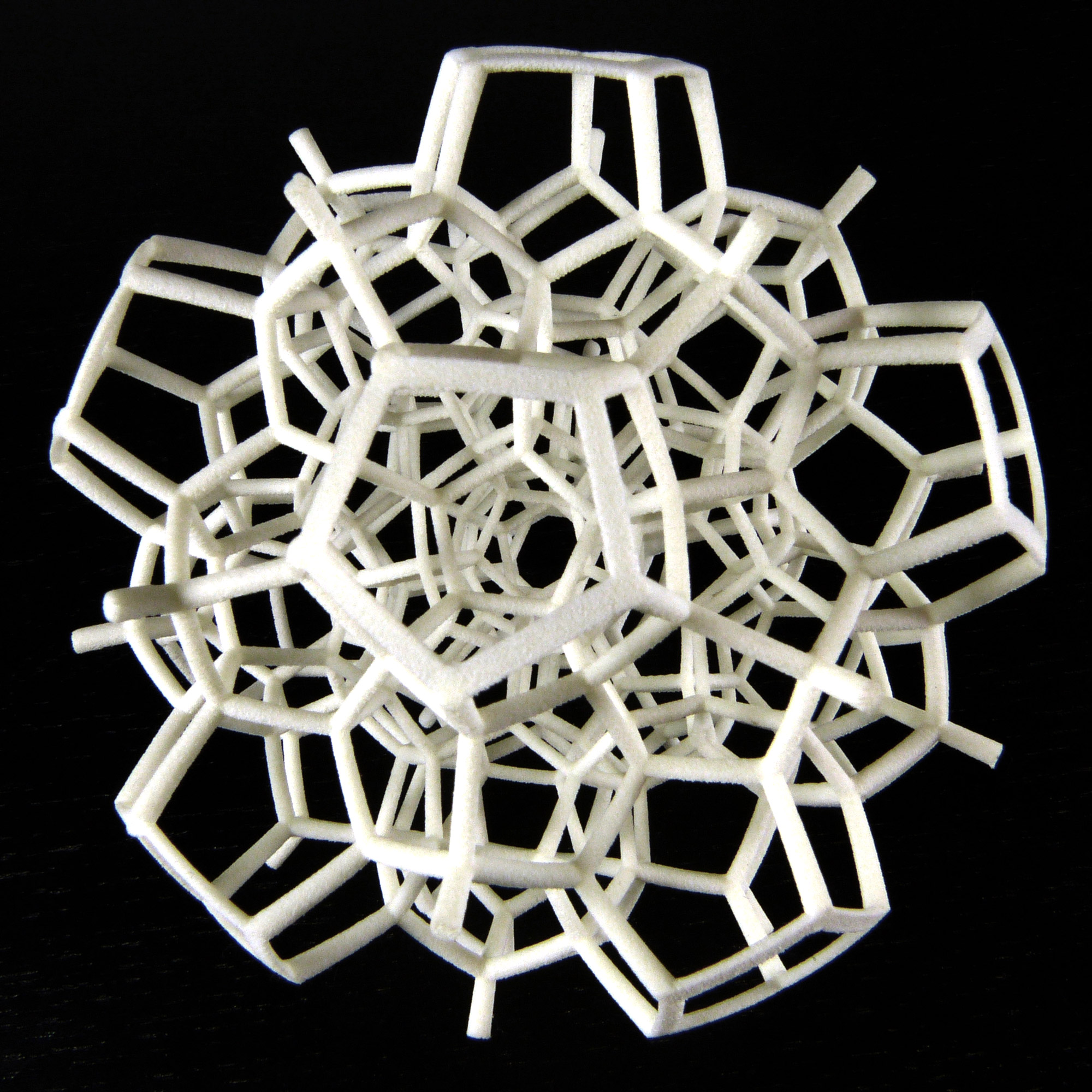}
\caption{One-half of the one-skeleton of the tiling $\calT_{120}$.
  This is the half nearest to the south pole, after cell-centred
  stereographic projection to $\mathbb{R}^3$.  A virtual
  three-dimensional model is available at
  \url{https://skfb.ly/EsTp}.  See also~\cite[colour
    plate]{Sullivan91}.}
\label{Fig:120}
\end{figure}

\begin{theorem}
\label{Thm:120}
The $120$--cell $C$ is a regular polytope.
\end{theorem}

\begin{proof}
We must show that the group $\calC = \Sym(\calT_{120})$ acts
transitively on the flags of $C$.  Now, the flags of $C$ are
four-simplices with one vertex at the origin.  These are in one-to-one
correspondence with the 14,400 spherical flag tetrahedra of $\calT_{120}$.  
It suffices to fix a right-handed spherical flag tetrahedron $T$ of
$\Vor(1)$ and to prove that any other tetrahedron $T'$ in
$\calT_{120}$ can be taken to $T$ by an element of $\calC$.  By
\reflem{Sym120} we may use the left action of $\calD^*$ to transport
$T'$ into $\Vor(1)$.  Now, if $T'$ is also right handed then we may
use the twisted action of $\calD^*$ to send $T'$ to $T$.  There are
several ways to deal with left-handed spherical flags; we resort to a
simple trick.  The conjugation map
\[
a + bi + cj + dk \mapsto a - bi - cj - dk
\]
is the product of three reflections, so is orientation reversing in
$\HH$.  It preserves $S^3$ and is again orientation reversing there.
Since $\calD^*$ is a group of quaternions, it is closed under
conjugation.  Since the tiling $\calT_{120}$ is metrically defined in
terms of $\calD^*$, it is also invariant under conjugation.  This
reverses the handedness of flags, and we are done.
\end{proof}

\begin{corollary}
\label{Cor:Dihedral}
The spherical dodecahedra of $\calT_{120}$ have dihedral angle $2\pi/3$.
\end{corollary}

\begin{proof}
It suffices to check this for $\Vor(1)$, the Voronoi cell about
$1$.  With notation as in the proof of \reflem{Voronoi}: let $1$, $q$,
and $q'$ be elements of $\calD^*$, all at distance $\pi/5$ from each
other.  Let $R$ be the regular spherical triangle having $1$, $q$, and
$q'$ as vertices.  The centre $c = \centerpt(R)$ is equidistant from
the vertices of $R$.  Also, there is a reflection symmetry of
$\calT_{120}$ that fixes $R$ pointwise.
It follows that $\Vor(1)$, $\Vor(q)$, and $\Vor(q')$ share an edge and
this edge is perpendicular to $R$.  As all of these cells are
isometric regular spherical dodecahedra, the corollary follows.
\end{proof}

\begin{remark}
\label{Rem:24Cell}
Note the $24$--cell can be constructed in the same way as the
$120$--cell, by starting with the regular tetrahedron in place of the
dodecahedron.  The symmetries of the cube (equivalently, octahedron)
do not give rise to a regular four-dimensional polytope; the reason
can be traced to the failure of the inequality at the heart of
\reflem{Voronoi}.
\end{remark}




\section{Combinatorics of the $120$--cell}
\label{Sec:Comb}

With the $120$--cell in hand, we turn to the combinatorics of
$\calT_{120}$, the spherical $120$--cell.  By \reflem{Voronoi} and
\refcor{Dihedral}, the cells of $\calT_{120}$ are regular spherical
dodecahedra with dihedral angle $2\pi/3$.

\subsection{Layers of dodecahedra}
\label{Sec:Layers}

Recall that the centres of the cells of $\calT_{120}$ are the elements
of the binary dodecahedral group $\calD^*$.  Recall also that
\refcor{ElementsOfD} lists the elements of $\calD$, ordered by their
angle of rotation.  We deduce that the cells of $\calT_{120}$ divide
into spherical layers, ordered by their distance from the identity
element in $S^3$.  According to our conventions, the identity lies at
the south pole of $S^3$.  \reffig{120_layers} displays the
stereographic projections of the first five layers, expanding from the
south pole out to the equatorial great sphere.  The next four layers,
nesting down to the north pole, are not shown.


\begin{figure}[htbp]
\centering 
\hspace{-1cm}
\subfloat[$0$]
{
\includegraphics[width=0.18\textwidth]{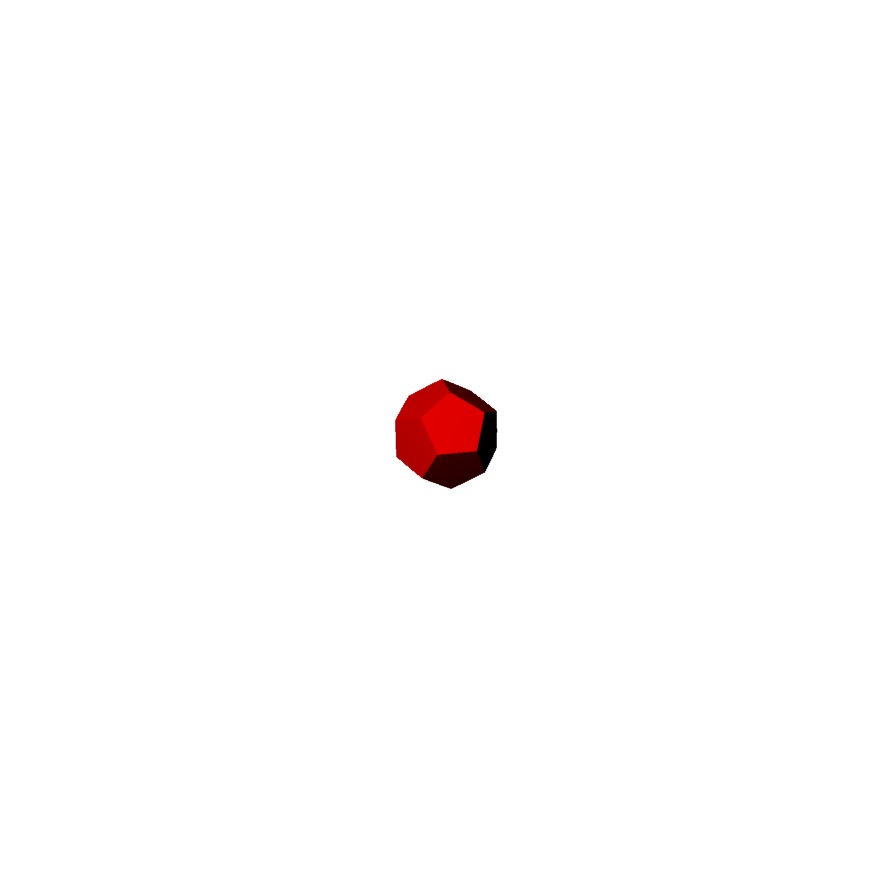}
\label{Fig:120-cell layer 1}
} 
\hspace{-0.5cm}
\subfloat[$\pi/5$]
{
\includegraphics[width=0.18\textwidth]{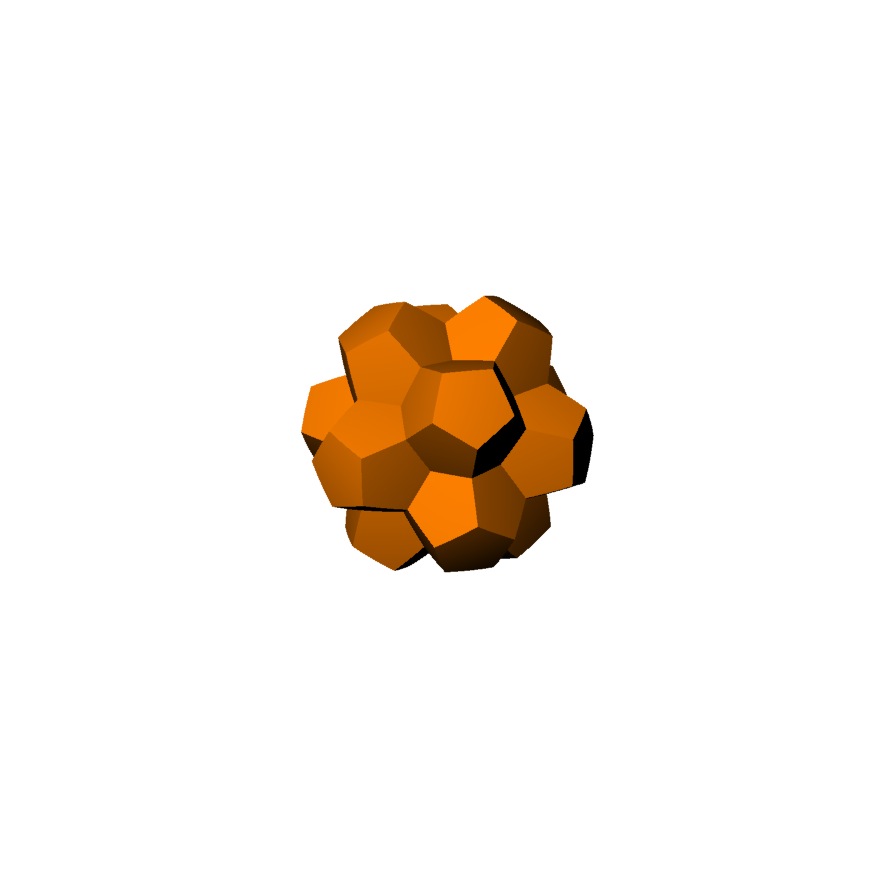}
\label{Fig:120-cell layer 2}
}
\hspace{-0.2cm}
\subfloat[$\pi/3$]
{
\includegraphics[width=0.18\textwidth]{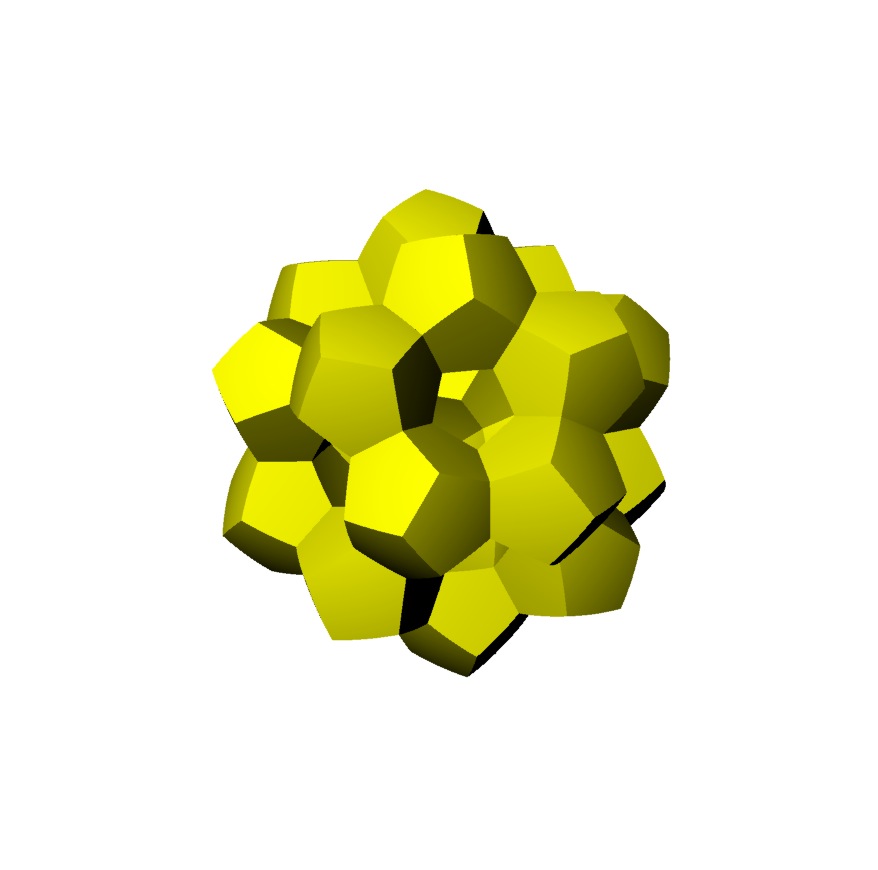}
\label{Fig:120-cell layer 3}
}
\subfloat[$2\pi/5$]
{
\includegraphics[width=0.18\textwidth]{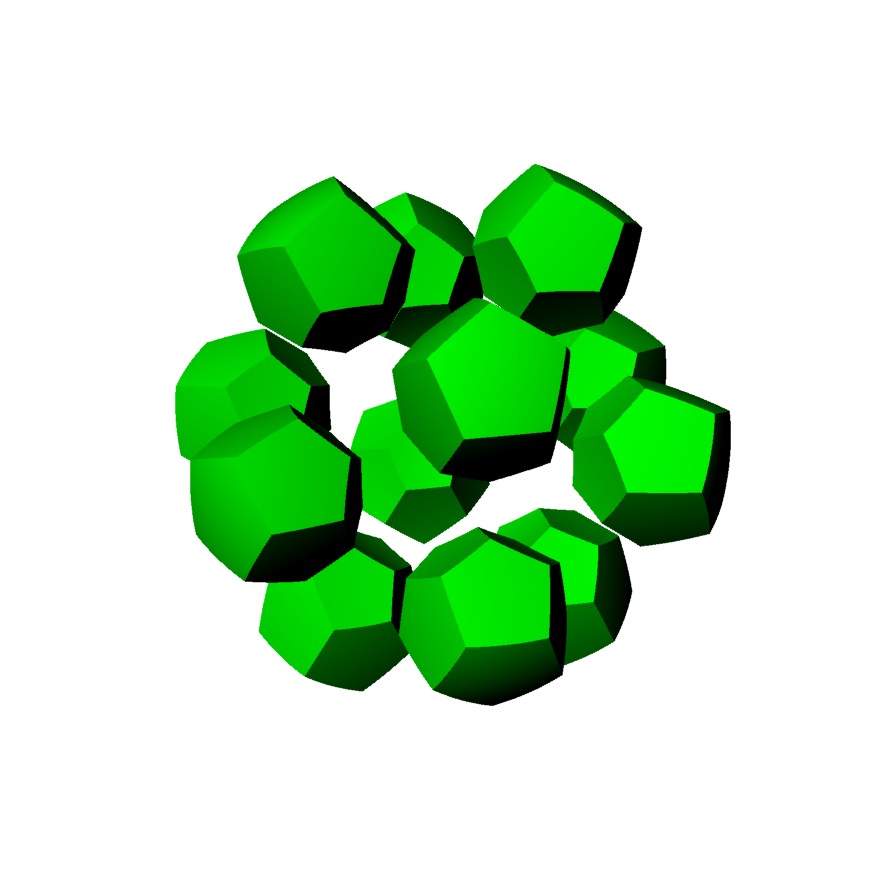}
\label{Fig:120-cell layer 4}
}
\hspace{0.2cm}
\subfloat[$\pi/2$]
{
\includegraphics[width=0.18\textwidth]{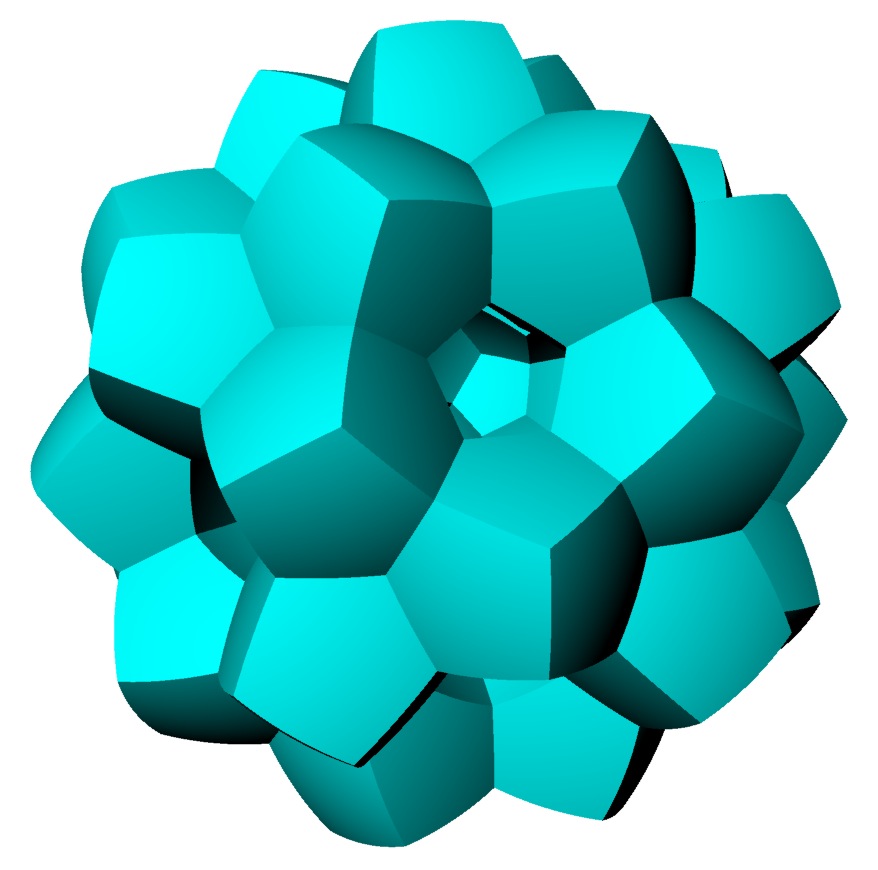}
\label{Fig:120-cell layer 5}
}
\caption{The five layers in the southern hemisphere, ordered by their
  spherical distance from the south pole.  The colours of the cells
  follow the convention of \reffig{DodecWithRotationArrows}.}
\label{Fig:120_layers}
\end{figure}

In \refprop{Ring}, below, for each layer $L$ we list the spherical
distance between $1$ and the cell-centres of $L$, the type of the
covered rotation in $\SO(3)$, the number of cells in $L$, as well as
other data.  See also~\cite[page~176]{Sommerville58}.




\subsection{Rings of dodecahedra}
\label{Sec:Rings}

\begin{figure}[htbp]
\centering 
\subfloat[]
{
\includegraphics[width=0.15\textwidth]{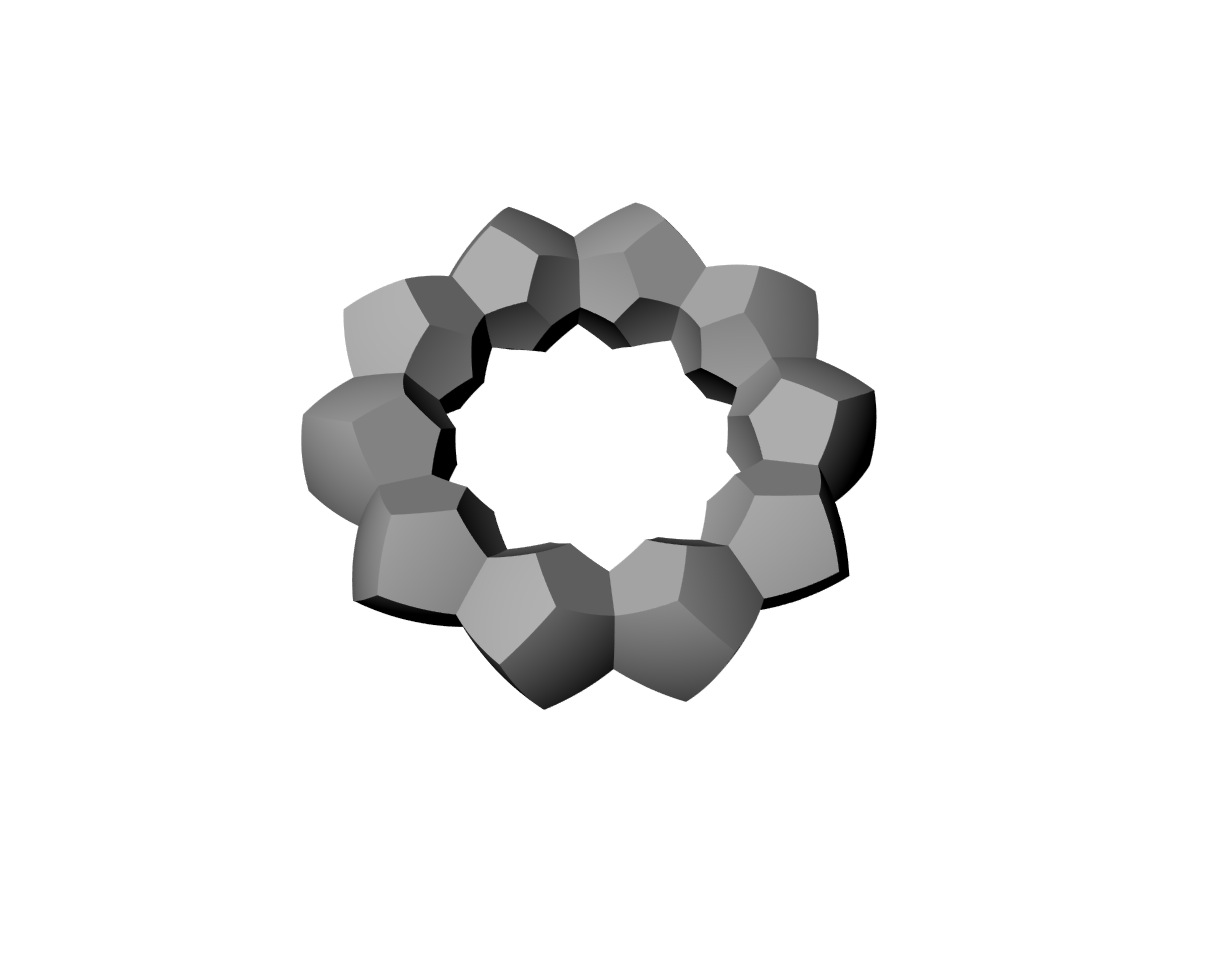}
\label{Fig:SolidDodecRings0}
} 
\subfloat[]
{
\includegraphics[width=0.15\textwidth]{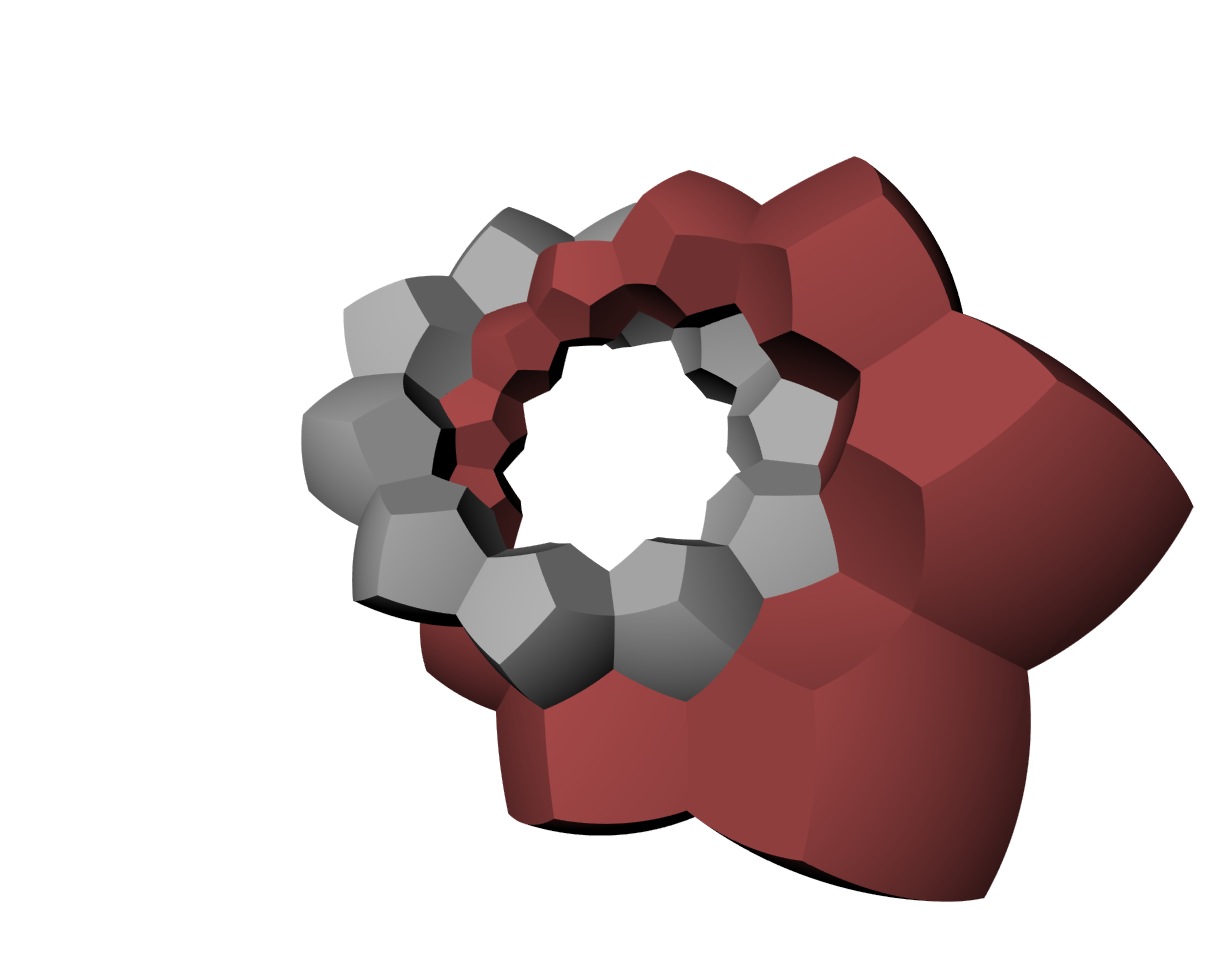}
\label{Fig:SolidDodecRings1}
} 
\subfloat[]
{
\includegraphics[width=0.15\textwidth]{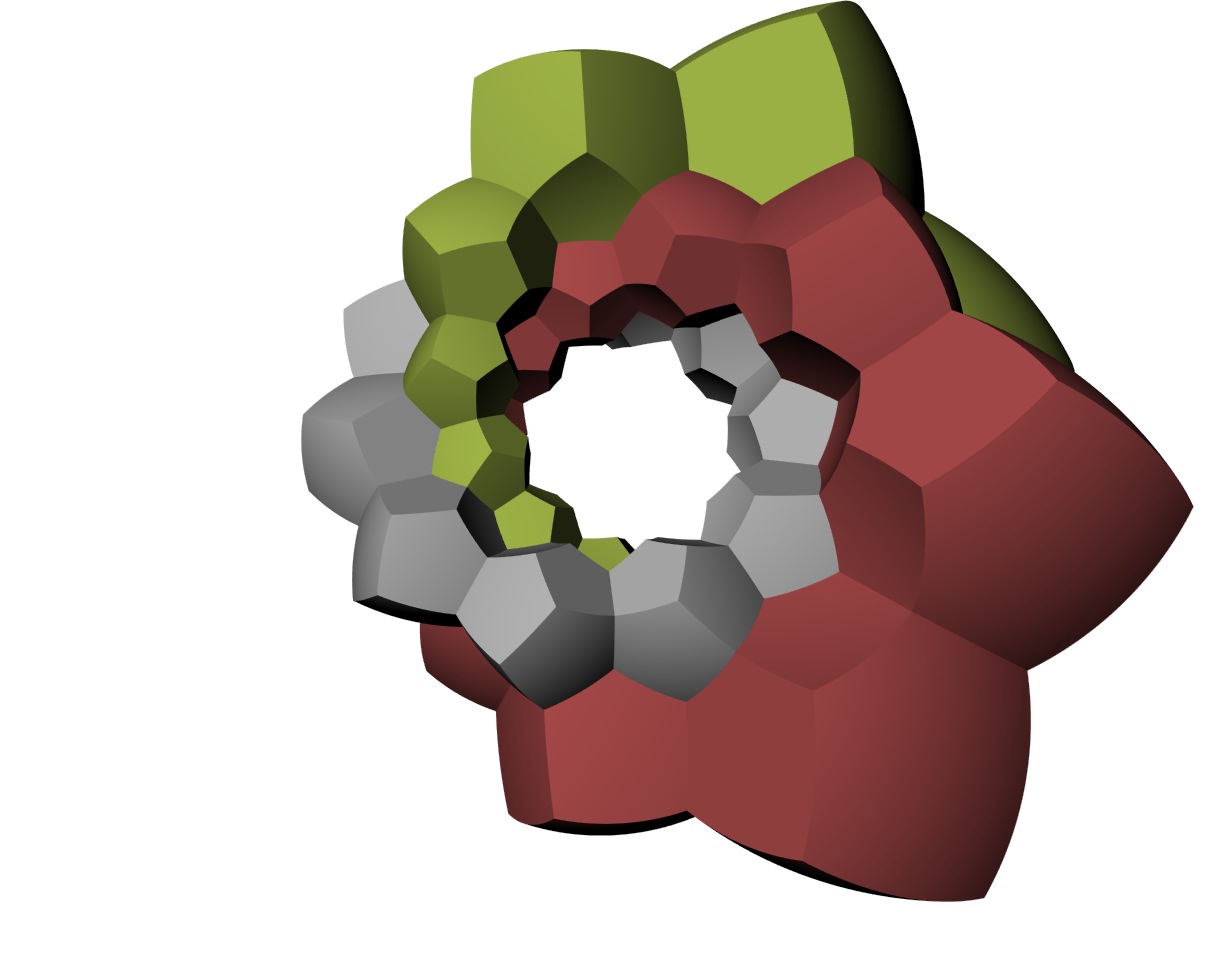}
\label{Fig:SolidDodecRings2}
}
\subfloat[]
{
\includegraphics[width=0.15\textwidth]{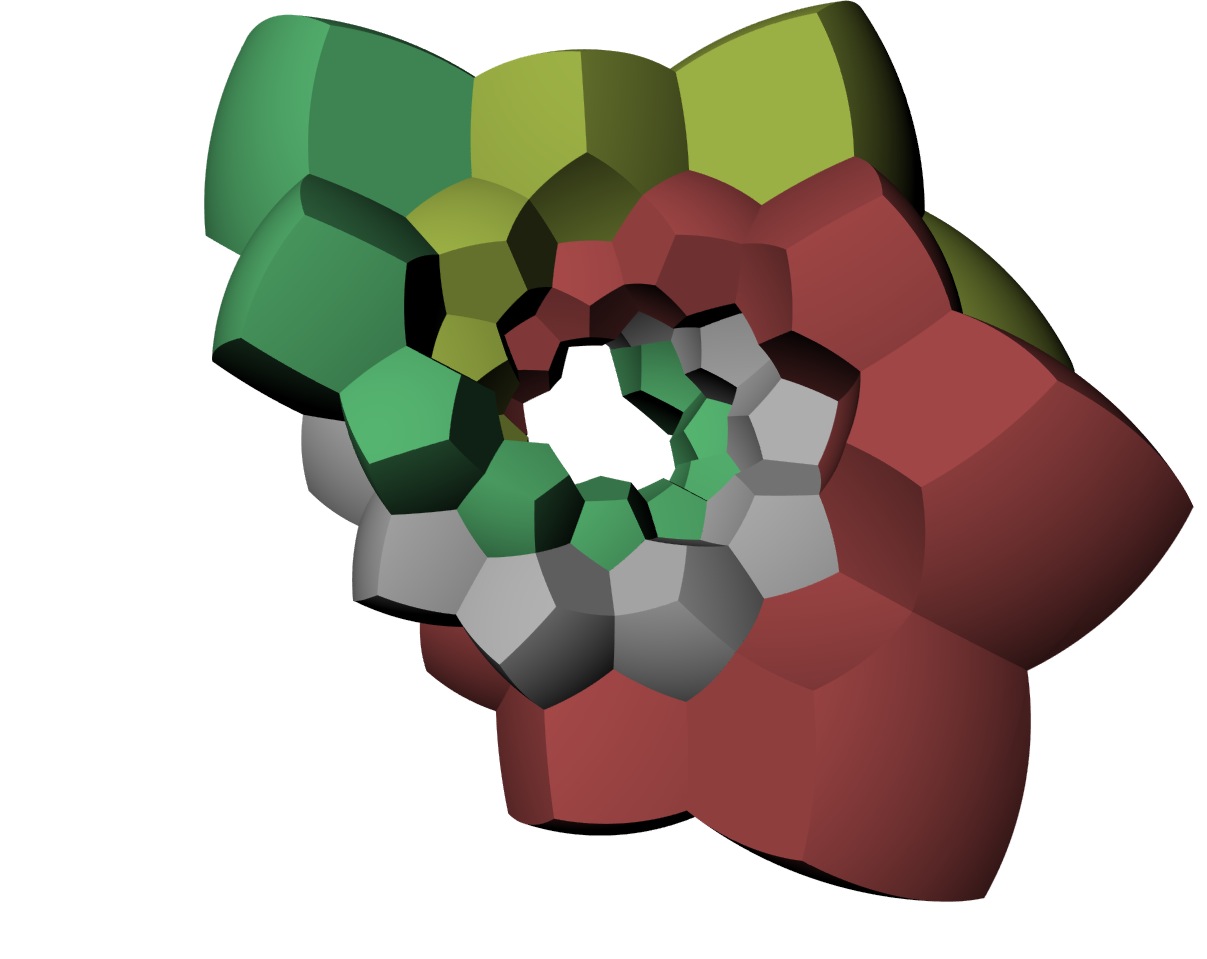}
\label{Fig:SolidDodecRings3}
}
\subfloat[]
{
\includegraphics[width=0.15\textwidth]{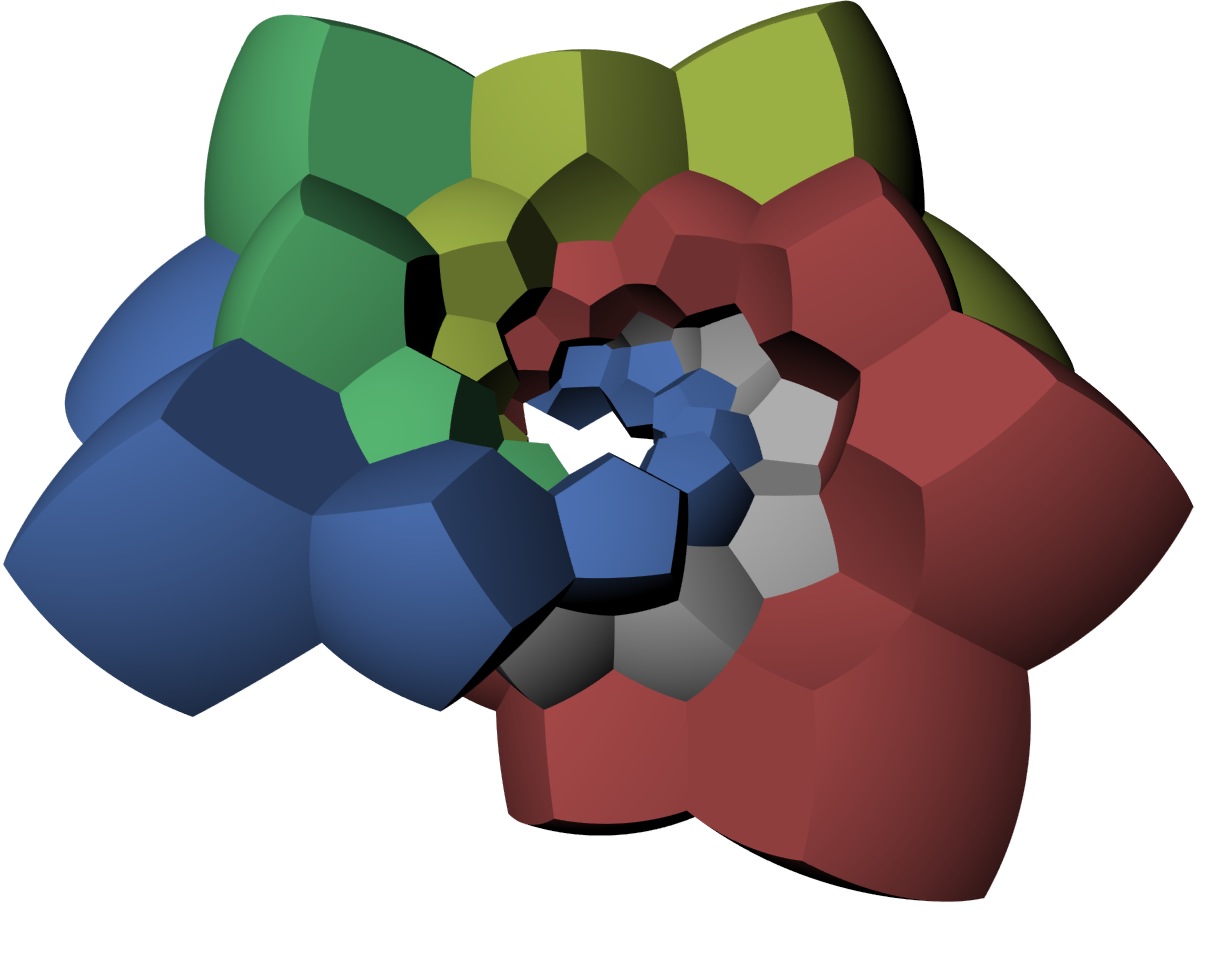}
\label{Fig:SolidDodecRings4}
}
\subfloat[]
{
\includegraphics[width=0.15\textwidth]{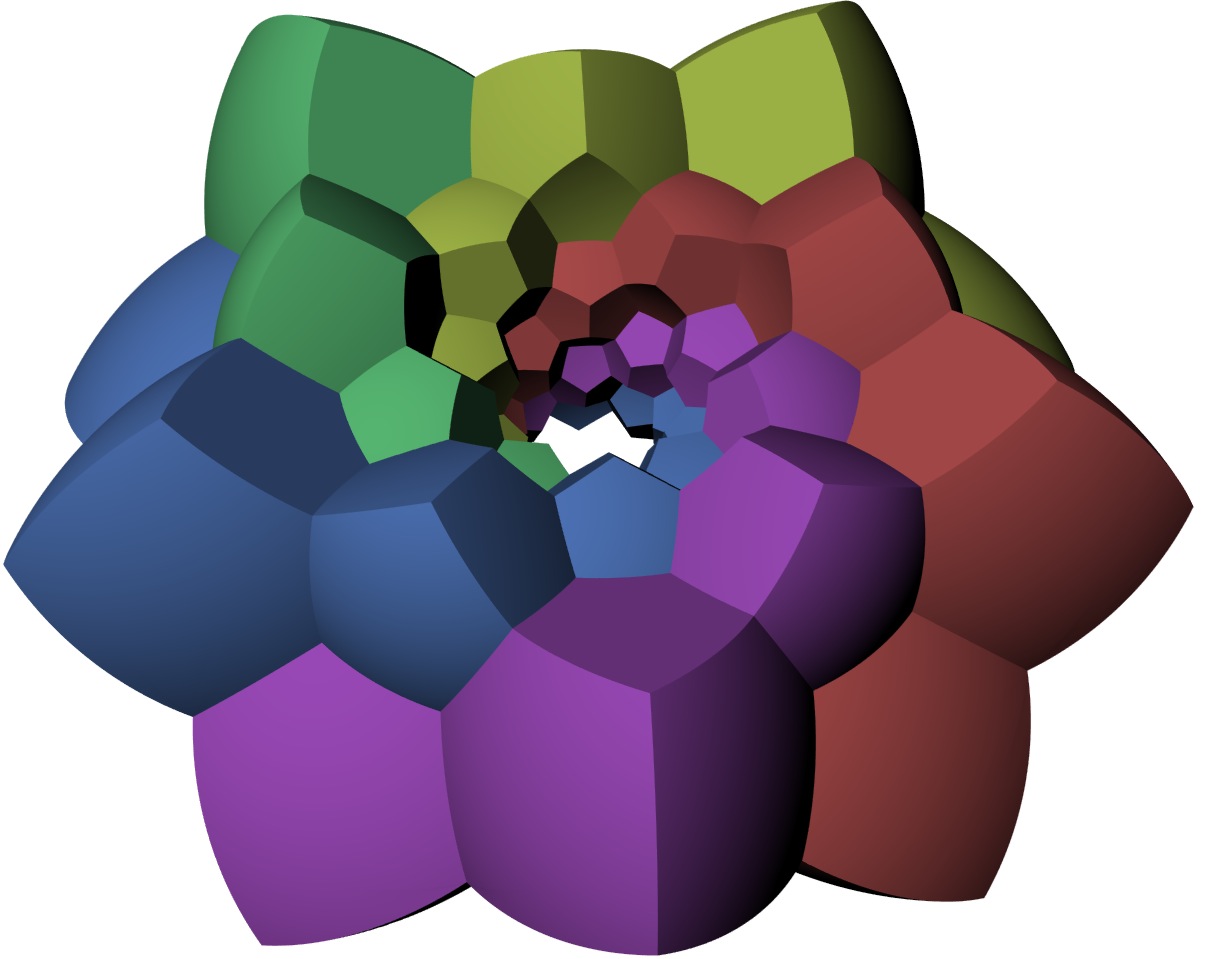}
\label{Fig:SolidDodecRings5}
}


\caption{Rings of dodecahedra. \reffig{SolidDodecRings0} shows the
  equatorial ring. Figures~\ref{Fig:SolidDodecRings1}
  through~\ref{Fig:SolidDodecRings5} show the outer rings wrapping
  around it.}
\label{Fig:SolidDodecRings}
\end{figure}

With notation as in \refsec{Pos}, suppose that $q \in \calD^*$ is the
lift of the face rotation $A \in \calD$ of angle $2\pi/5$ about the
vector $f$.  Let $R = \subgp{q} < \calD^*$ be the resulting cyclic
group of order ten.  Note that $R$ has twelve right cosets in
$\calD^*$.  We call the cosets \emph{rings} because each corresponding
union of spherical dodecahedra forms a solid torus in $S^3$.  We give
the rings the following names: $R$ is the \emph{spinal ring}, $R^\eq$
is the \emph{equatorial ring} (having all elements at distance $\pi/2$
from the south pole), $R^\inn_0$ to $R^\inn_4$ are the \emph{inner
  rings} (each incident to the spine), and $R^\out_0$ to $R^\out_4$
are the \emph{outer rings} (each incident to the equator).  The names
are justified by the following proposition.


\begin{proposition}
\label{Prop:Ring}
The rings meet the spherical layers of $\calT_{120}$ as follows.
\begin{center}
\begin{tabular}{clcccccc}
distance & rotation type & $\#$~cells & spinal & equatorial & remaining & inner & outer \\
\midrule
$0$      & identity & 1  & 1 & 0  & 0  & 0 & 0 \\
$\pi/5$  & face     & 12 & 2 & 0  & 10 & 2 & 0 \\
$\pi/3$  & vertex   & 20 & 0 & 0  & 20 & 2 & 2 \\
$2\pi/5$ & face     & 12 & 2 & 0  & 10 & 0 & 2 \\
$\pi/2$  & edge     & 30 & 0 & 10 & 20 & 2 & 2 \\
$3\pi/5$ & face     & 12 & 2 & 0  & 10 & 0 & 2 \\
$2\pi/3$ & vertex   & 20 & 0 & 0  & 20 & 2 & 2 \\
$4\pi/5$ & face     & 12 & 2 & 0  & 10 & 2 & 0 \\
$\pi$    & identity & 1  & 1 & 0  & 0  & 0 & 0
\end{tabular}
\end{center}
The column titled ``remaining'' counts the number of cells left in
each layer after the spinal and equatorial rings have been removed.
\end{proposition}

\begin{proof}
Let $P$ be the pentagon of $\calT_D$ with centre $f$ and let $-P$ be
the antipodal pentagon to $P$, which exists by \reflem{Tiling}.  Let
$\calD_P < \calD$ be the stabiliser of $\pm P = P \cup -P$.  

\begin{claim*}
The stabiliser $\calD_P$ is a dihedral group of order ten: it contains
the rotations of $P$, contains five edge rotations perpendicular to
$f$, and acts dihedrally on the plane $f^\perp$.
\end{claim*}

\begin{proof}
The pentagons $\pm P$ contain ten right-handed spherical flag
triangles.  Thus $\calD_P$ contains at most ten elements.  Five of
these are the face rotations about $f$, the centre of $P$.  Note, as
shown in \reffig{DodecaPosition}, the face centre $f$ is perpendicular
to the edge centre $k$.  Thus the edge rotation about $k$ swaps $P$
and $-P$.  The four images of $k$, under the face rotation about $f$,
provide the remaining edge rotations in $\calD_P$.
\end{proof}

Let $\calD_P^*$ be the lift of $\calD_P$ to $\calD^*$.  So $\calD_P^*
\subset S^3$ is a binary dihedral group.  The spinal ring $R =
\subgp{q}$ is an index two subgroup of $\calD_P^*$.  
The equatorial ring $R^\eq$ is the unique coset of $R$ inside of
$\calD_P^*$.  By the claim immediately above, every element of $R$ is
a lift of a face rotation and every element of $R^\eq$ is a lift of an
edge rotation.  This verifies the spine and equatorial columns in the
table.

We now consider how the remaining $100$ cells are distributed among
the five outer and five inner rings.  For any great circle $C$ and any
great sphere $S \subset S^3$ the intersection $C \cap S$ is either two
antipodal points, or all of $C$.  When $S$ is round, but not great, $C
\cap S$ is zero, one, or two points.  As noted immediately after
\refeqn{Param} the elements of $R = \subgp{q}$ lie on a great circle
through the identity.  Since the right action of $S^3$ on itself is
via isometry, the right cosets $R \cdot p$ also lie on great circles.
Since these great circles are disjoint, deduce $R^\eq$ is the only one
of them contained in the equatorial sphere.  The remaining cosets meet
the equatorial sphere in either zero elements or two, antipodal,
elements.

Since there are twenty cells left in the equatorial sphere, deduce
that each inner ring, and each outer ring, contains exactly two
equatorial cells.  This finishes the row labelled $\pi/2$ in the
table.  The same counting argument applies to the rows labelled
$\pi/3$ and $2\pi/3$.  This accounts for six elements of each ring; we
must pin down the remaining four.

Recall the definition of $q'$ from the proof of \reflem{Voronoi}: the
quaternion $q'$ is the lift of a face rotation about $f'$, where $f'$
is the centre of a face $P'$ of the tiling $\calT_D$, and where $P'$
is adjacent to the face $P$.  The inner rings are the cosets $R^\inn_i
= R \cdot q' q^{-i}$, for $i = 0, 1, 2, 3, 4$.  As shown in the proof
of \reflem{Voronoi}, the real part of $(q^{-1}) \cdot q'$ is
$\cos(\pi/5)$.  Thus $R^\inn_0 = R \cdot q'$ meets the layer at
distance $\pi/5$ in exactly two elements, namely $q'$ and $(q^{-1})
\cdot q'$.

Note also that $R^\inn_i = R \cdot q' q^{-i} = q^i (R \cdot q') q^{-i}
= \phi_q^i(R^\inn_0)$.  That is, the $i^\thsup$ coset is obtained from
$R^\inn_0$ via the twisted action.  It is now an exercise to show that
all of the $R^\inn_i$ are distinct.


Note that all cosets are invariant under the antipodal map, because
$-1 \in R$.  This implies $R^\inn_0$ also meets the layer at distance
$4\pi/5$ in two points.  This accounts for all ten elements of
$R^\inn_0$.  Since $\phi_q$ fixes each spherical layer setwise, and
since $R^\inn_i = \phi_q^i(R^\inn_0)$, the inner column of the table
is verified.

There are only twenty elements of $D^*$ left to be accounted for; all
of these make angle $2\pi/5$ or angle $3\pi/5$ with the south pole.  It
follows that each outer ring (the cosets $R^\out_i$) contains two
elements from each of those layers.  This verifies the outer column of
the table.
\end{proof}

\begin{remark}
\label{Rem:Hopf}
The \emph{Hopf fibration} is the partition of $S^3$ into cosets of the
one-parameter subgroup $\{\exp(i\alpha)\}$.  After a rotation, we see
that the cosets of $R$ give a combinatorial Hopf fibration: they
divide the $120$--cell into 12 rings of ten dodecahedra each.  The
centres of the rings lie on 12 great circles of the Hopf fibration.
Note also that the quotient space of the Hopf fibration is
homeomorphic to $S^2$.  In similar fashion there is a kind of
combinatorial map from the $120$--cell to the dodecahedron. 
\end{remark}

\section{Rings to ribs}
\label{Sec:Ribs}

We describe the ribs of Quintessence: a collection of puzzle pieces,
in $\mathbb{R}^3$, that combined in various ways to produce burr
puzzles.  The puzzle pieces are based on the rings of spherical
dodecahedra described in \refsec{Rings}.  We use stereographic
projection, $\rho$, defined in \refsec{StereoProj}, to move the pieces
into $\RR^3$ where we can 3D print the resulting ribs.

Following the notation of \refsec{StereoProj} we have
\[
\frac{d\rho}{d\alpha} = \frac{1}{1 + \cos(\alpha)}\cdot u.
\]
In particular, if $e^{u\alpha}$ is near the south pole then $\alpha$
is close to zero and stereographic projection shrinks objects by a
factor of approximately two.  If $e^{u\alpha}$ is near the equatorial
sphere then $\alpha$ is close to $\pi/2$.  In this case stereographic
projection leaves sizes essentially unchanged.  However, if
$e^{u\alpha}$ approaches the north pole then $\alpha$ approaches $\pi$
and sizes blow up.  Thus a dodecahedron of the $120$--cell close to
the south pole shrinks slightly, and a dodecahedron close to the north
pole becomes much larger.

All of the calculations so far have been dimensionless.  When we wish
to 3D print a rib, we have to choose a scale $\lambda$, say in
millimetres or inches, corresponding to a unit distance in the image
of $\rho$.  Many considerations need to be taken into account in
choosing $\lambda$; the scale is sensitive to the design of the ribs.
However, two issues are clear: a large feature on a rib causes the
cost to grow with the cube of $\lambda$ while a very small feature may
be too fragile or may fall below the resolution of the printer.

These two issues are in tension, and lead to the general principle
that features that are identical in $S^3$ should have sizes in bounded
ratio in $\RR^3$, after projection.  In this particular case, the
features of the ribs are the dodecahedra.  The principle tells us that
we should not be using dodecahedra that are too close to the north
pole.  However, the ratio of two between sizes near the equator and
near the south pole is acceptable.

\begin{figure}[htbp]
\centering 
\hspace{-1cm}
\subfloat[Spine.]
{
\includegraphics[width=0.2\textwidth]{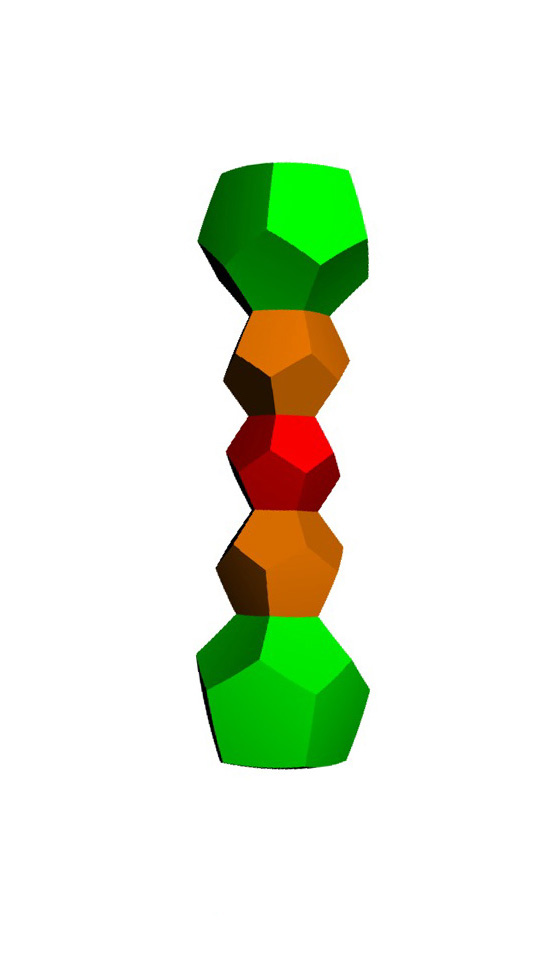}
\label{Fig:Spine}
} 
\hspace{-0.5cm}
\subfloat[Inner six.]
{
\includegraphics[width=0.2\textwidth]{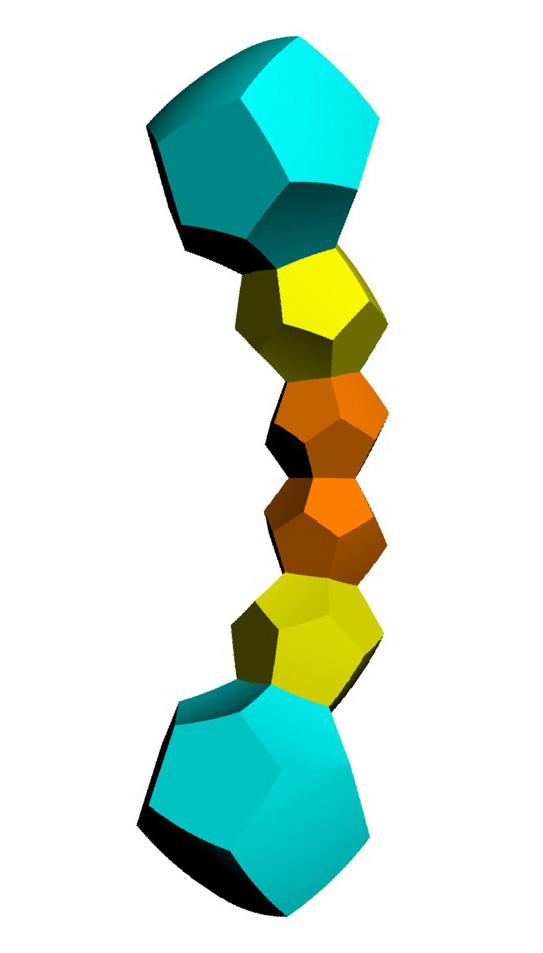}
\label{Fig:Inner}
}
\hspace{-0.2cm}
\subfloat[outer six.]
{
\includegraphics[width=0.2\textwidth]{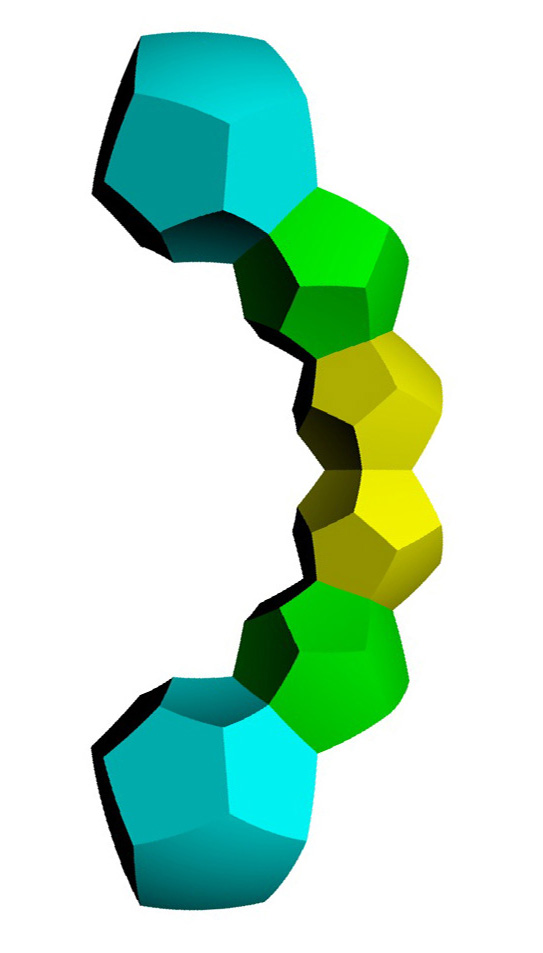}
\label{Fig:Outer}
}
\subfloat[Equator.]
{
\includegraphics[width=0.2\textwidth]{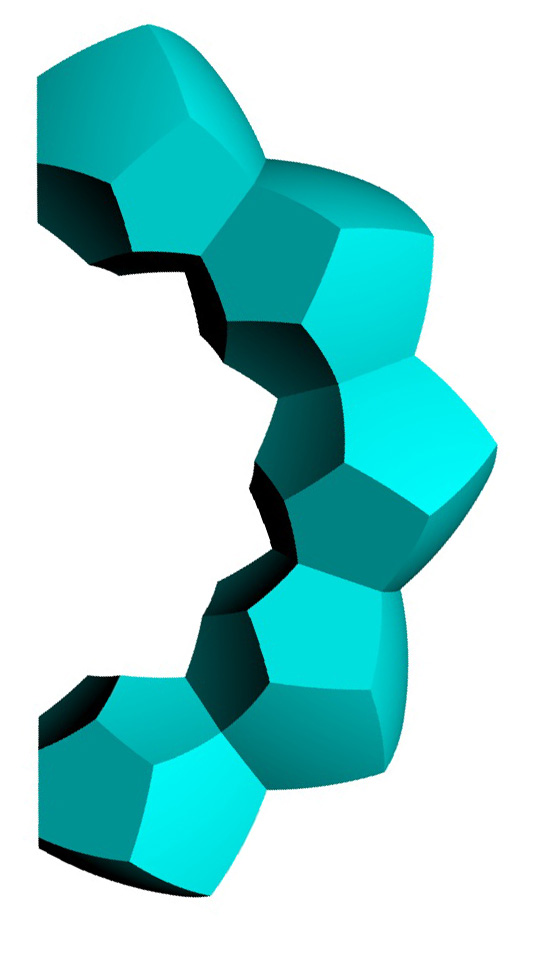}
\label{Fig:Equator}
}
\caption{The colouring of the cells is by layer and is consistent
  with \reffig{120_layers}.  We obtain the inner four and outer four
  ribs by deleting the equatorial cells.}
\label{Fig:CellsInRibs}
\end{figure}

So, we remove from our rings any dodecahedra that lie strictly in the
northern hemisphere, giving us the \emph{spine}, the \emph{inner six}
ribs and the \emph{outer six} ribs.  Experimentation shows that many
interesting constructions require even shorter ribs; hence we also
make the \emph{inner four} ribs and the \emph{outer four} ribs.  These
are the result of removing the two equatorial dodecahedra from the
inner six and outer six.  The equatorial ring can be printed as is,
but again experimentation shows that more puzzles are possible if we
break the equatorial ring into two ribs of five dodecahedra each.  See
\reffig{Ribs} as well as \reffig{CellsInRibs}.

\begin{figure}[htbp]
\centering 
\hspace{-1cm}
\subfloat[]
{
\includegraphics[width=0.15\textwidth]{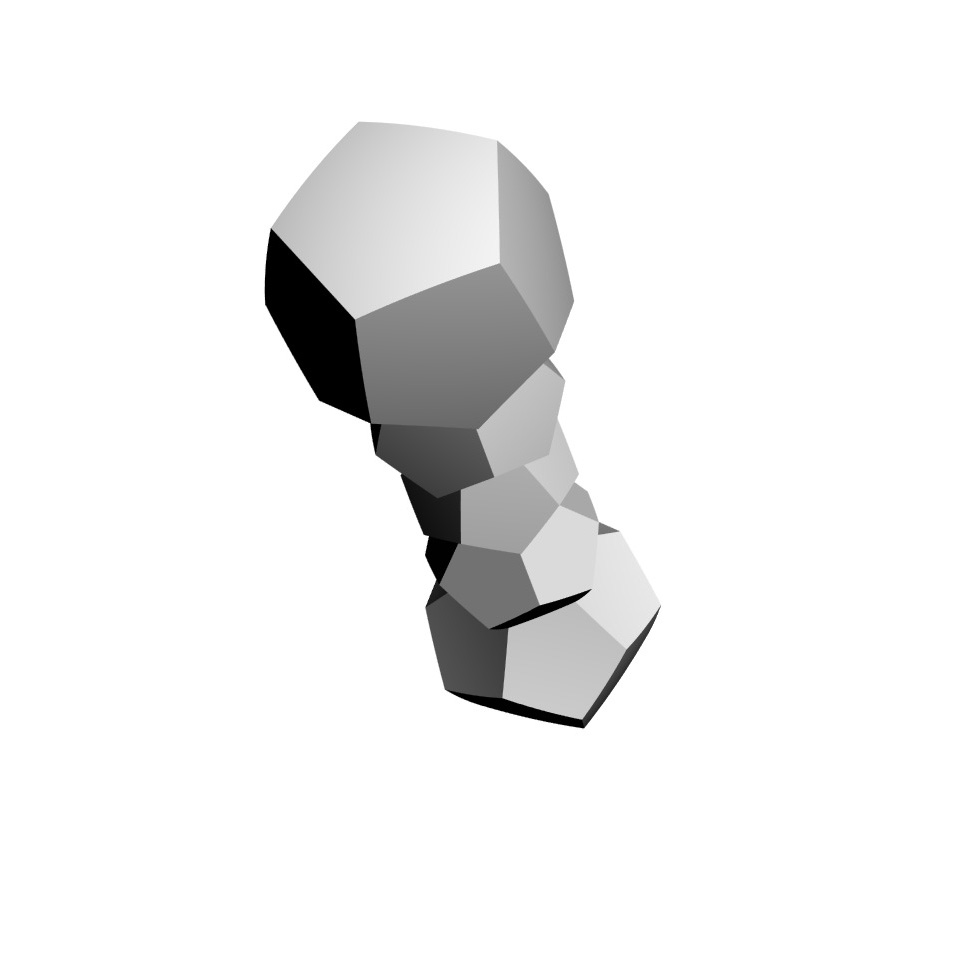}
\label{Fig:Dc45Meteor0}
} 
\hspace{-0.5cm}
\subfloat[]
{
\includegraphics[width=0.15\textwidth]{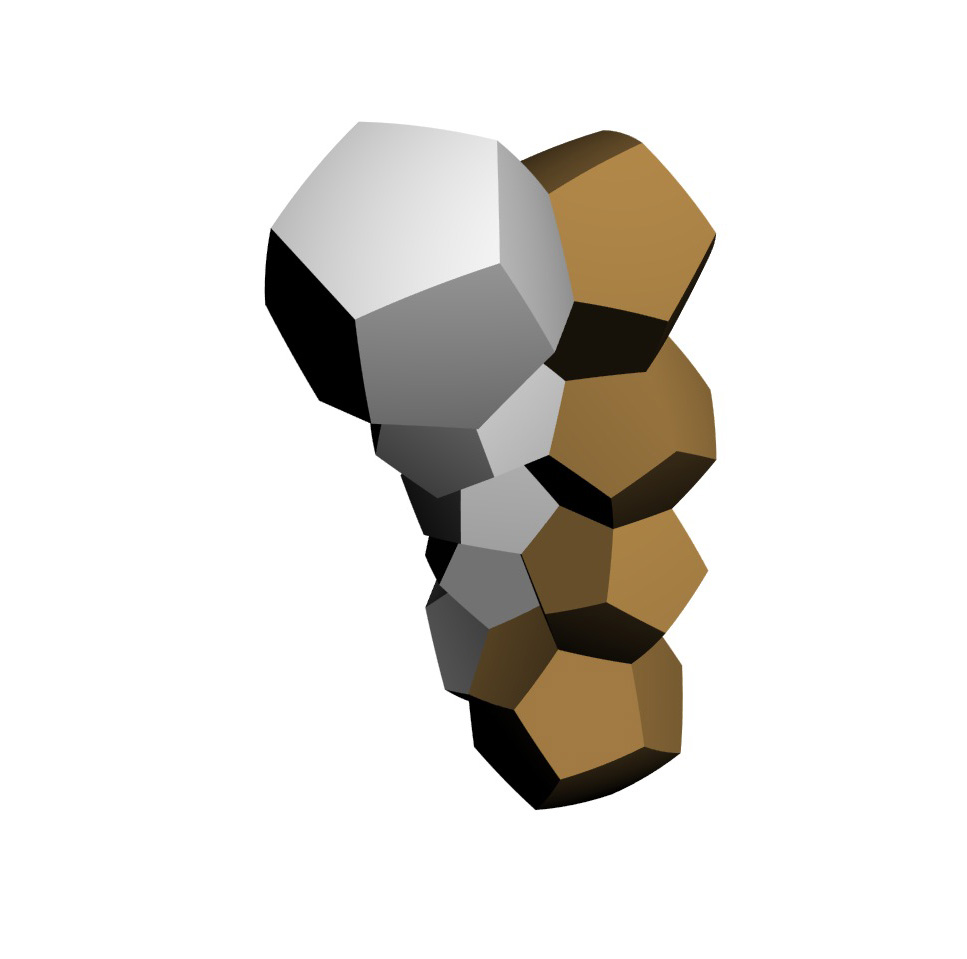}
\label{Fig:Dc45Meteor1}
} 
\hspace{-0.5cm}
\subfloat[]
{
\includegraphics[width=0.15\textwidth]{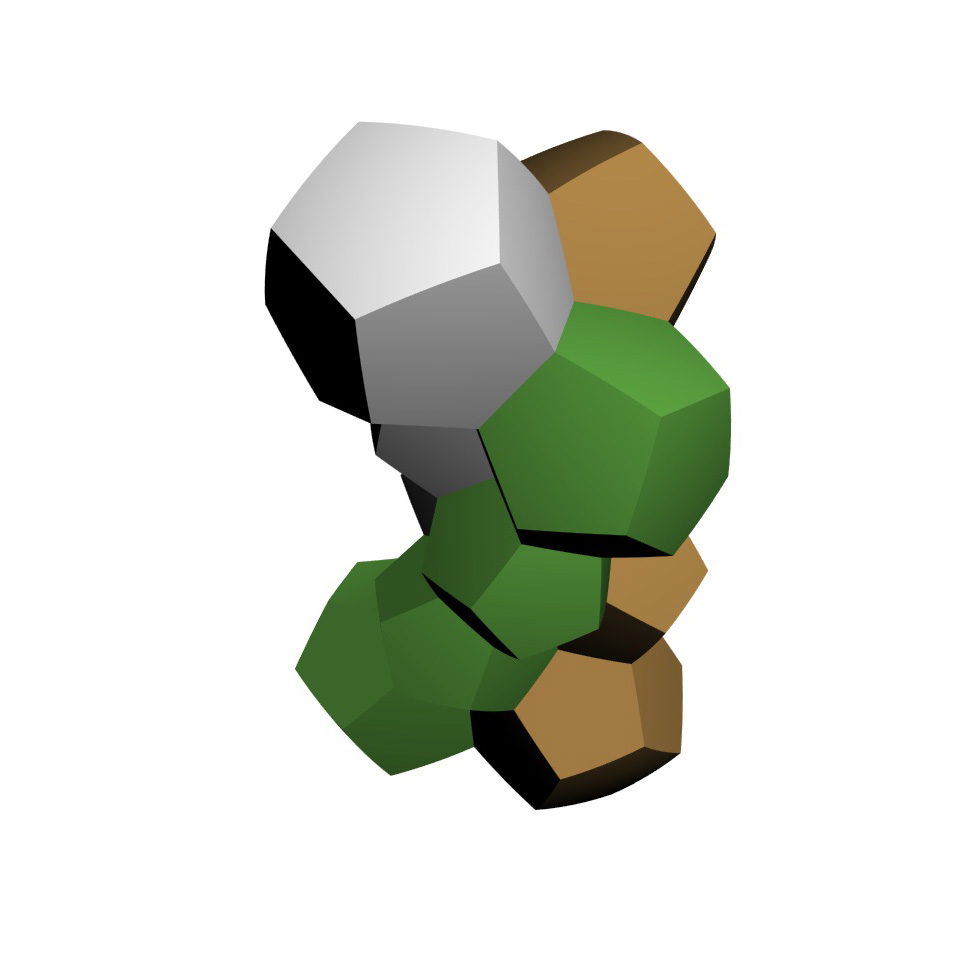}
\label{Fig:Dc45Meteor2}
}
\hspace{-0.2cm}
\subfloat[]
{
\includegraphics[width=0.15\textwidth]{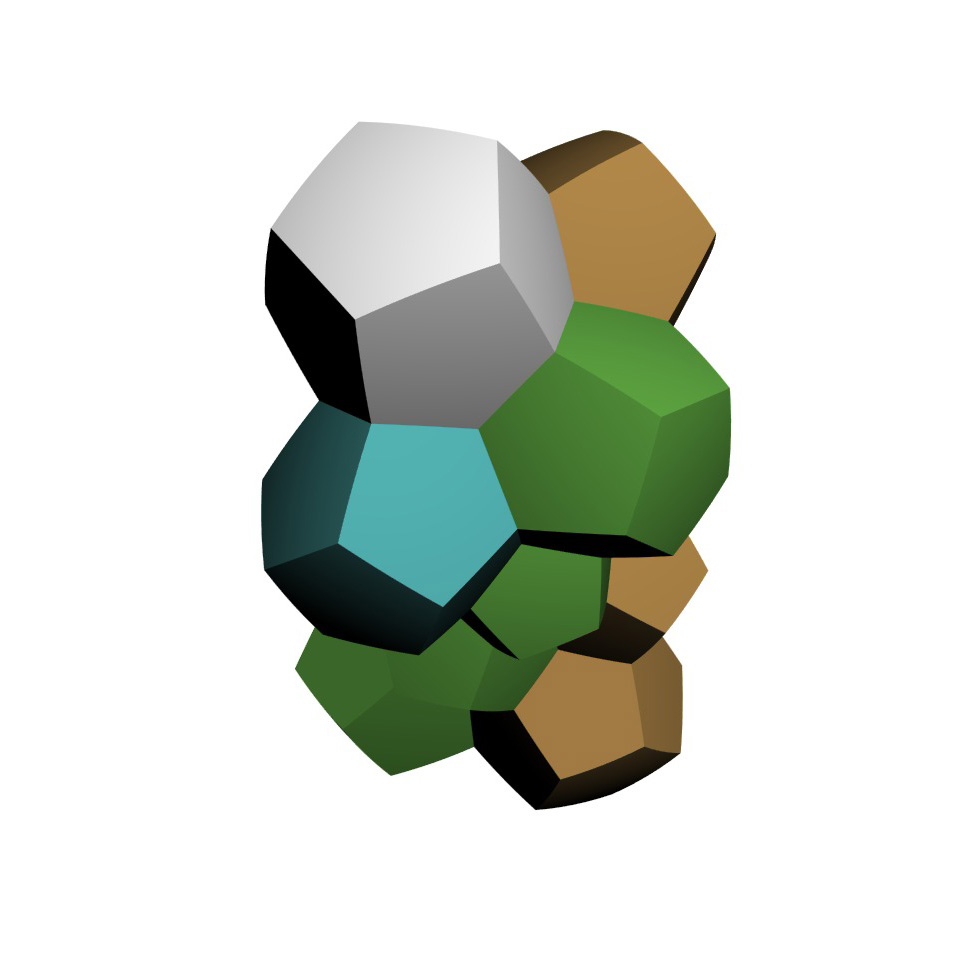}
\label{Fig:Dc45Meteor3}
}
\subfloat[]
{
\includegraphics[width=0.15\textwidth]{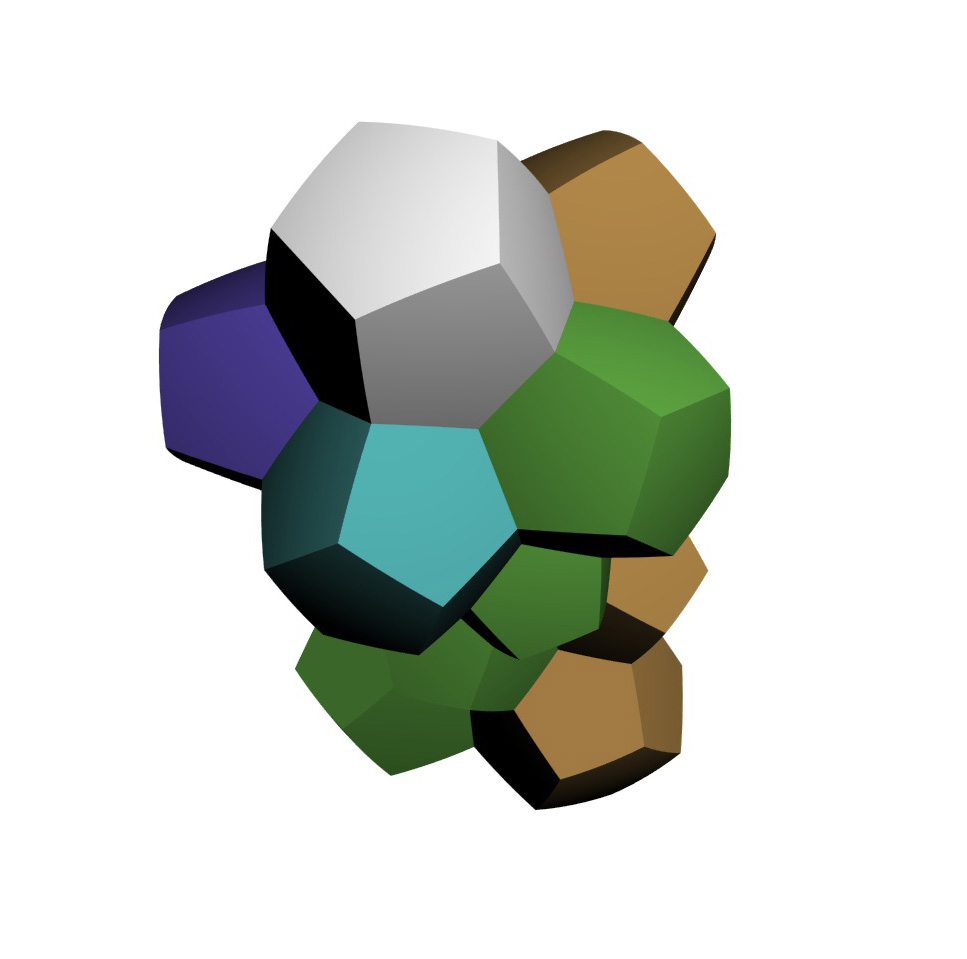}
\label{Fig:Dc45Meteor4}
}
\hspace{0.2cm}
\subfloat[]
{
\includegraphics[width=0.15\textwidth]{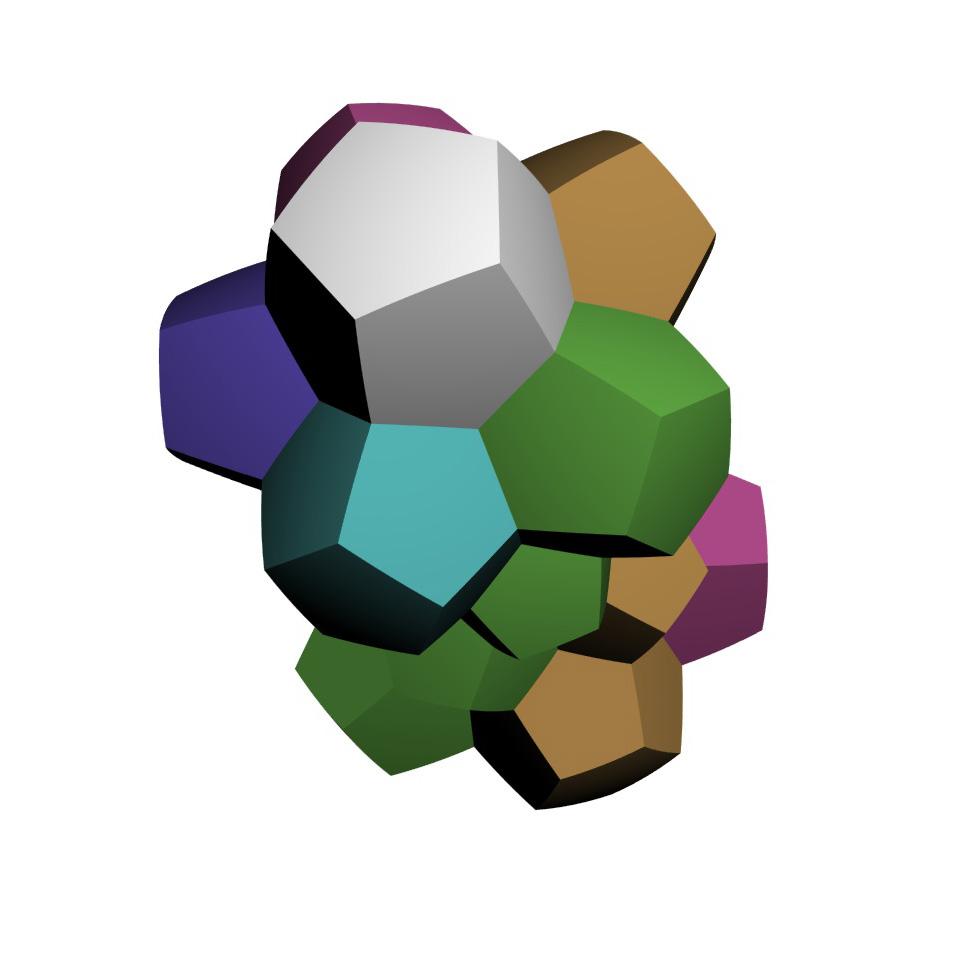}
\label{Fig:Dc45Meteor5}
}
\hspace{-0.5cm}
\subfloat[]
{
\includegraphics[width=0.15\textwidth]{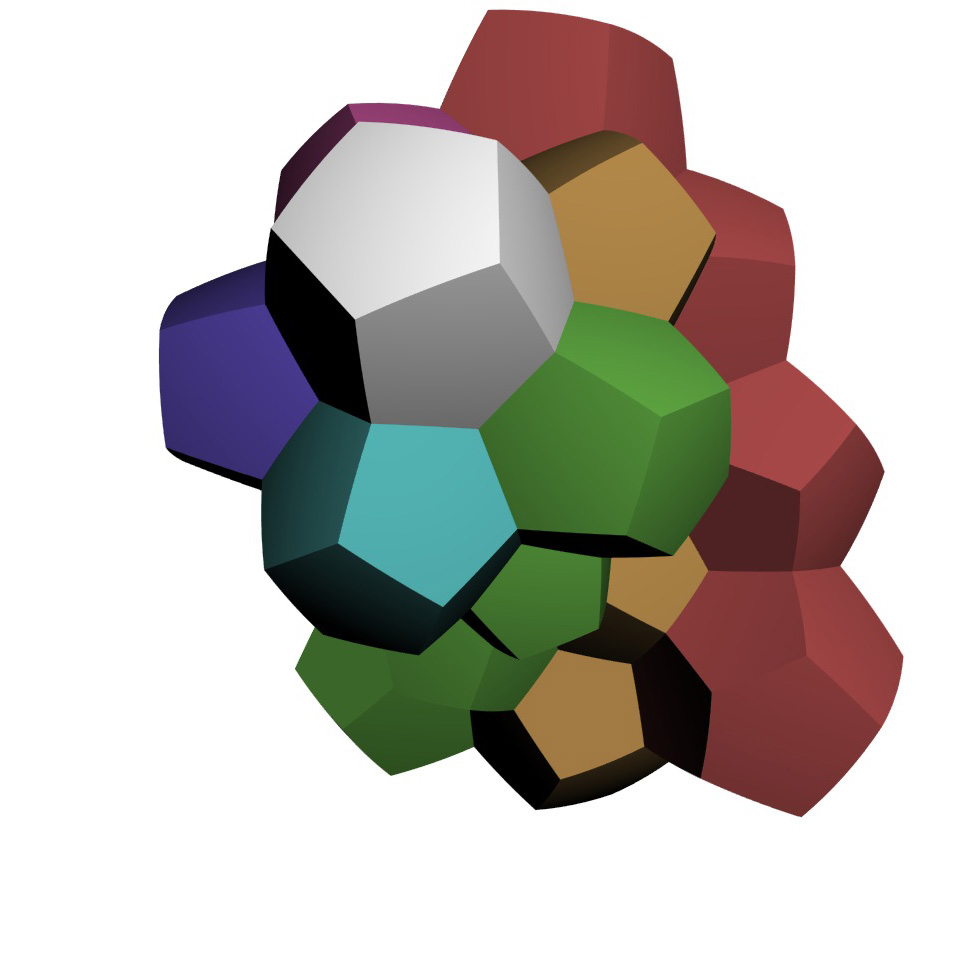}
\label{Fig:Dc45Meteor6}
} 
\hspace{-0.5cm}
\subfloat[]
{
\includegraphics[width=0.15\textwidth]{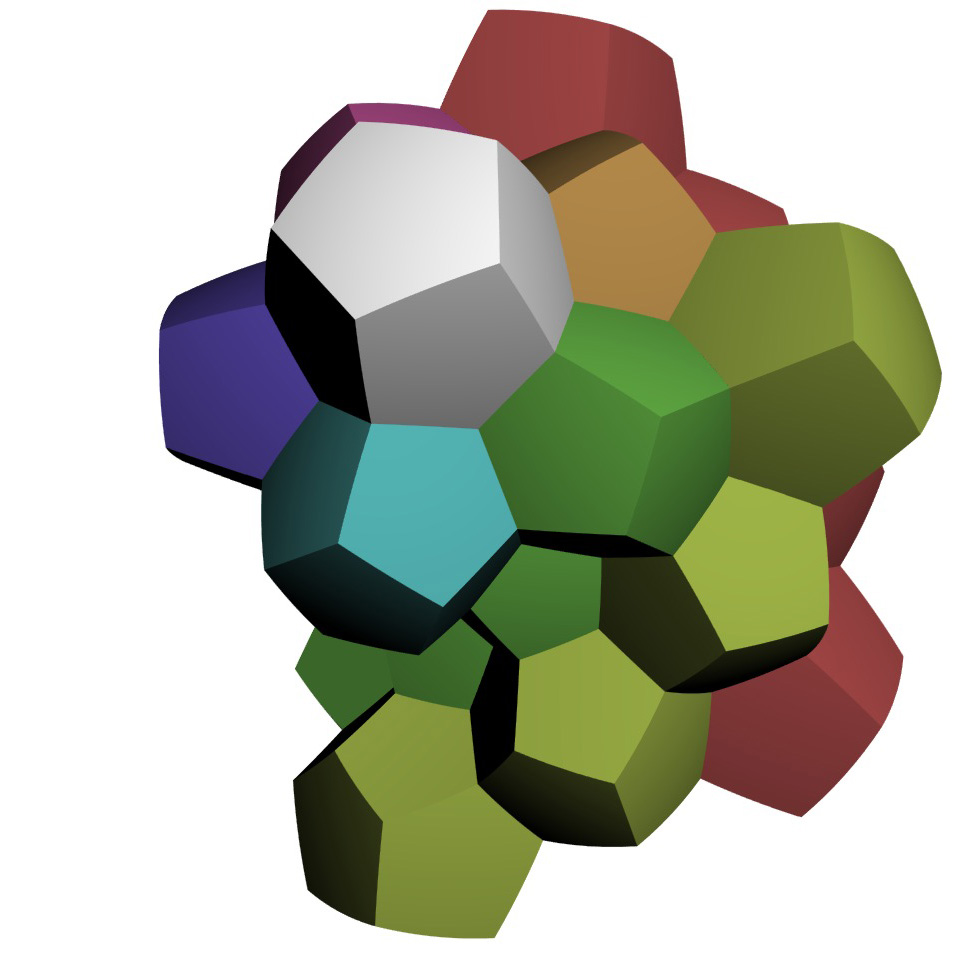}
\label{Fig:Dc45Meteor7}
}
\hspace{-0.2cm}
\subfloat[]
{
\includegraphics[width=0.15\textwidth]{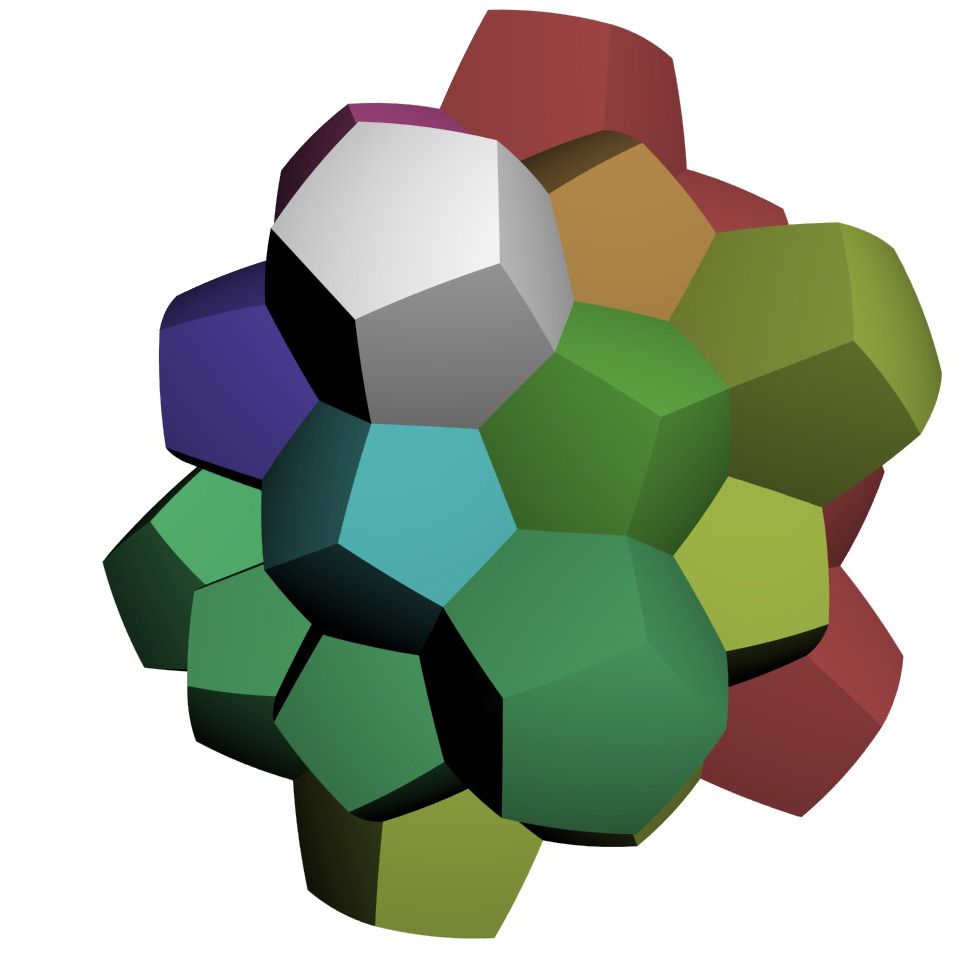}
\label{Fig:Dc45Meteor8}
}
\subfloat[]
{
\includegraphics[width=0.15\textwidth]{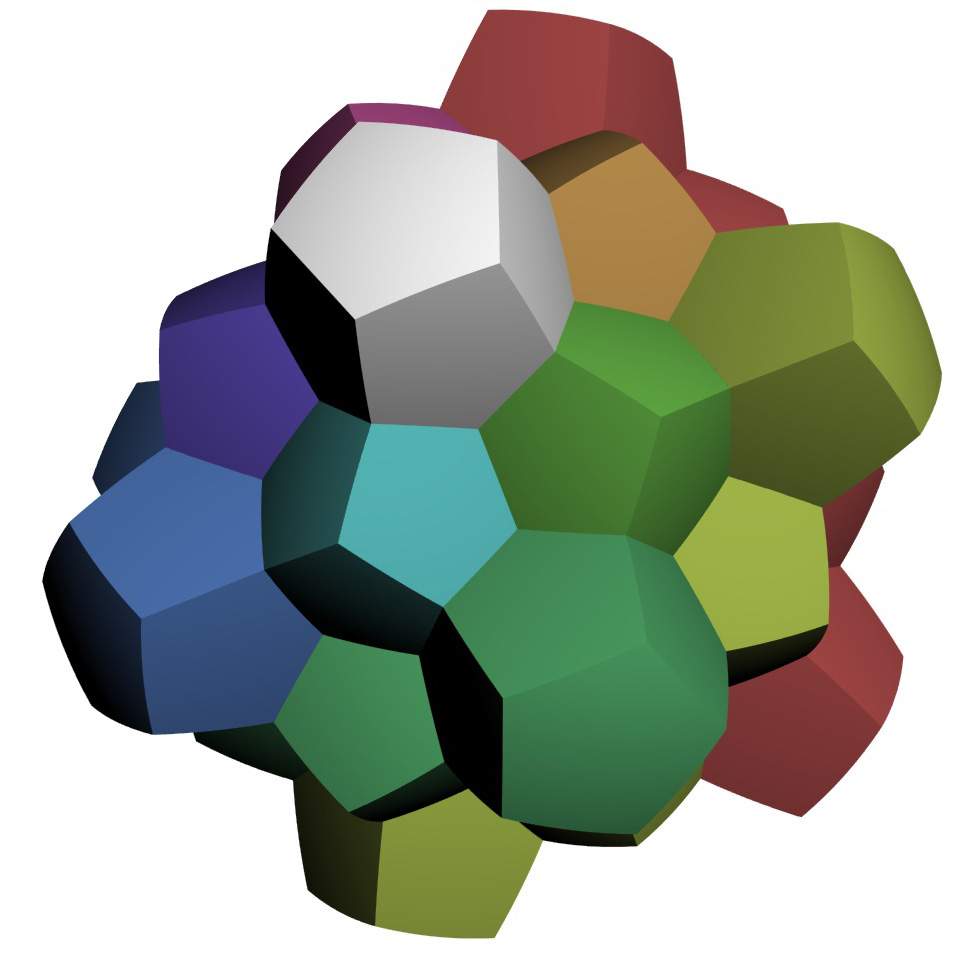}
\label{Fig:Dc45Meteor9}
}
\hspace{0.2cm}
\subfloat[]
{
\includegraphics[width=0.15\textwidth]{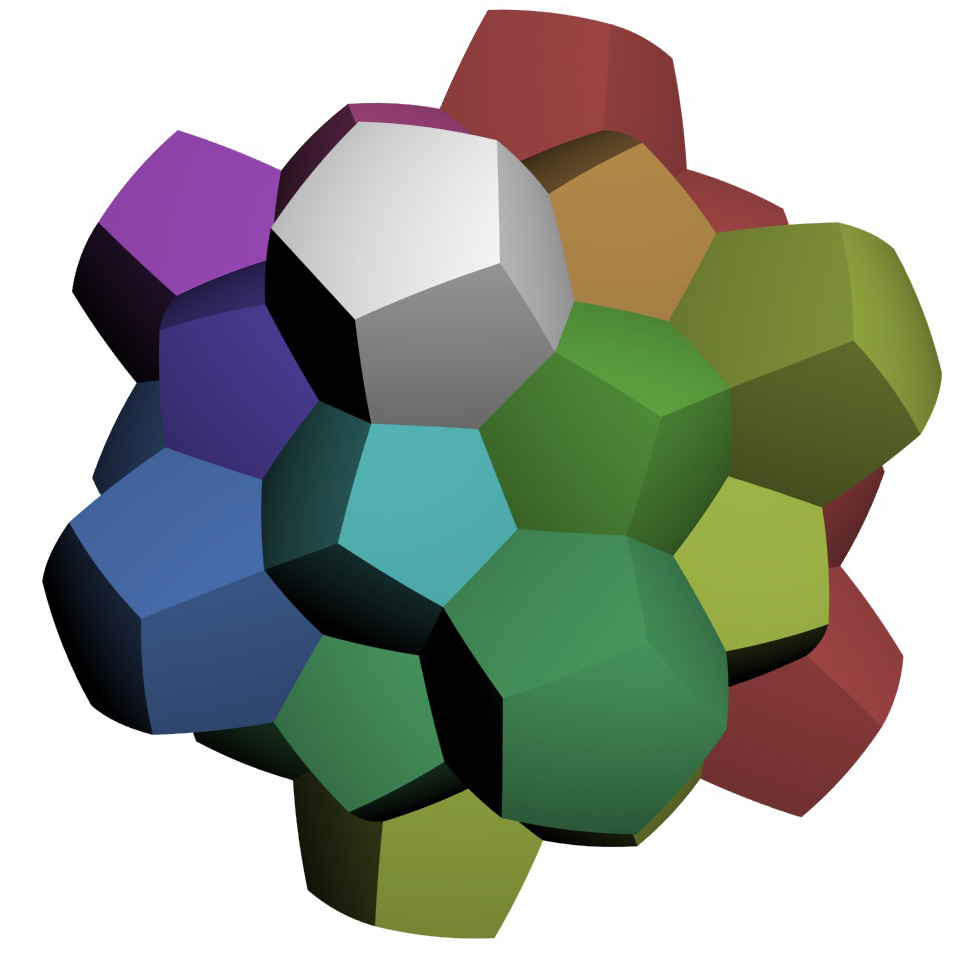}
\label{Fig:Dc45Meteor10}
}
\caption{Building the Dc45 Meteor: start with just the spine, in
  \reffig{Dc45Meteor0}.  One at a time add five copies of the inner four
  rib in Figures~\ref{Fig:Dc45Meteor1} through~\ref{Fig:Dc45Meteor5}.
  Then add five copies of the outer four rib, as in
  Figures~\ref{Fig:Dc45Meteor6} through~\ref{Fig:Dc45Meteor10}.}
\label{Fig:Dc45Meteor}
\end{figure}

With the spine and short ribs in hand, we can build, in $\RR^3$, the
stereographic projection of (almost) one-half of the $120$--cell.  We
call the resulting puzzle the \emph{Dc45 Meteor}; its construction is
shown in \reffig{Dc45Meteor}.  The spine and ribs are arranged
according to the combinatorial Hopf fibration (\refrem{Hopf}).  Since
all dodecahedra near the south pole are retained, and all dodecahedra
near the north pole are discarded, the result looks very much like
\reffig{120}: one-half of the $120$-cell.

It is not at all obvious that the puzzle can be constructed in
Euclidean space using physical objects.  However, when printed in
plastic the Meteor \emph{is} possible to assemble.  Also, when
complete it holds together with no other support.  For photos see Dc45
Meteor in \refapp{Catalog}.  Apparently a small amount of flex in the
ribs is necessary; we have not been able to solve a similar puzzle
(the Dc30 Ring) when printed in a steel/bronze composite.

It came as a surprise to us that there are numerous other burr puzzles
based on these ribs; most are not based on the combinatorial Hopf
fibration.  We list many of our discoveries in \refapp{Catalog}.  In
the remainder of this section we derive a combinatorial restriction on
the ribs that can be used in any burr puzzle.  This theorem is sharp,
as shown by the examples in \refapp{Catalog}.

\begin{theorem}
\mbox{}
\begin{enumerate}
\item At most six inner ribs are used in any puzzle. \label{6 inner ribs}
\item At most six outer ribs are used in any puzzle. \label{6 outer ribs}
\item At most ten inner and outer ribs are used in any puzzle. \label{10 ribs}
\end{enumerate}
\end{theorem}

\begin{proof}
The stereographic projection map $\rho$ is equivariant: $\rho$
transports the twisted action on $S^3$ to the $\SO(3)$ action on
$\RR^3$.  That is, $\rho$ respects the $S^2$ symmetry about the
identity in $S^3$.  Thus any two cells in a given layer (at fixed
distance from the south pole) are congruent in $\RR^3$, after
projection.  Also, any pair of cells in different layers are
different, due to the growth of $d\rho/d\alpha$.

Examining the table in \refprop{Ring}, we see that the each inner rib
contains exactly two cells adjacent to the south pole.  By the table
in \refsec{Layers}, there are exactly 12 such cells.  Part (\ref{6
  inner ribs}) follows.  We prove part (\ref{6 outer ribs}) by
examining the layer at distance $2\pi/5$ and we prove part (\ref{10
  ribs}) using the layer at distance $\pi/3$.  A colour-coded guide is
provided in \reffig{CellsInRibs}.
\end{proof}

\section{Leonardo da Vinci's polytopes}
\label{Sec:Leonardo}


\begin{wrapfigure}[25]{r}{0.45\textwidth}
\vspace{3pt}
\centering 
\includegraphics[width=0.40\textwidth]{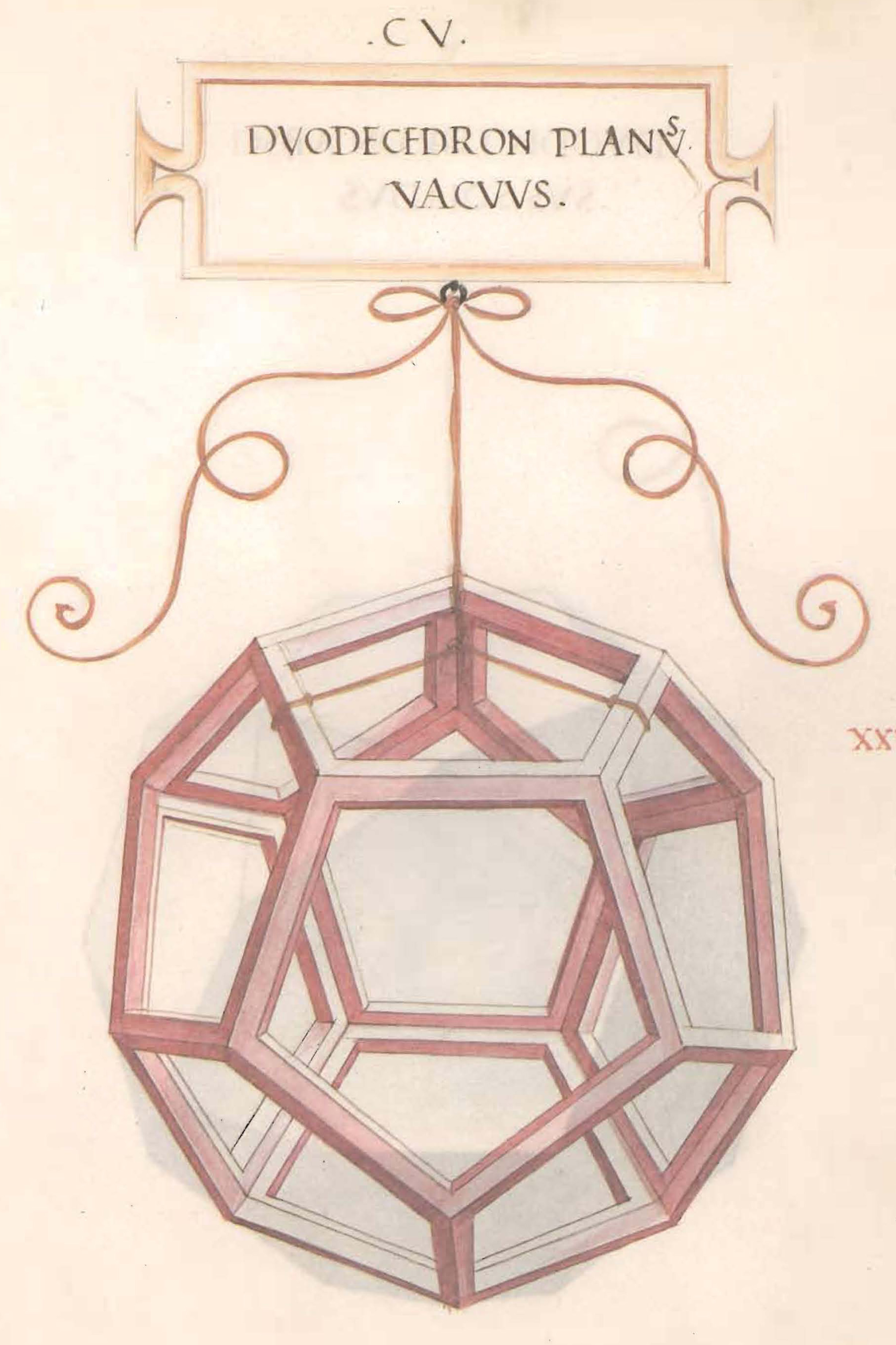}
\caption{The dodecahedron, as drawn by Leonardo da
  Vinci~\cite[Folio CV recto]{Leonardo09}.}
\label{Fig:Leonardo}
\end{wrapfigure}

If we use injection moulding to make the ribs, then the simplest route
would be to realise each rib as a union of solid dodecahedra.
However, since we are 3D printing the ribs, we are able to reduce
costs by hollowing out the dodecahedra. Our design is closely related
to Leonardo da Vinci's technique for drawing polytopes. See
\reffig{Leonardo}.

Da Vinci's design retains all of the symmetry of the dodecahedron
itself.  Since the dodecahedron is regular, we need only determine the
design inside of a single flag tetrahedron.  Then the symmetries of
the dodecahedron copy this geometry to all other flag tetrahedra,
recreating the entire design.  We do something very similar, by
constructing our design inside of a spherical flag polytope of the spherical
120-cell, $\calT_{120}$.

We have two versions of the design in the flag tetrahedron for
$\calT_{120}$, depending on whether or not the flag meets a boundary
pentagonal face of the rib, or meets an internal pentagon between two
adjacent dodecahedra.  See Figures \ref{Fig:PolytopeDesign} and
\ref{Fig:FaceDesign}.  In the former case, we add a surface in the
pentagonal face to separate the inside of the rib from the outside.
This is not necessary in the latter case.  The ``outer'' parts of the
geometry of the ribs are identical (in $S^3$) for all dodecahedra in
our ribs.  For reasons of cost and strength, we slightly thicken the
internal geometry of the smaller dodecahedra closer to the south pole,
and thin that of those further from the south pole.  Note that
\reffig{120} is modelled similarly, using only the internal design.

\begin{figure}[htbp]
\centering 
\hspace{-1cm}
\subfloat[Geometry for an external face of a puzzle piece within the
 spherical flag tetrahedron.]
{
\includegraphics[width=0.43\textwidth]{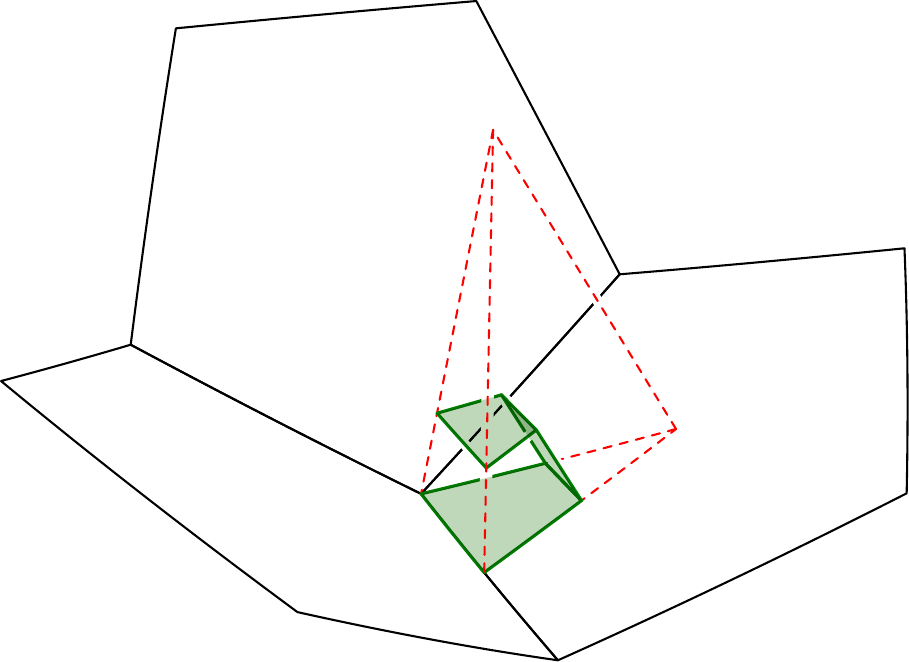}
\label{Fig:FlagPolytopeExternal}
} 
\qquad
\subfloat[Geometry for an internal face of a puzzle piece within the
  spherical flag tetrahedron.]
{
\includegraphics[width=0.43\textwidth]{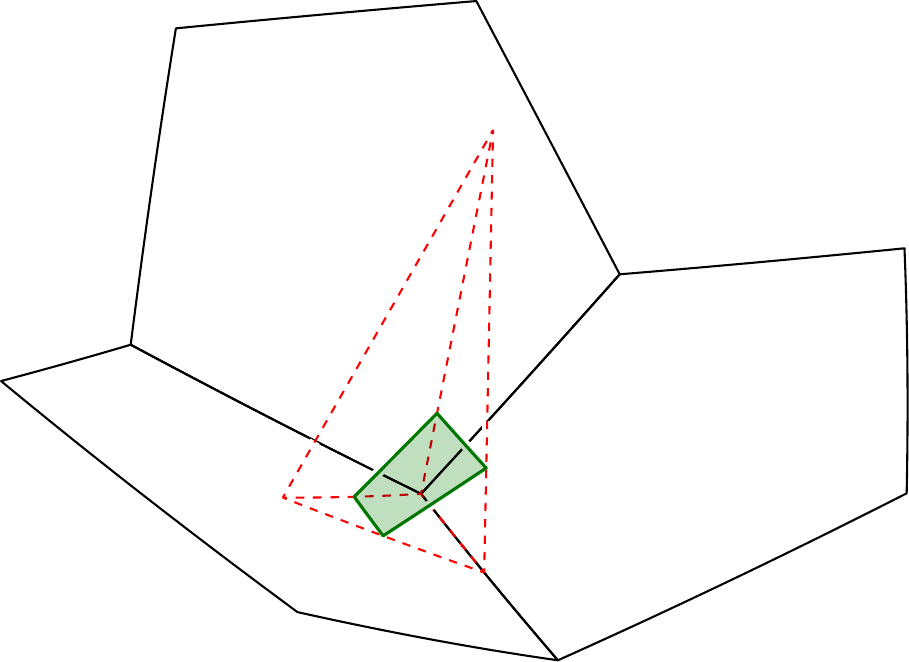}
\label{Fig:FlagPolytopeInternal}
}
\caption{The two versions of the spherical flag polytope design.  Here this is the tetrahedron drawn with a dashed line.  We show only
  three faces of the central dodecahedron of the stereographic
  projection to $\RR^3$.}
\label{Fig:PolytopeDesign}
\end{figure}

\begin{figure}[htbp]
\centering 
\hspace{-1cm}
\subfloat[Twenty copies of the external design, forming two faces of a
  dodecahedron drawn in the Da Vinci style.]
{
\includegraphics[width=0.43\textwidth]{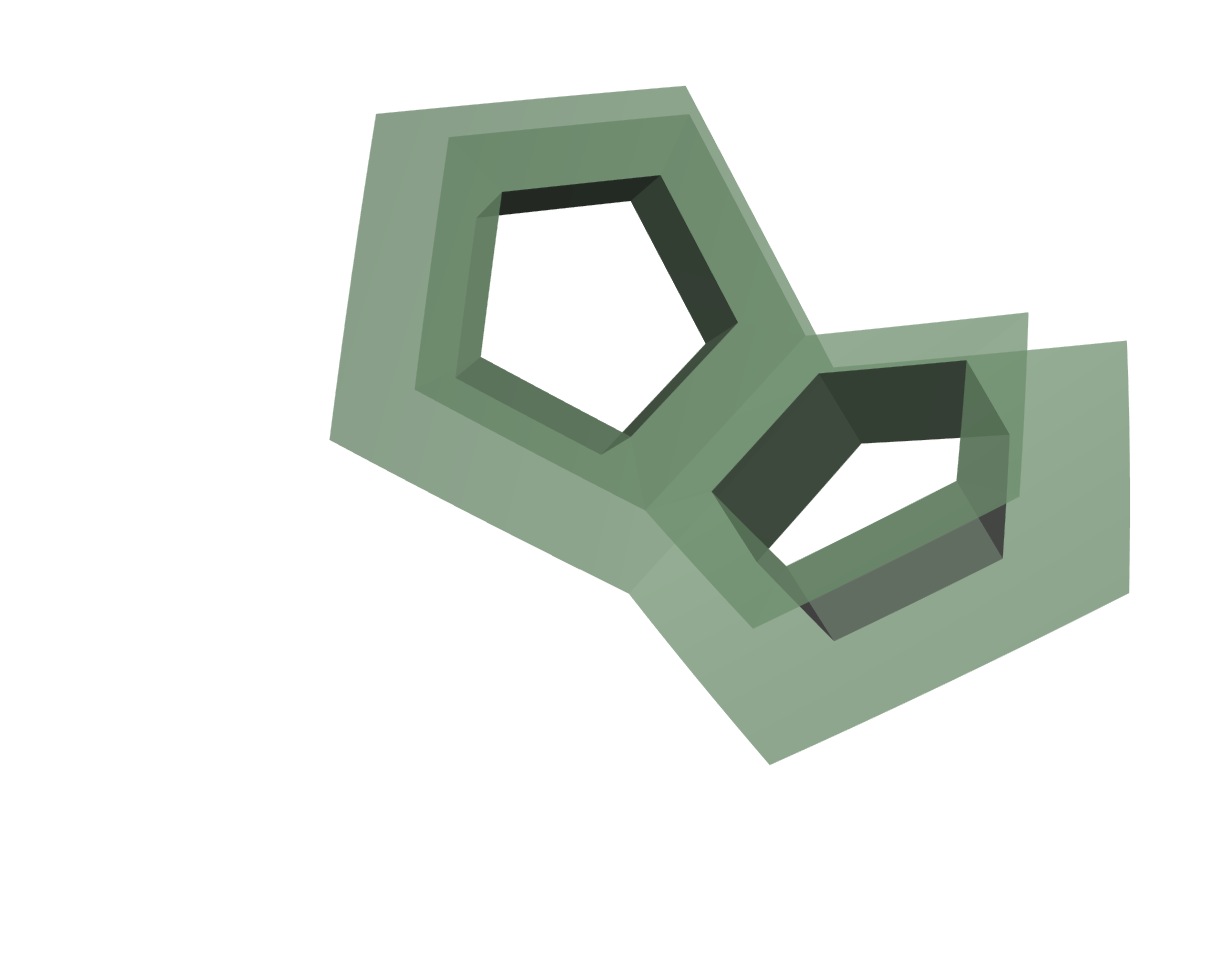}
\label{Fig:CopiesFlagPolytopeExternal}
} 
\qquad
\subfloat[Twenty copies of the external design and twenty copies of
  the internal design, forming two faces of two adjacent dodecahedra,
  and the face between those dodecahedra, drawn in the Da Vinci
  style.]
{
\includegraphics[width=0.43\textwidth]{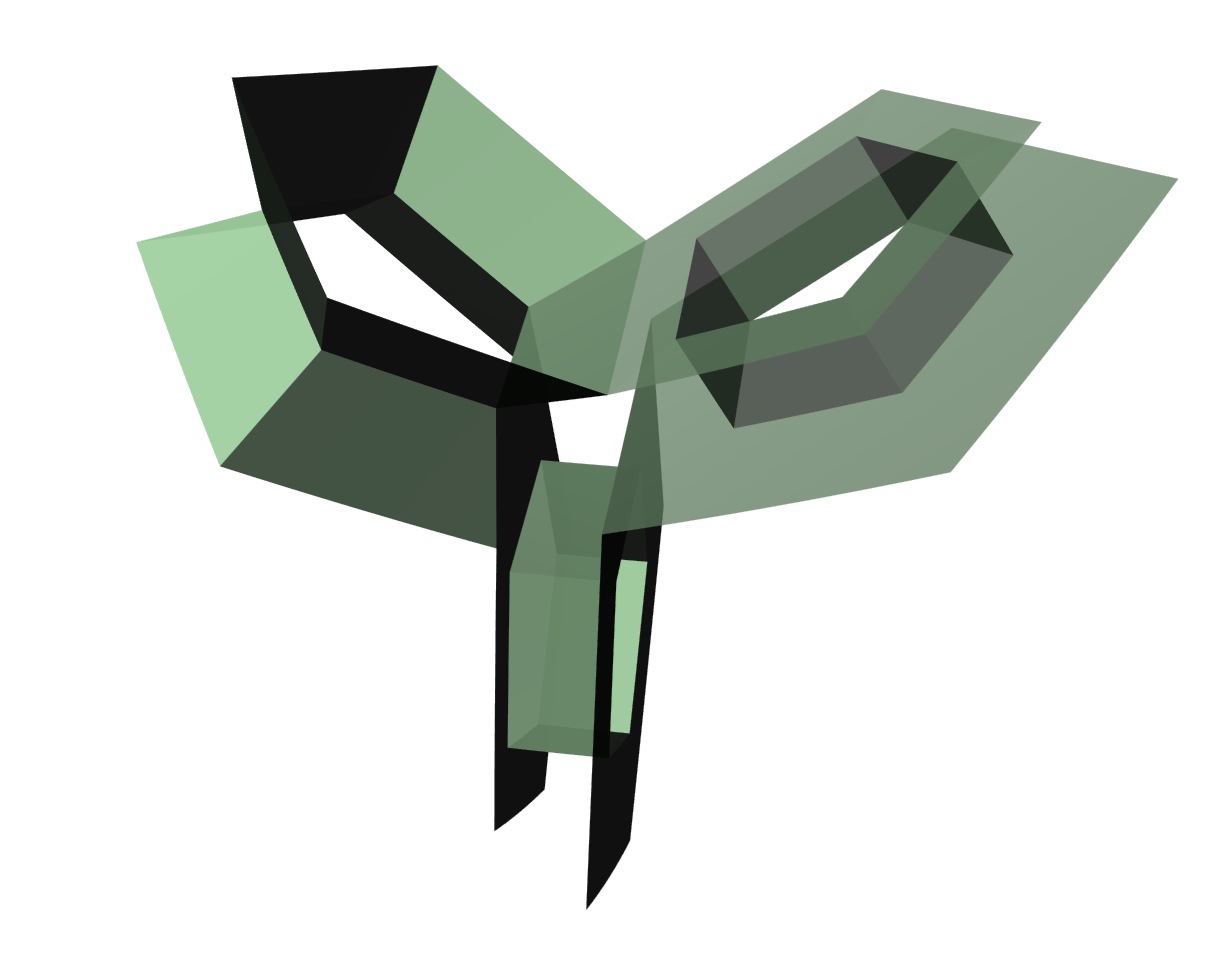}
\label{Fig:CopiesFlagPolytopeInternal}
}
\caption{Two examples of how the external and internal face designs
  fit together to form the geometry of the rib puzzle pieces.}
\label{Fig:FaceDesign}
\end{figure}

\bibliographystyle{hyperplain} 
\bibliography{bibfile}

\vfill

\appendix
\section{Catalog}
\label{App:Catalog}

When trying to build a puzzle out of the ribs, there is a spectrum of
possibilities.  At one end there are constructions that hold together
so loosely that a small tap causes them to fall apart.  At the other
end there are puzzles that hold together so tightly that there seems
to be no way to assemble them without applying large amounts of force.
Below we list those puzzles, avoiding both ends of this spectrum, that
we find visually pleasing.  Please let us know of any others you find!

\begin{remark*}
The designation Dc$N$ at the beginning of each puzzle stands for
``dodecahedral cell-centred'' and $N$ counts the number of cells.
Using other polytopes, such as the $600$--cell, would lead to puzzles
with different unit cells, such as tetrahedra.  Using other projection
points would lead to, say, vertex-centred puzzles.
\end{remark*}


\newpage
\begin{tabular}{lc}
Dc24 Star & \multirow{6}{*}{
\includegraphics[height=82pt]{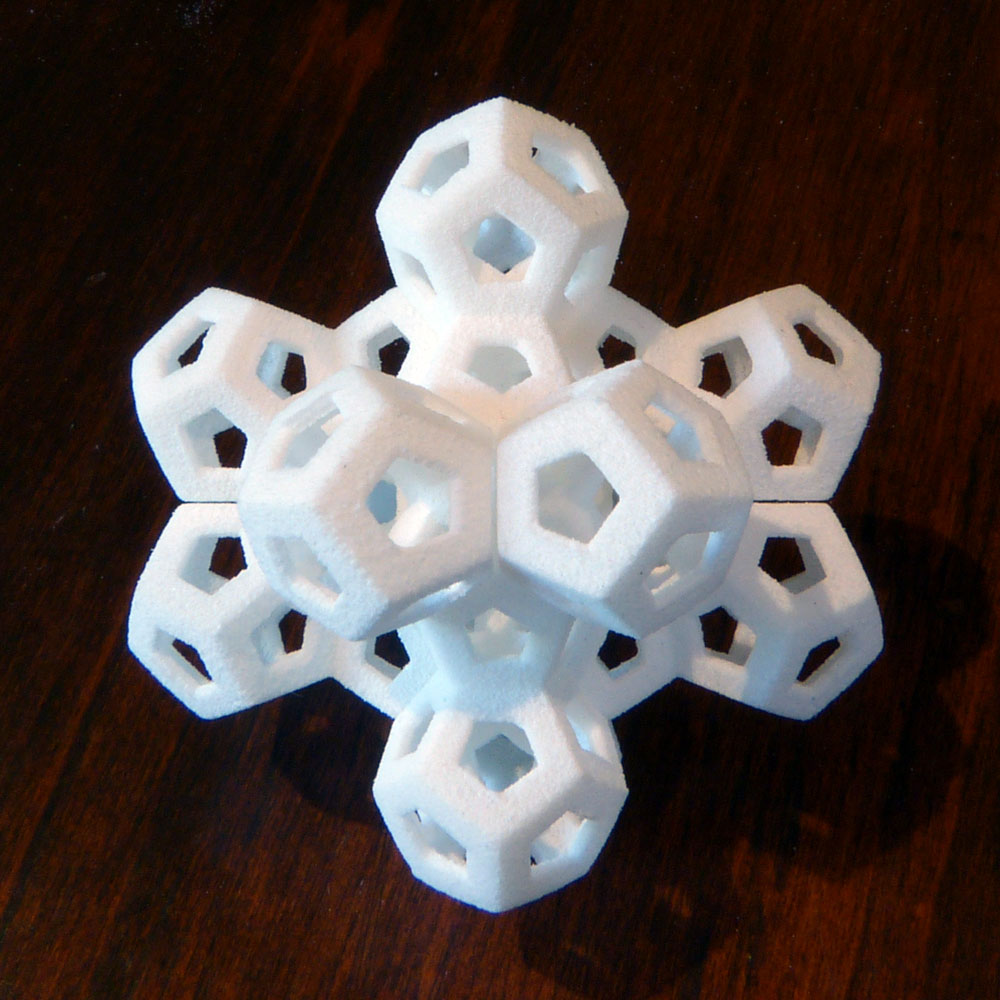}
\includegraphics[height=82pt]{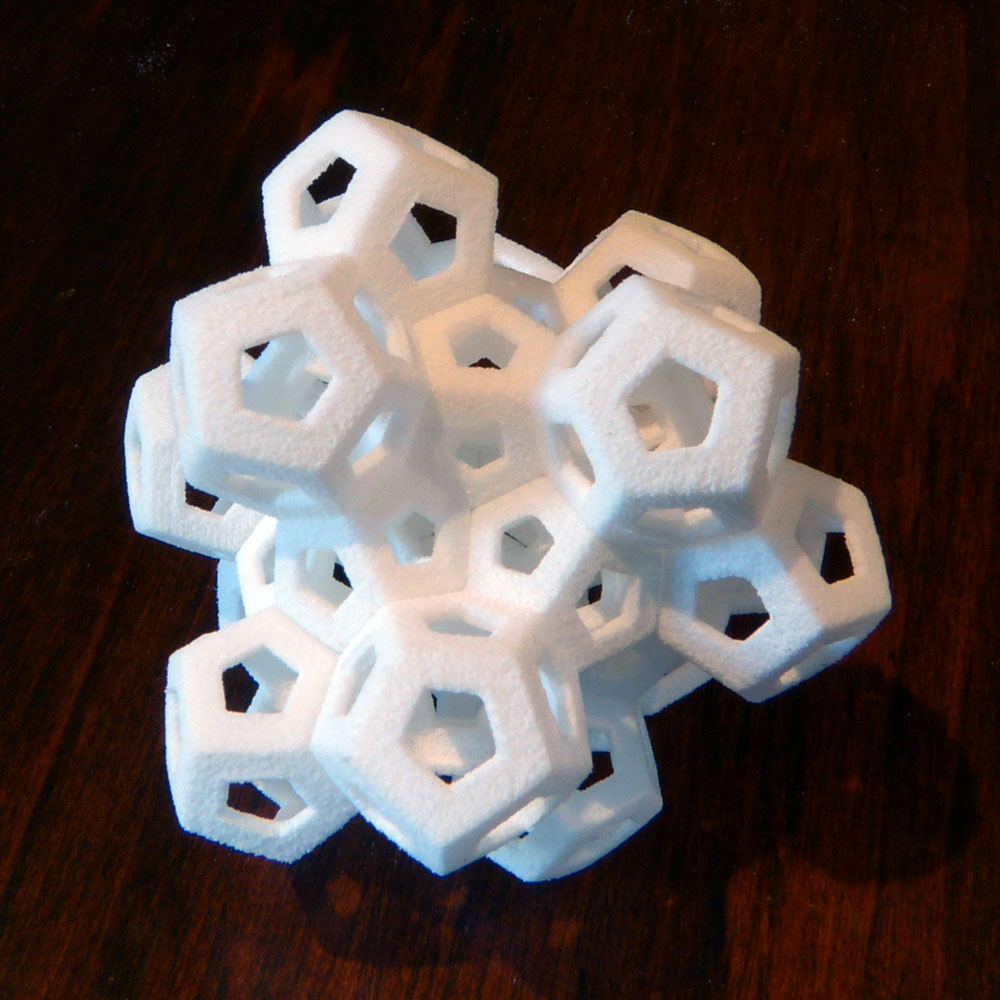}}\\
 \cline{1-1}
  $6\times \text{inner four}$ &\\
 \cline{1-1}
  \footnotesize{Up to three ribs}&\\
  \footnotesize{can be replaced}&\\
  \footnotesize{by inner sixs.}&\\
  \footnotesize{}&\\
  \footnotesize{}&\\

Dc24 Pulsar &\multirow{6}{*}{
\includegraphics[height=82pt]{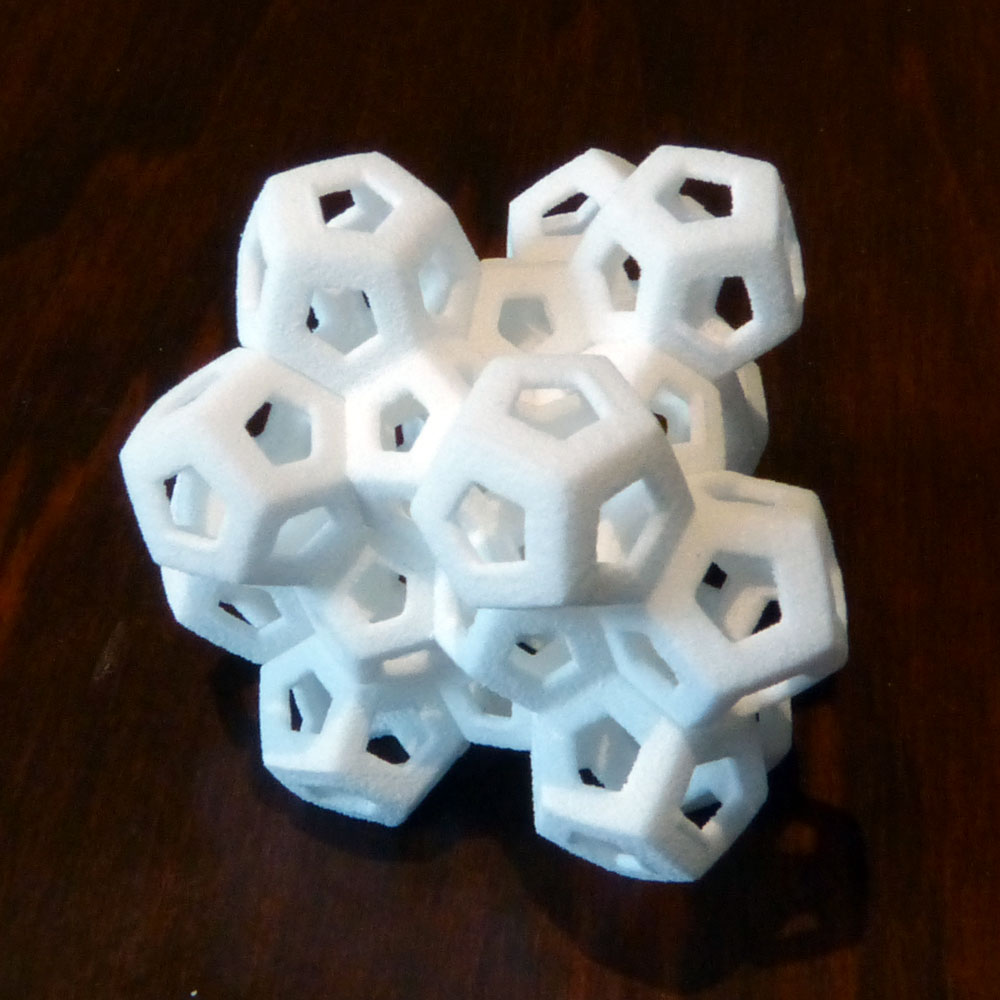}
\includegraphics[height=82pt]{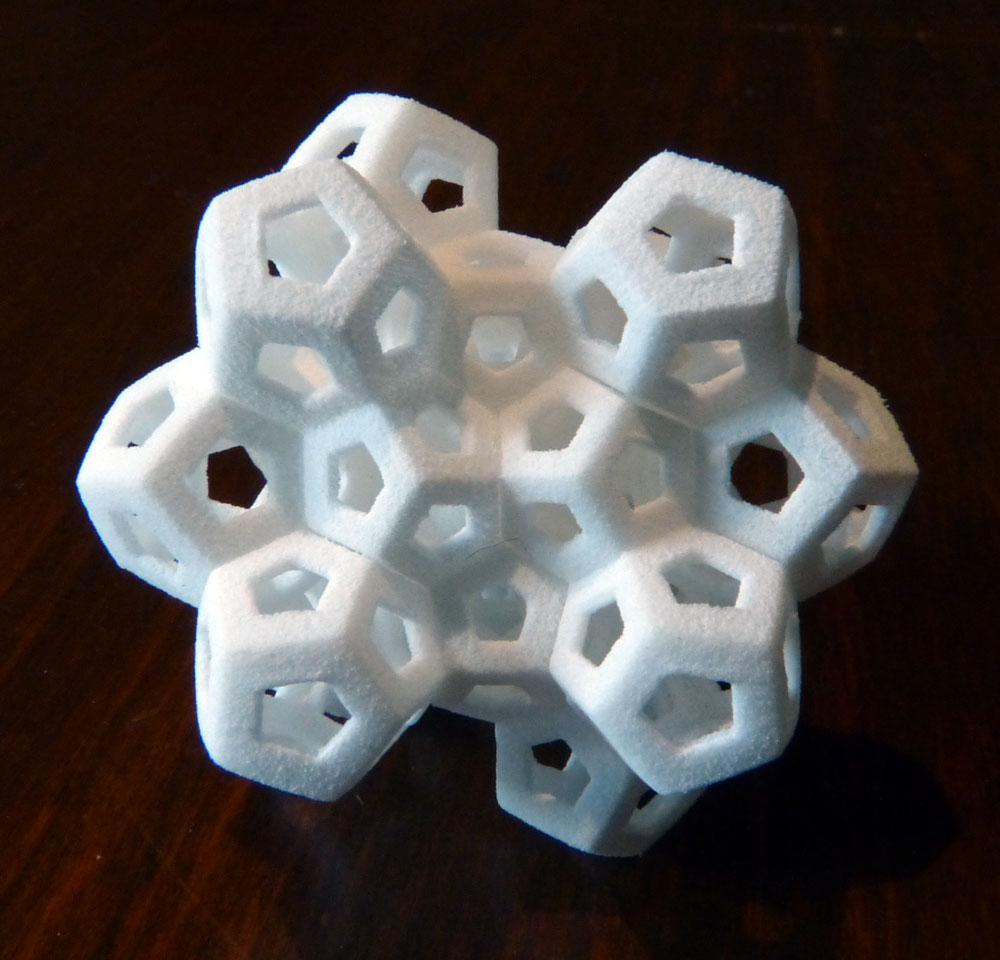}
\includegraphics[height=82pt]{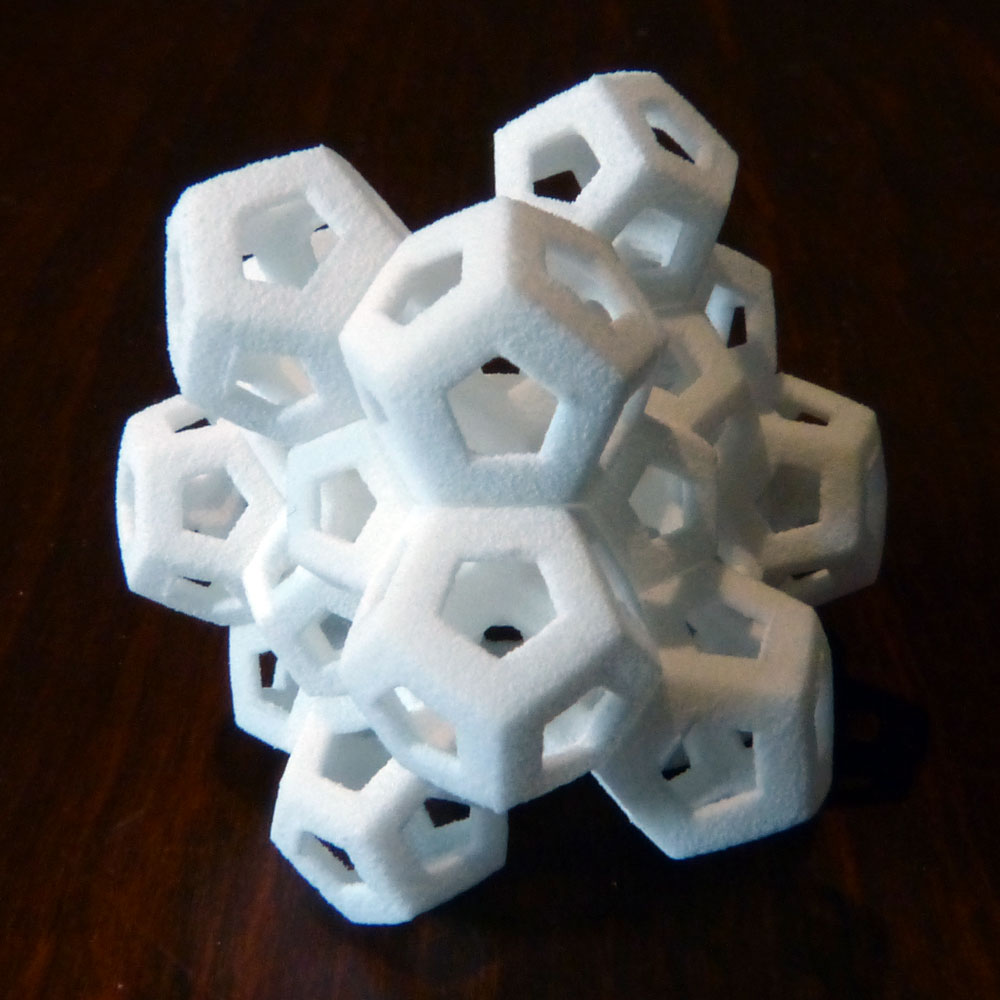}
\includegraphics[height=82pt]{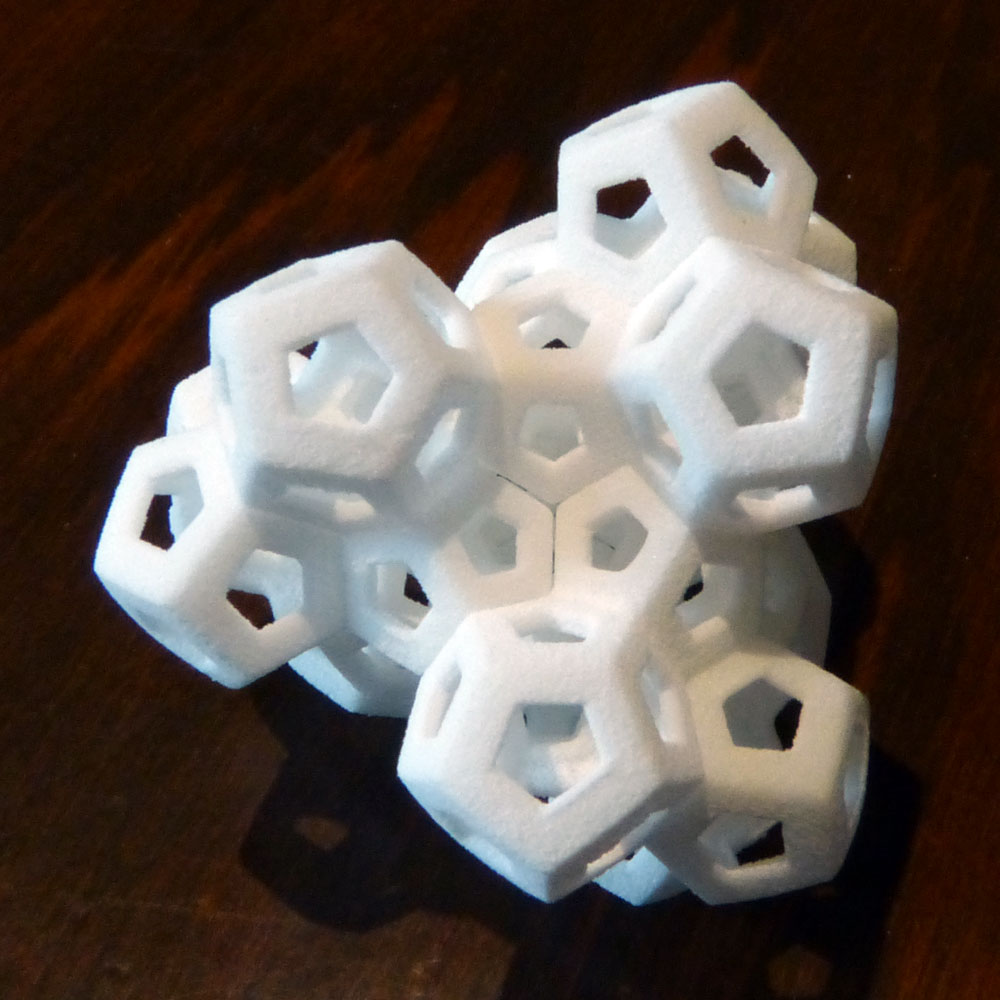}
}\\
 \cline{1-1}
  $6\times \text{inner four}$ &\\
 \cline{1-1}
  \footnotesize{Any number of ribs}&\\
  \footnotesize{can be replaced by}&\\
  \footnotesize{inner sixs.}&\\
  \footnotesize{}&\\
  \footnotesize{}&\\

 Dc29 Space Invader &\multirow{6}{*}{
\includegraphics[height=82pt]{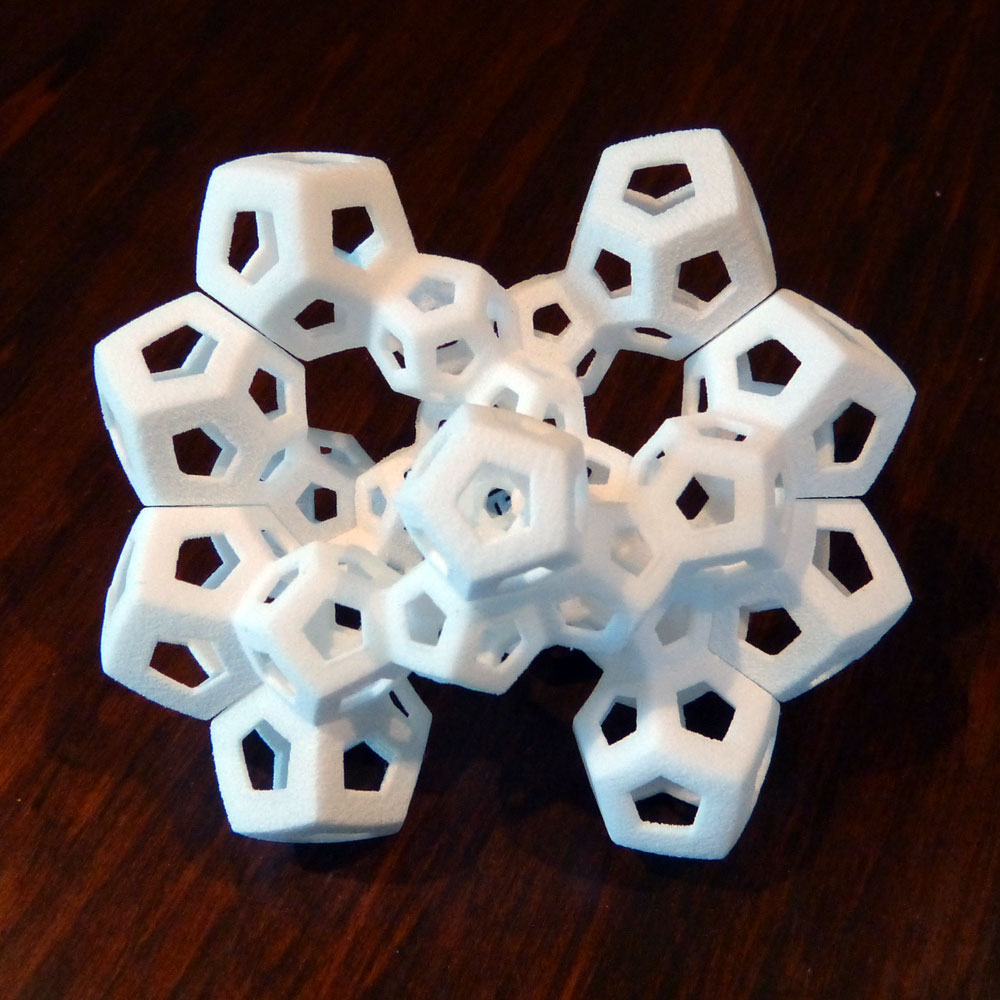}
\includegraphics[height=82pt]{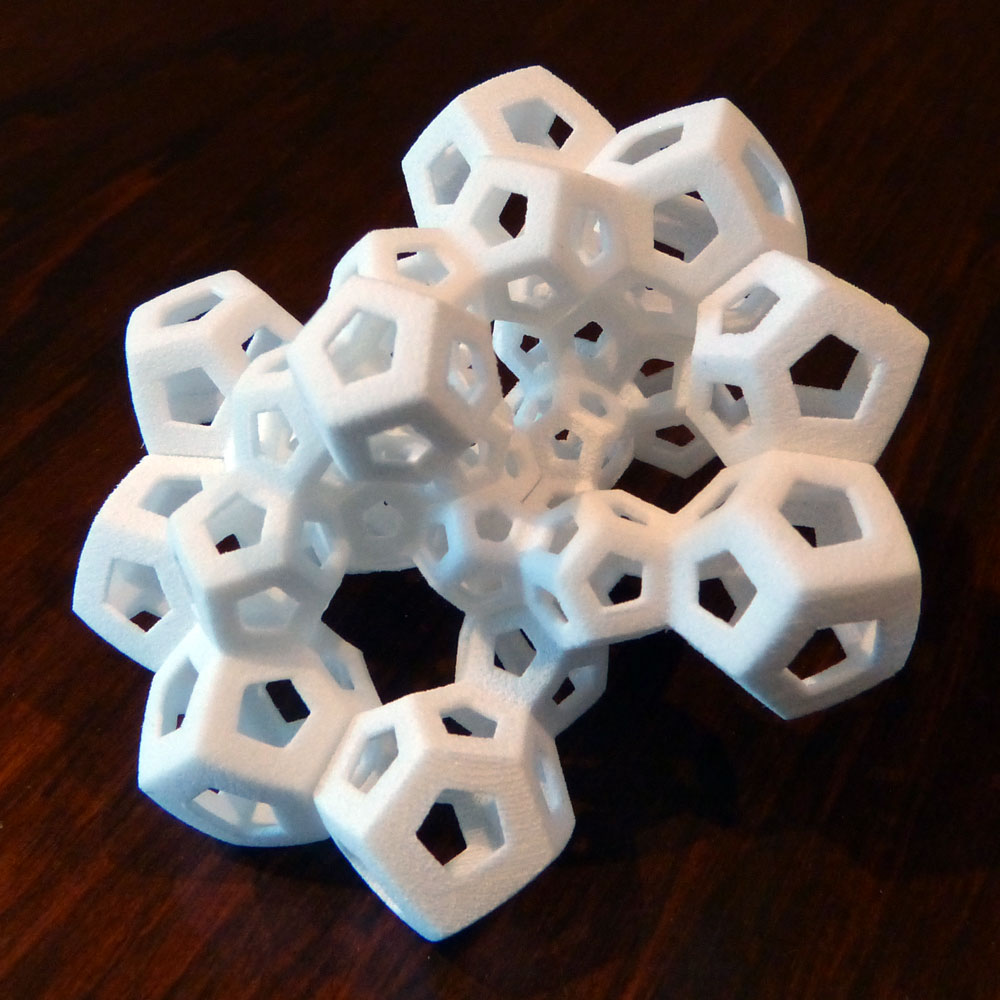}
\includegraphics[height=82pt]{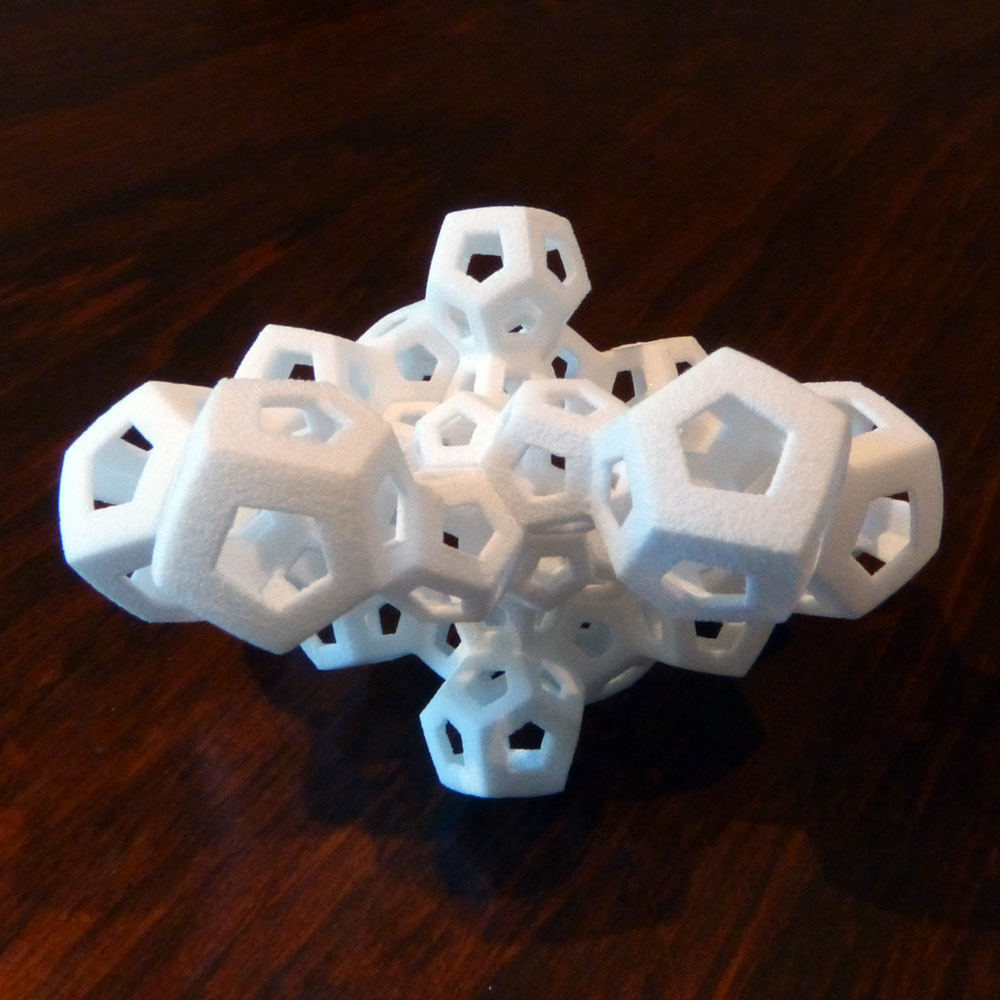}
\includegraphics[height=82pt]{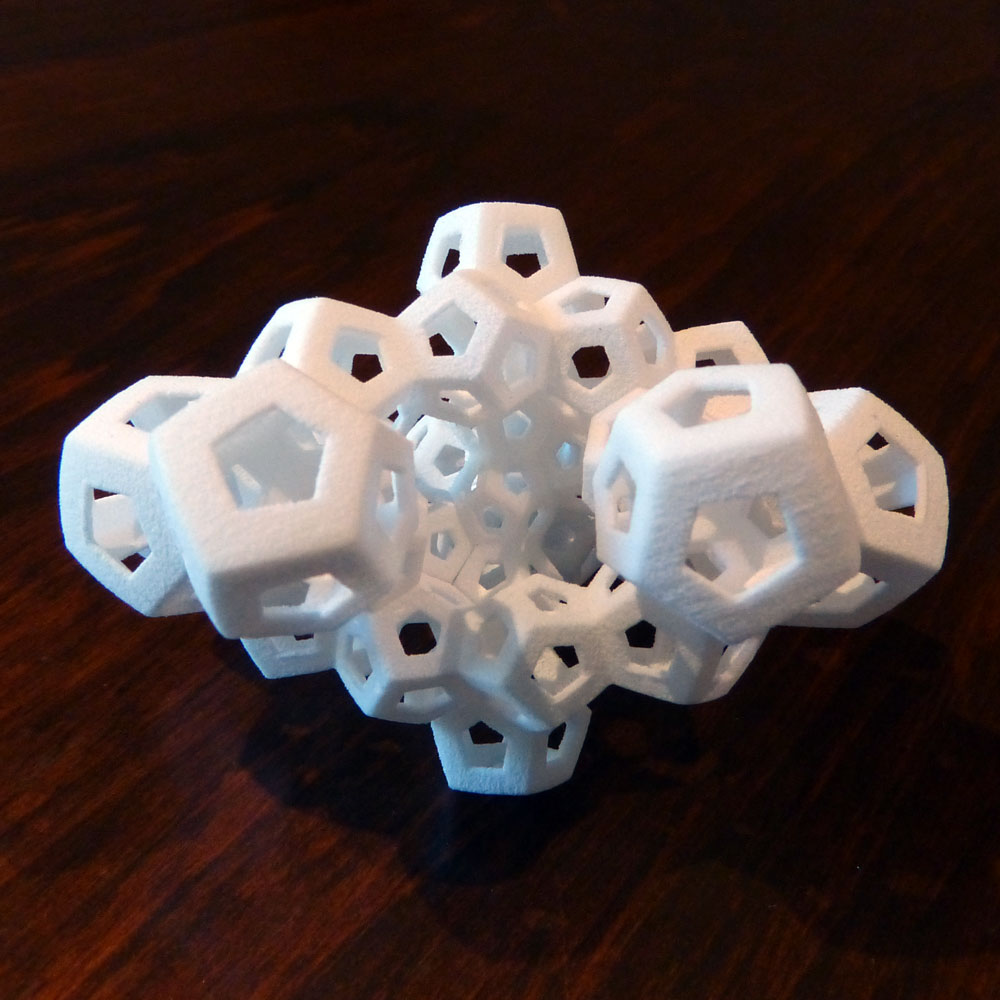}
}\\
 \cline{1-1}
  $2\times \text{inner six}$ &\\
  $2\times \text{outer six}$ &\\
  $1\times \text{spine}$&\\
 \cline{1-1}
  \footnotesize{Can add $2\times\text{equator}$.}&\\
  \footnotesize{}&\\
  \footnotesize{}&\\

Dc30 Star &\multirow{6}{*}{
\includegraphics[height=82pt]{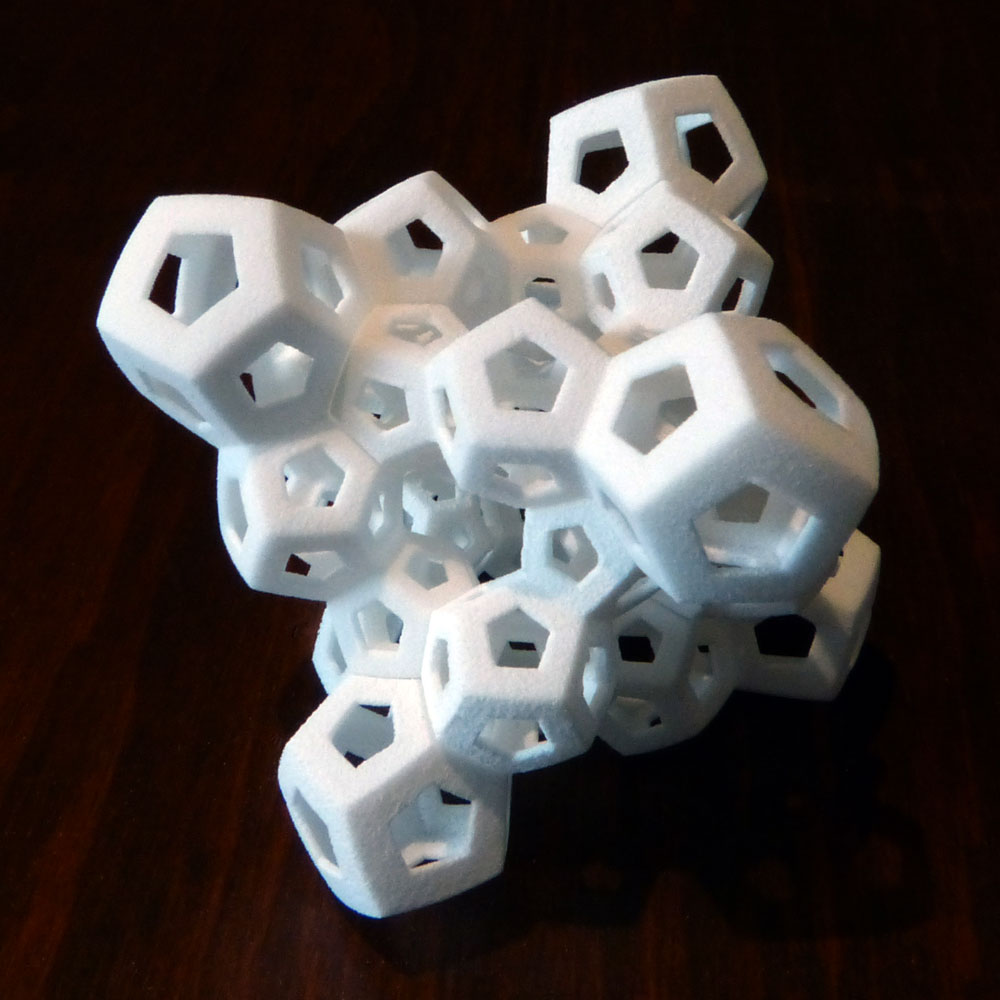}
\includegraphics[height=82pt]{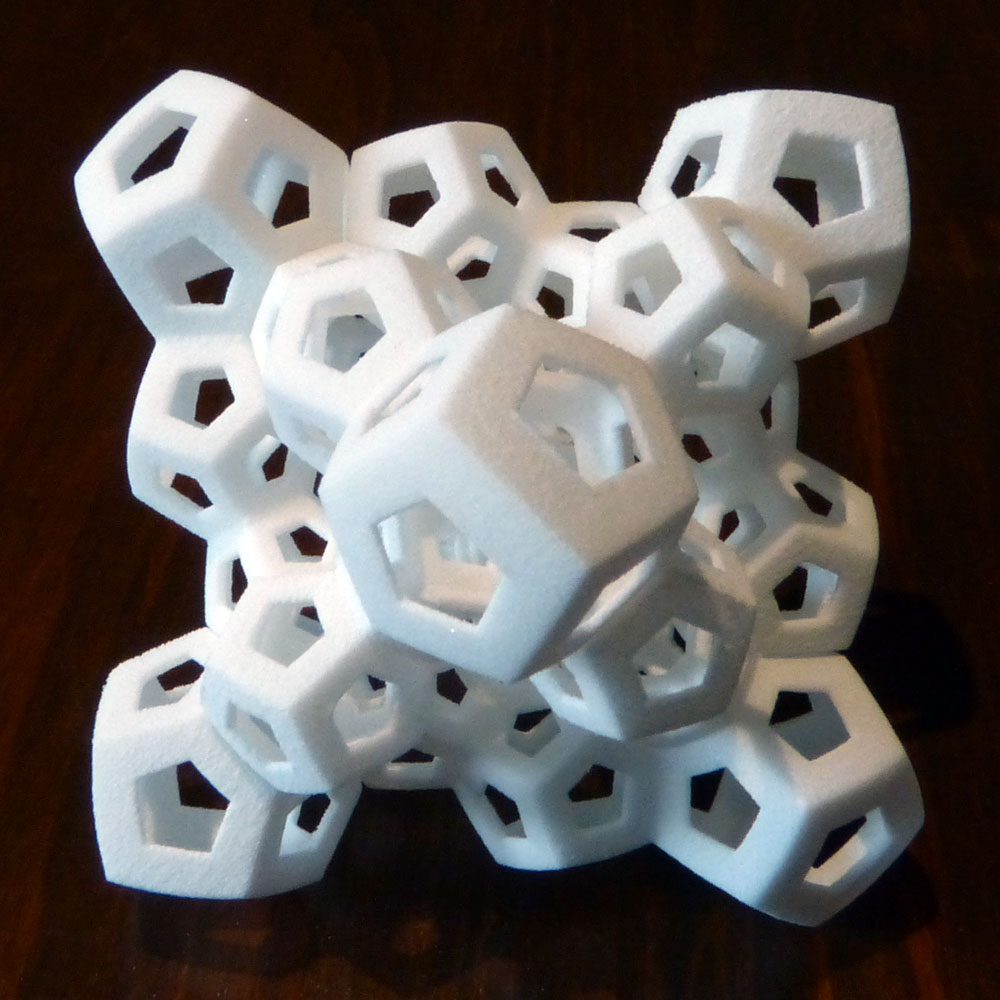}
\includegraphics[height=82pt]{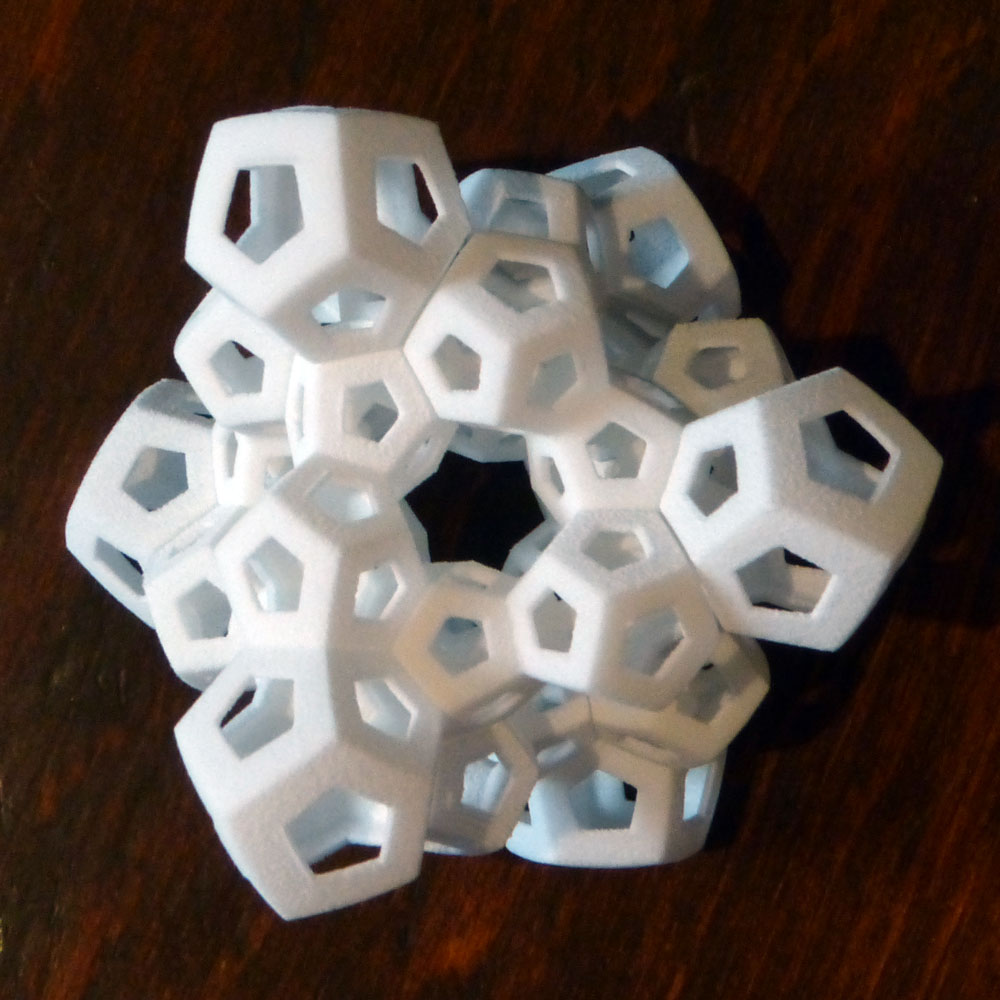}
}\\
 \cline{1-1}
  $3\times \text{outer four}$ &\\
  $3\times \text{outer six}$ &\\
  \footnotesize{}&\\
  \footnotesize{}&\\
  \footnotesize{}&\\
  \footnotesize{}&\\

 Dc30 Ring & \multirow{6}{*}{
\includegraphics[height=82pt]{Figures/D30_outer_wreathe_1_small.jpg}
\includegraphics[height=82pt]{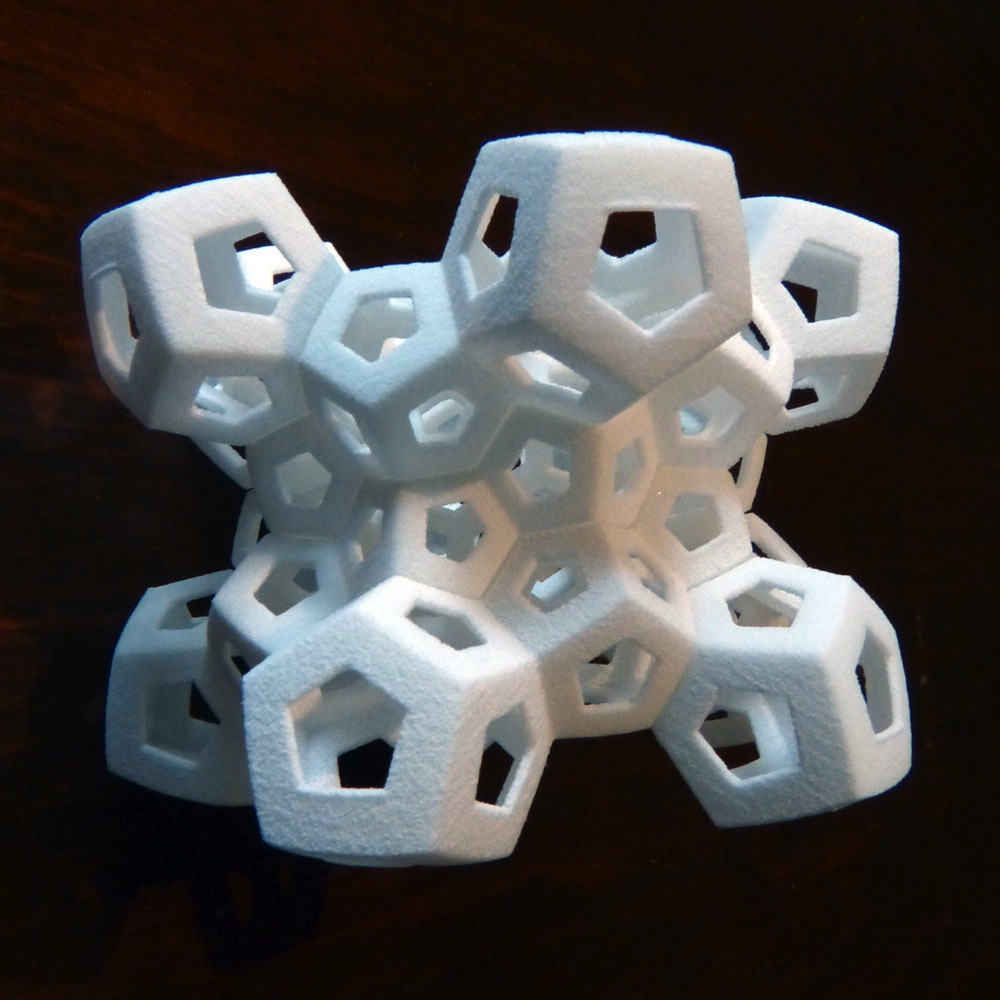}
\includegraphics[height=82pt]{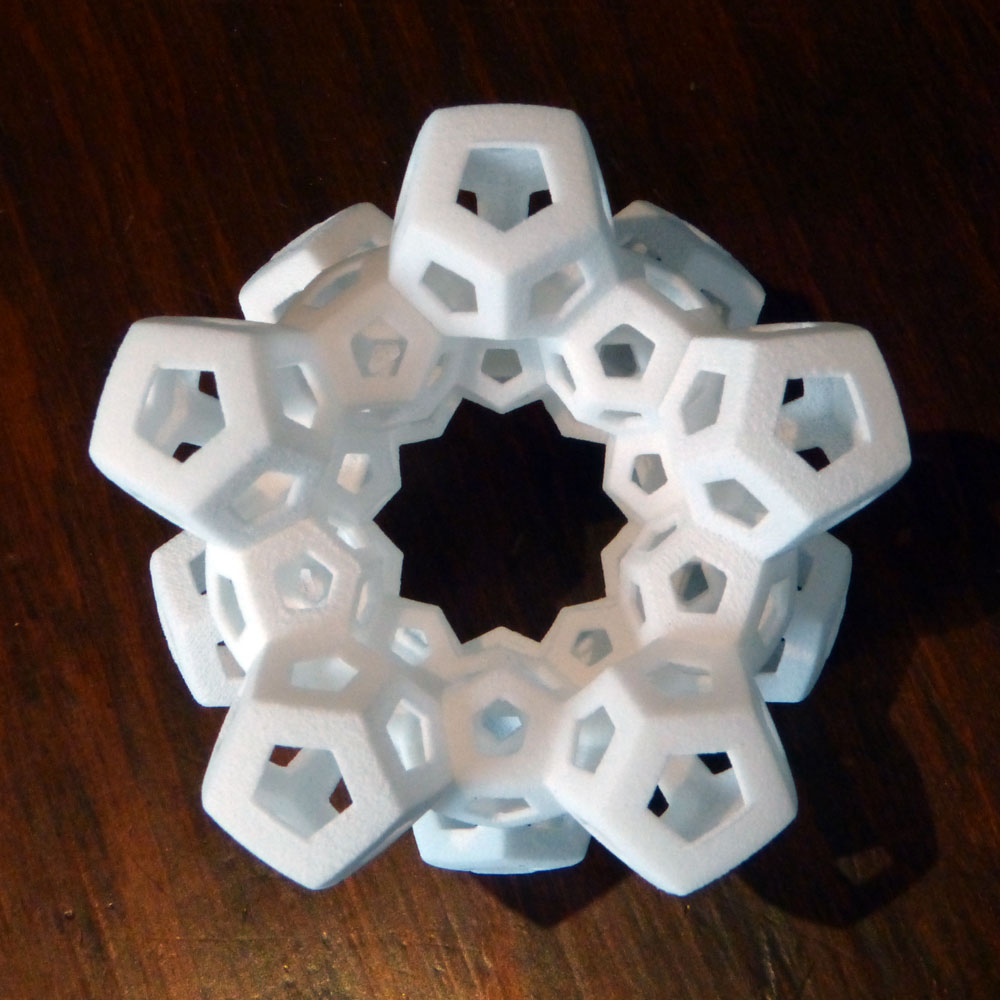}
}\\
 \cline{1-1}
  $5\times \text{outer six}$ &\\
 \cline{1-1}
  \footnotesize{Replace all ribs with}&\\
  \footnotesize{inner sixs to get the}&\\
  \footnotesize{Inner Ring.}&\\
  \footnotesize{}&\\
  \footnotesize{}&\\

 Dc30 Comet & \multirow{6}{*}{
\includegraphics[height=82pt]{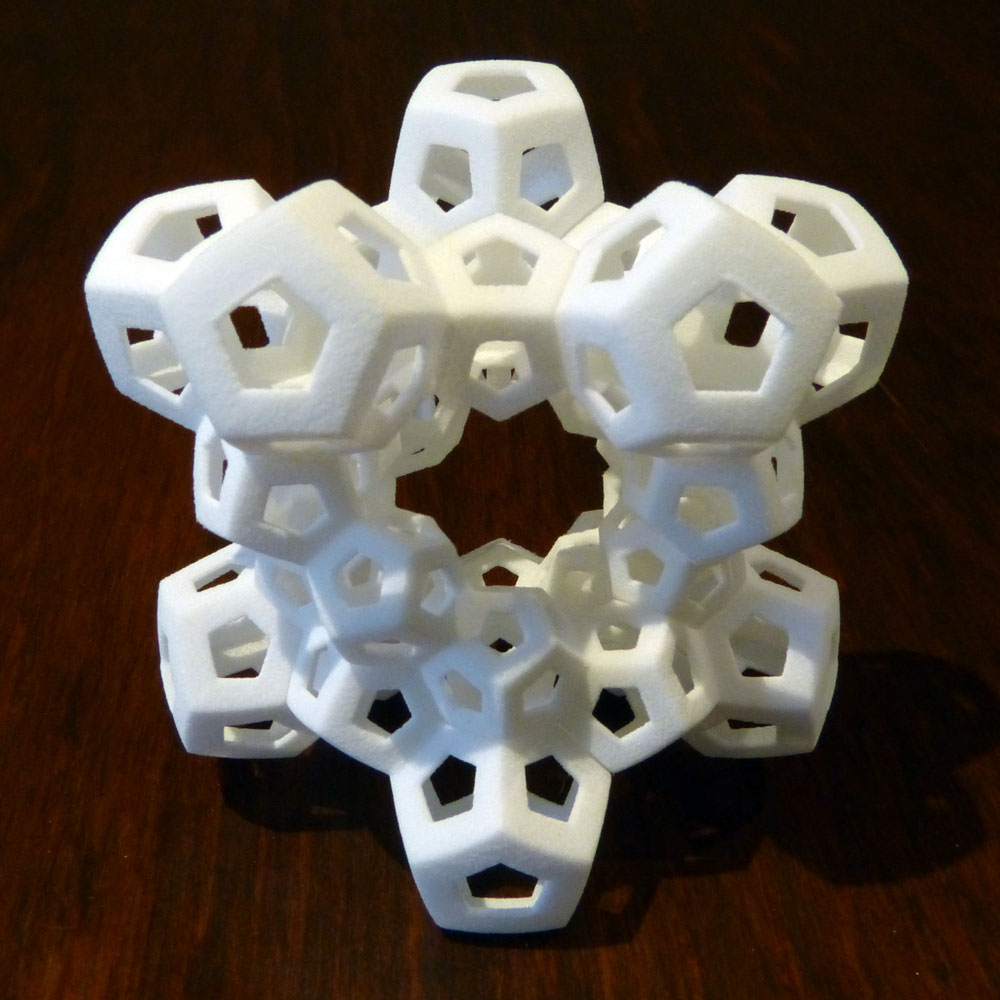}
\includegraphics[height=82pt]{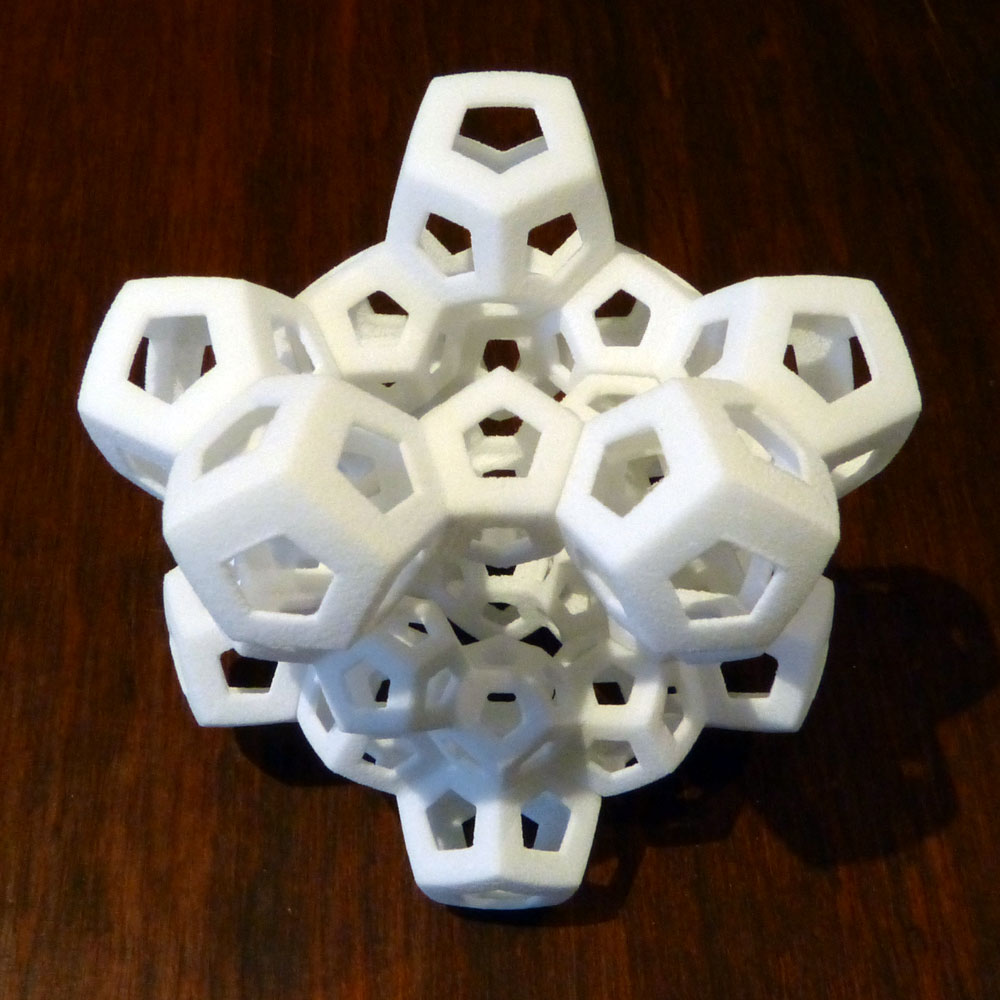}
\includegraphics[height=82pt]{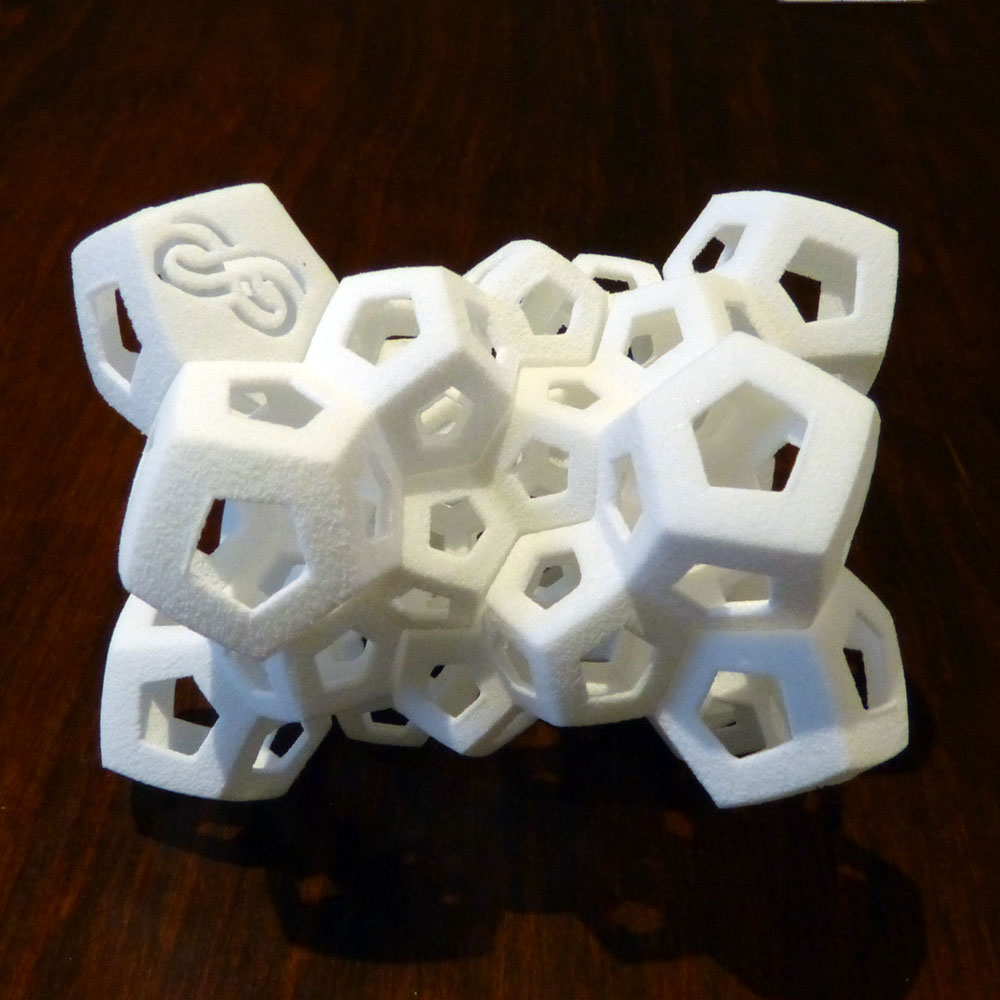}
\includegraphics[height=82pt]{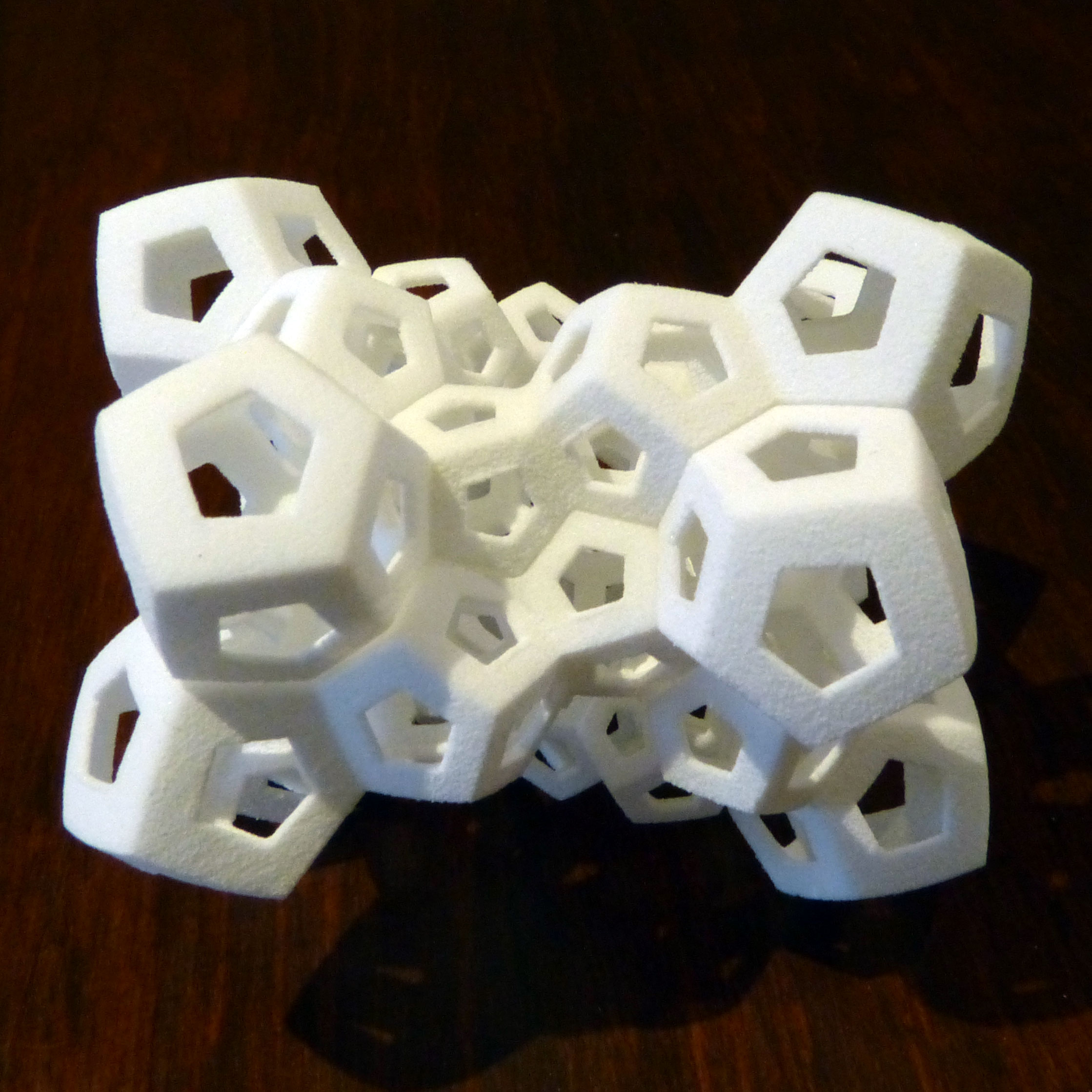}
}\\
 \cline{1-1}
  $5\times \text{outer six}$ &\\
 \cline{1-1}
  \footnotesize{Add a spine and one}&\\
  \footnotesize{inner four to make the}&\\
  \footnotesize{Comet more rigid.}&\\
  \footnotesize{}&\\
  \footnotesize{}&\\
\end{tabular}

\newpage
\begin{tabular}{lc}

Dc36 Alien &\multirow{6}{*}{
\includegraphics[height=82pt]{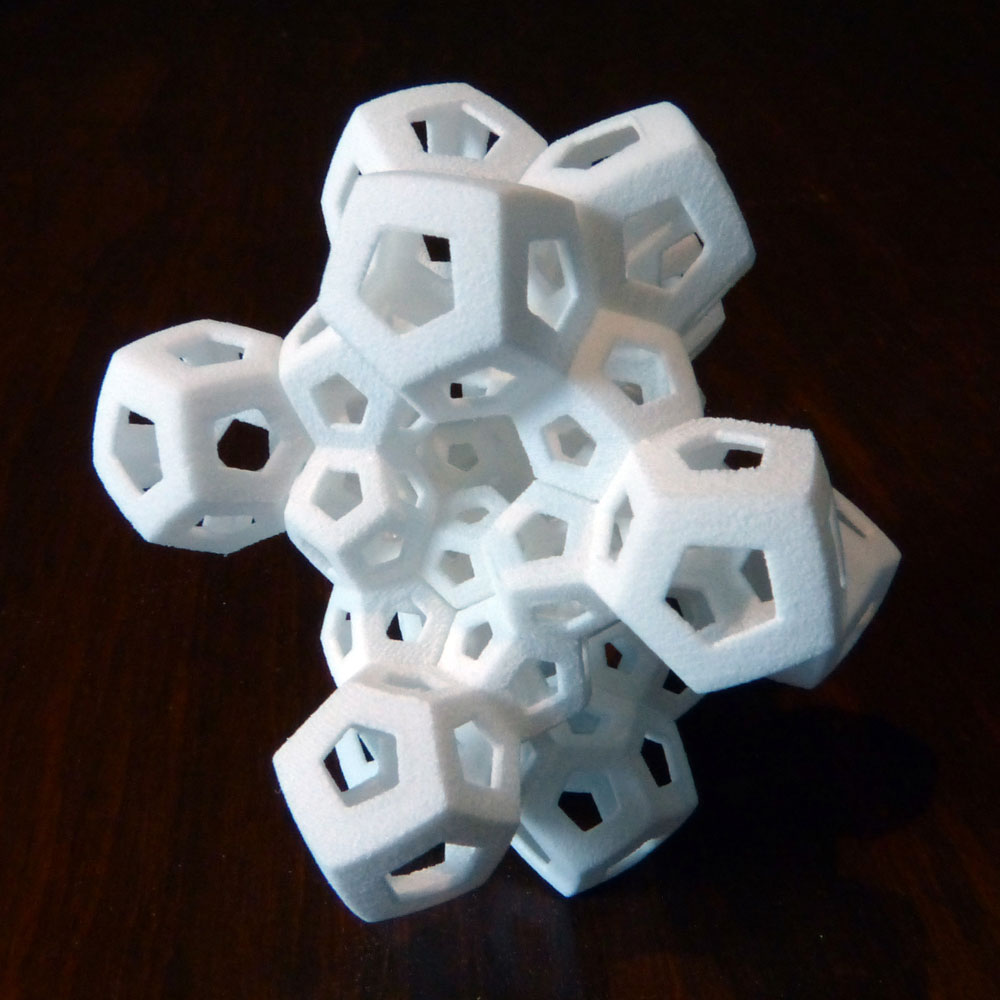}
\includegraphics[height=82pt]{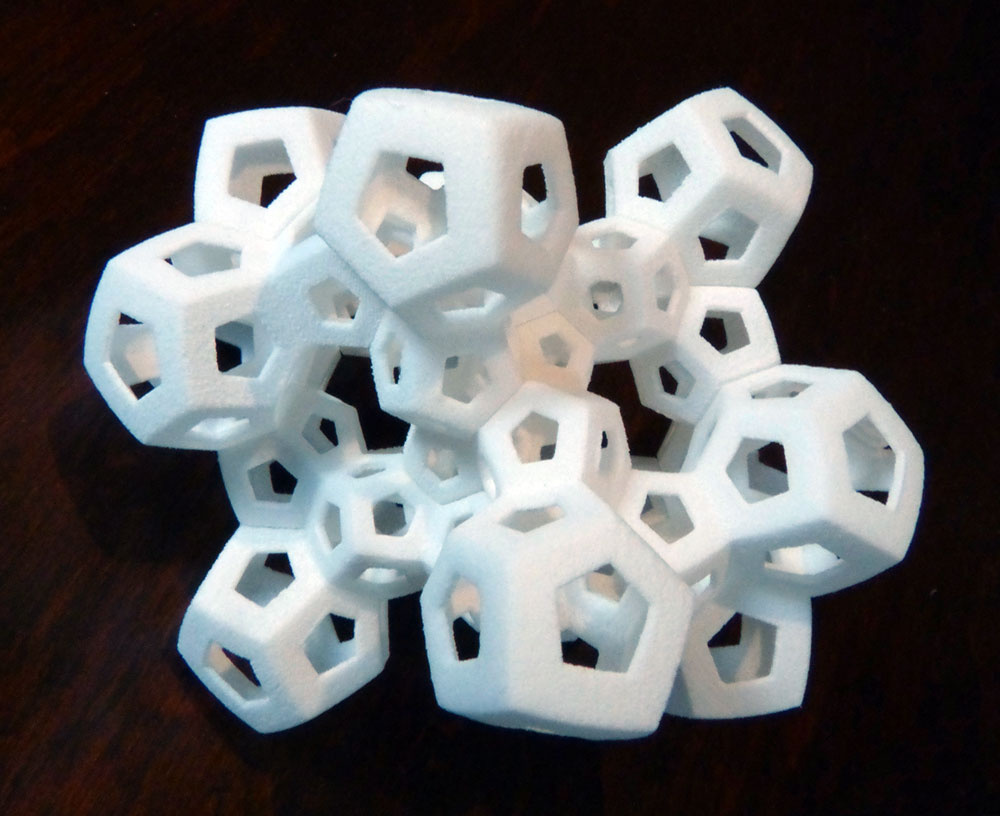}
\includegraphics[height=82pt]{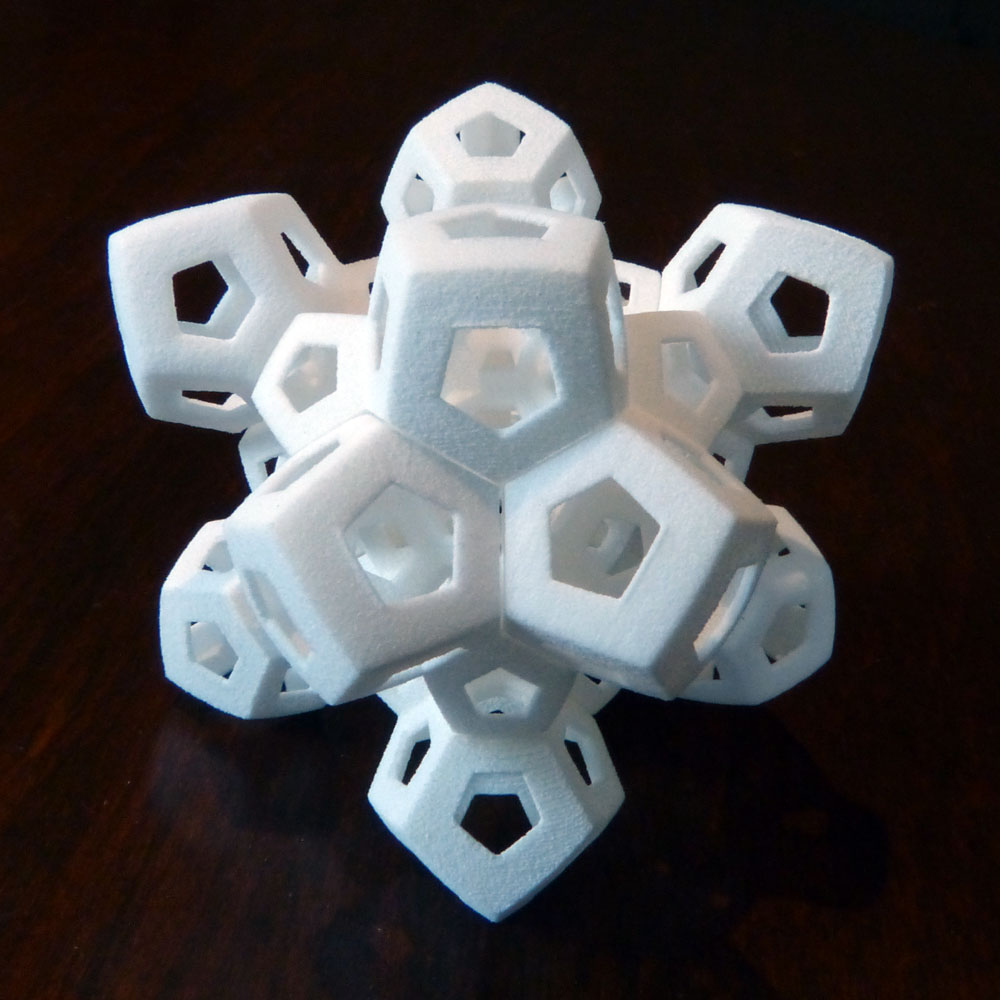}
}\\
 \cline{1-1}
  $3\times \text{inner six}$ &\\
  $3\times \text{outer six}$ &\\
 \cline{1-1}
  \footnotesize{Either set of 6s can}&\\
  \footnotesize{be replaced by 4s.}&\\
  \footnotesize{}&\\
  \footnotesize{}&\\

Dc36 Pulsar &\multirow{6}{*}{
\includegraphics[height=82pt]{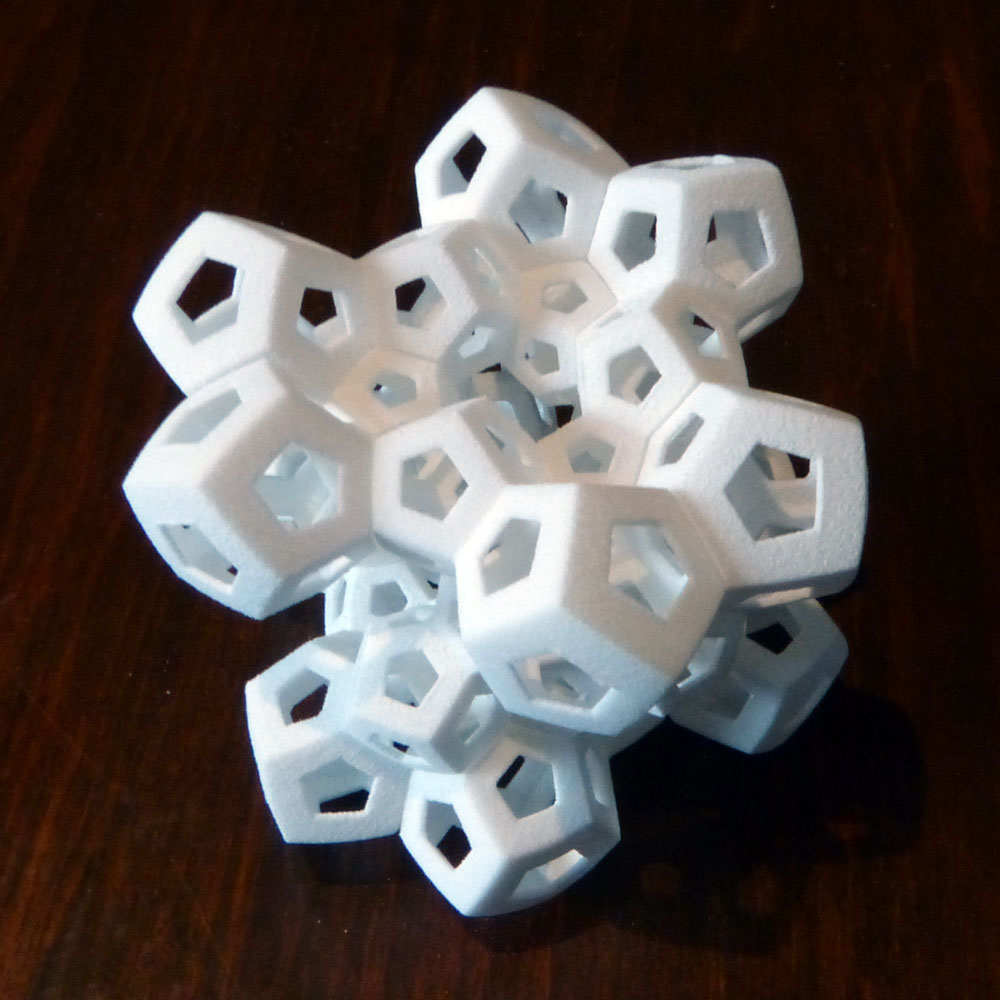}
\includegraphics[height=82pt]{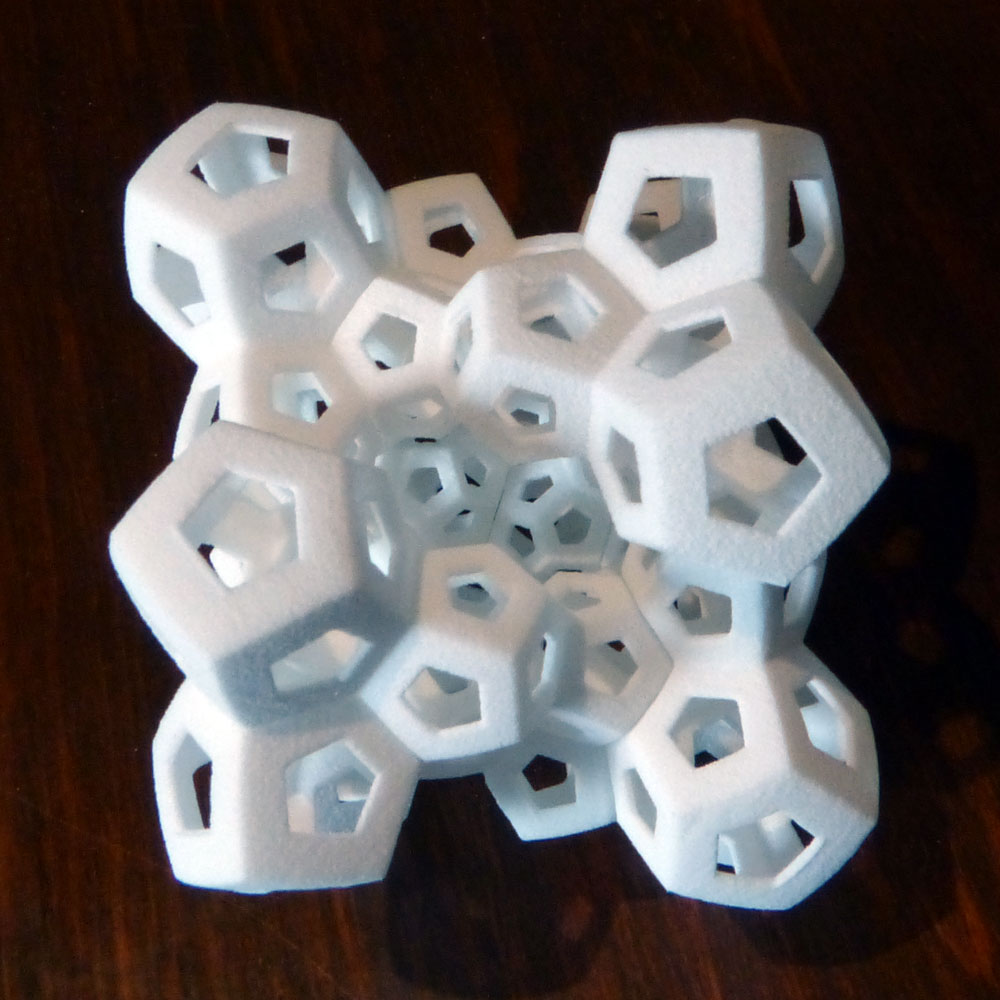}
\includegraphics[height=82pt]{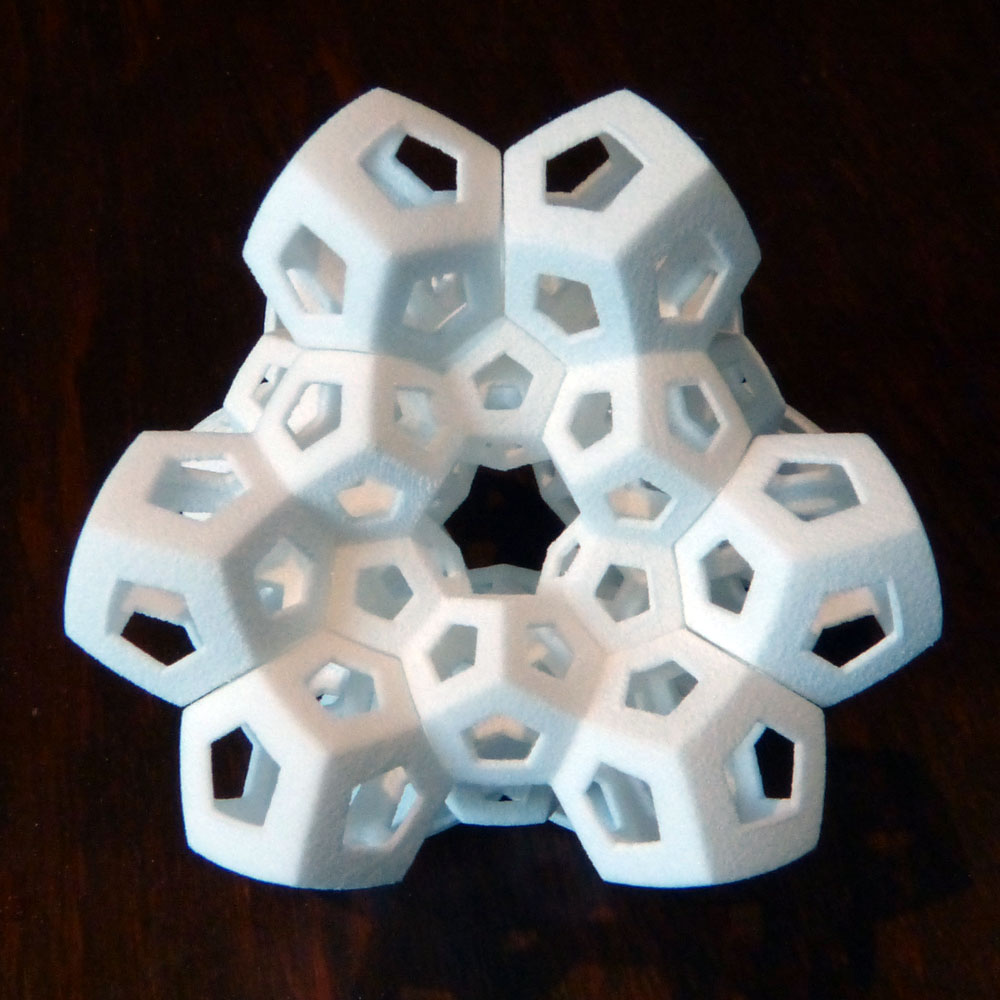}
}\\
 \cline{1-1}
  $6\times \text{outer six}$ &\\
 \cline{1-1}
  \footnotesize{Up to three ribs}&\\
  \footnotesize{can be replaced}&\\
  \footnotesize{by outer fours.}&\\
  \footnotesize{}&\\
  \footnotesize{}&\\

 Dc42 Alien &\multirow{6}{*}{
\includegraphics[height=82pt]{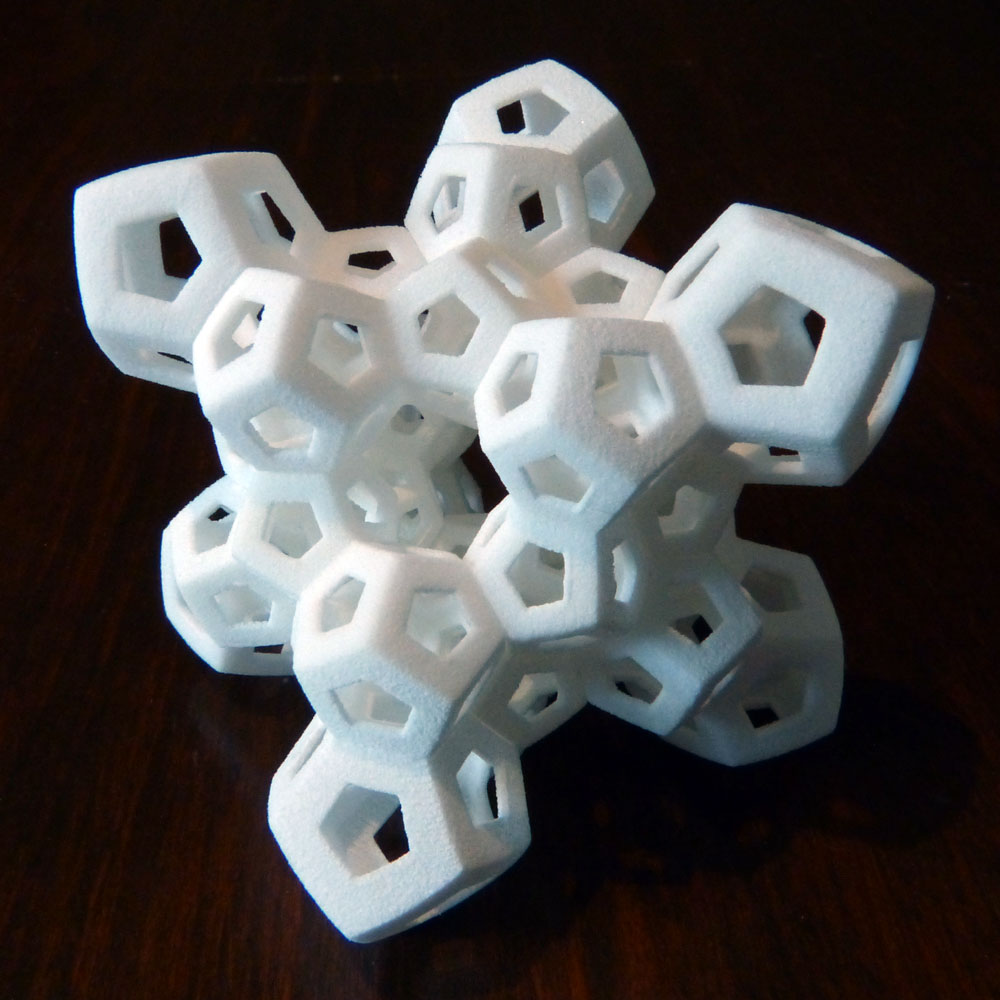}
\includegraphics[height=82pt]{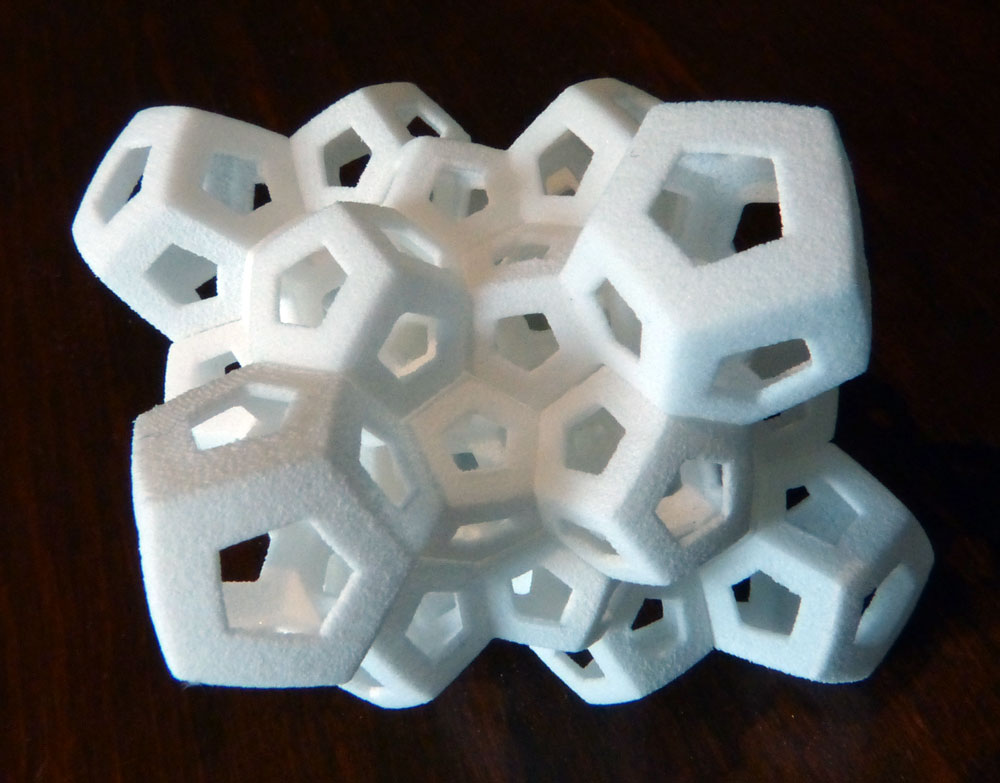}
\includegraphics[height=82pt]{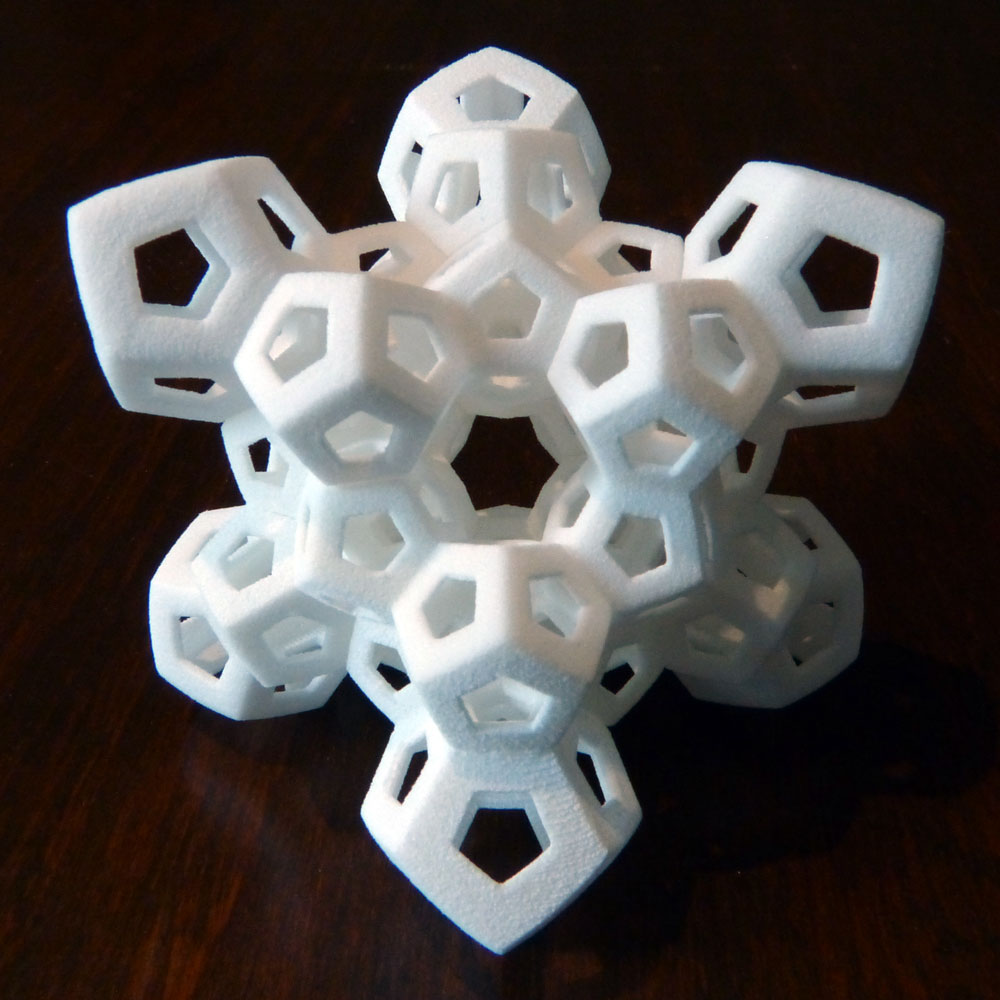}
}\\
 \cline{1-1}
  $6\times \text{outer four}$ &\\
  $3\times \text{inner six}$ &\\
  \footnotesize{}&\\
  \footnotesize{}&\\
  \footnotesize{}&\\
  \footnotesize{}&\\

  Dc45 Meteor & \multirow{6}{*}{
\includegraphics[height=82pt]{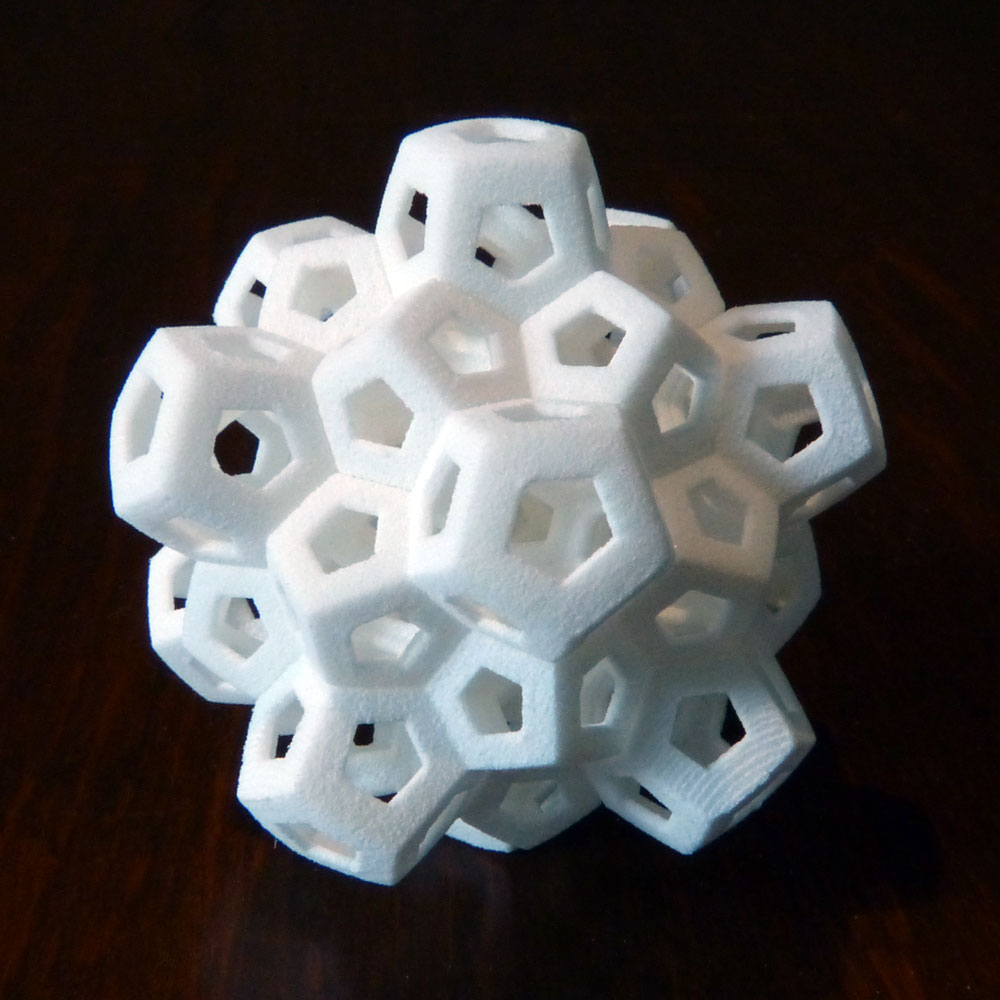}
\includegraphics[height=82pt]{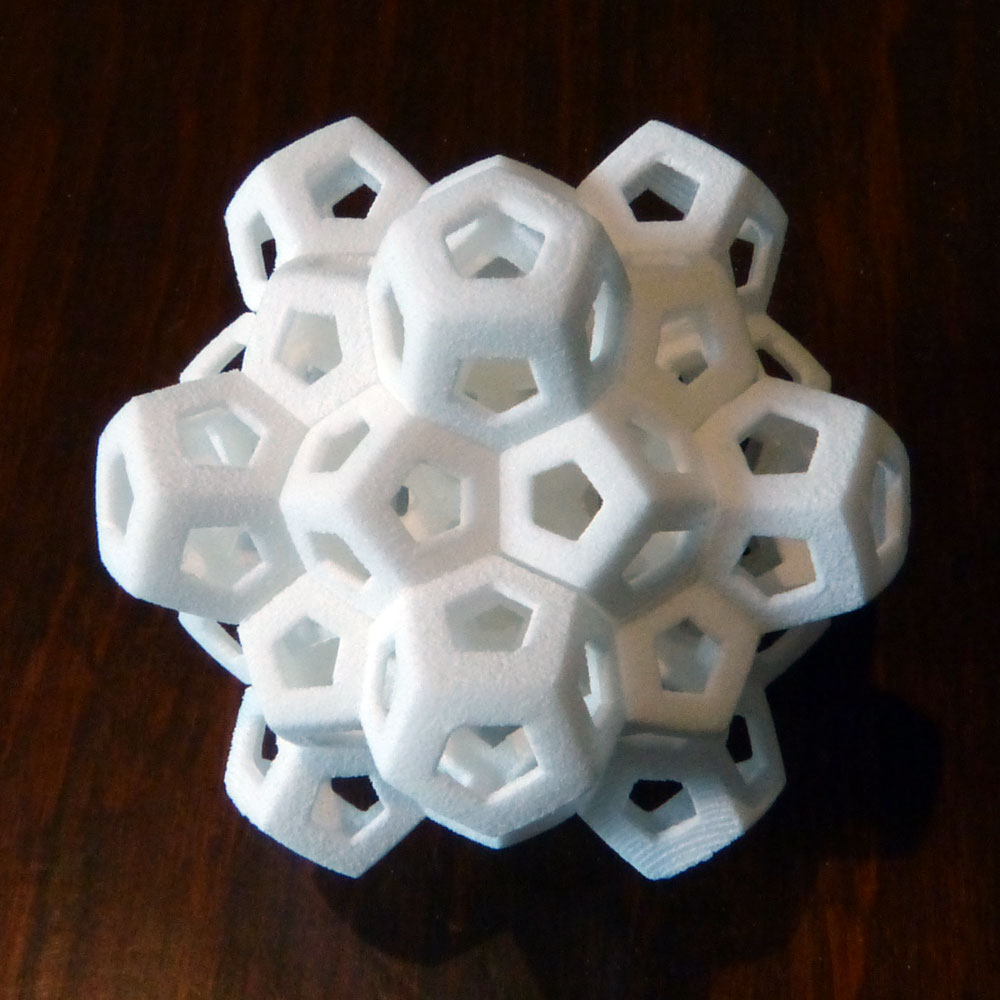}
\includegraphics[height=82pt]{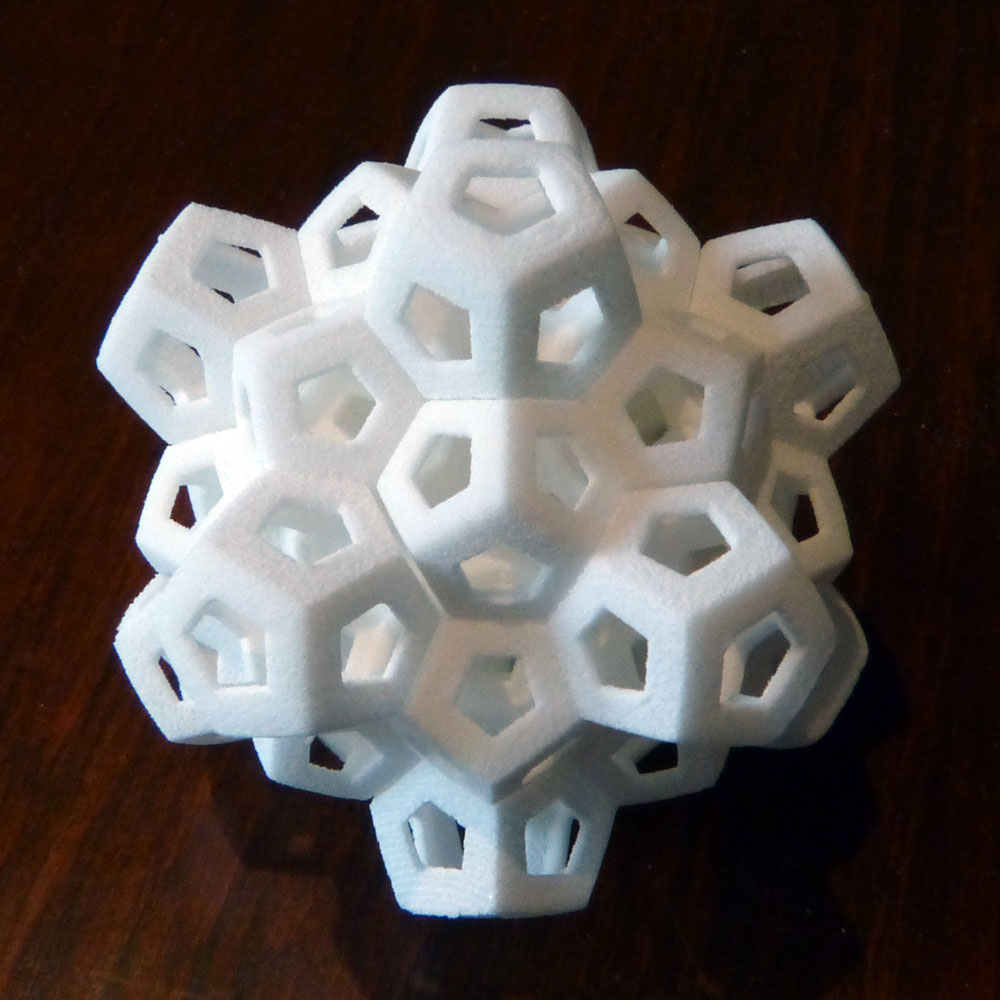}
\includegraphics[height=82pt]{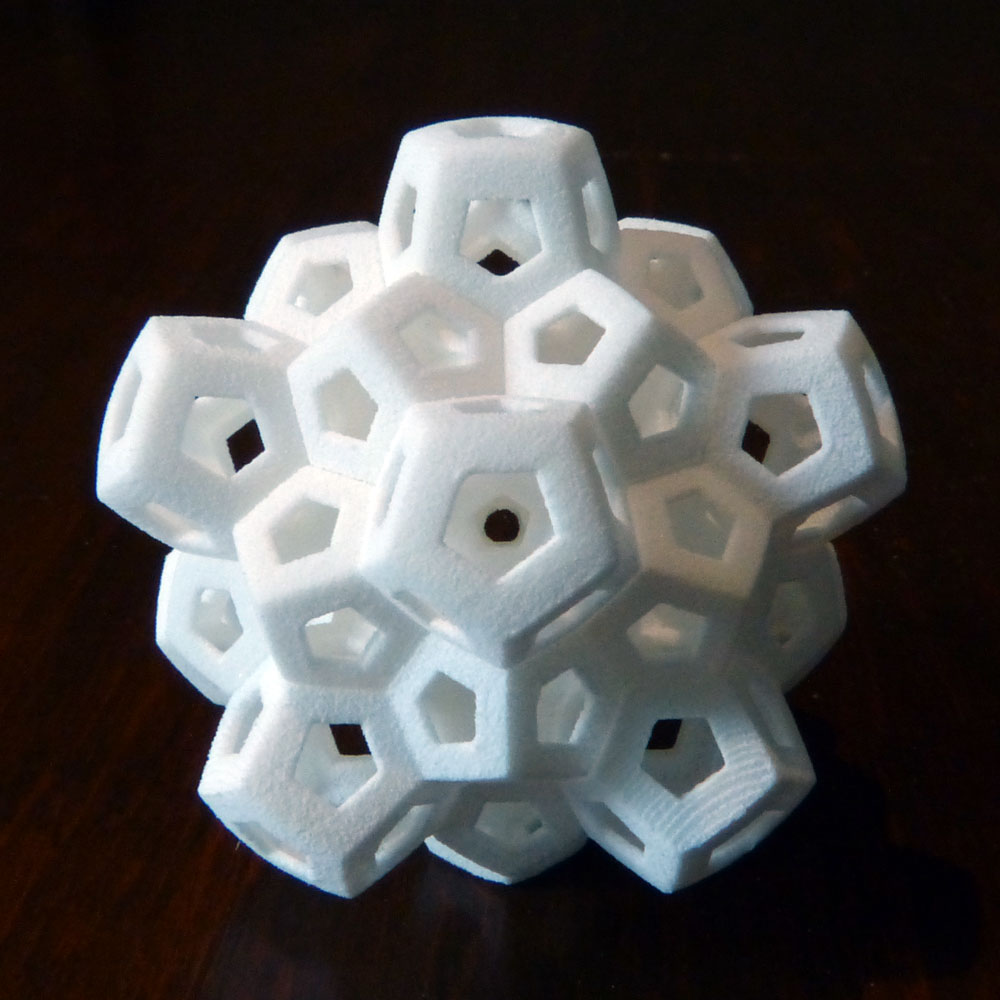}
}\\
 \cline{1-1}
  $5\times \text{inner four}$ &\\
  $5\times \text{outer four}$ &\\
  $1\times \text{spine}$ &\\
 \cline{1-1}
  \footnotesize{There are six ways}&\\
  \footnotesize{to build this.}&\\
  \footnotesize{}&\\

Dc50 Galaxy &\multirow{6}{*}{
\includegraphics[height=82pt]{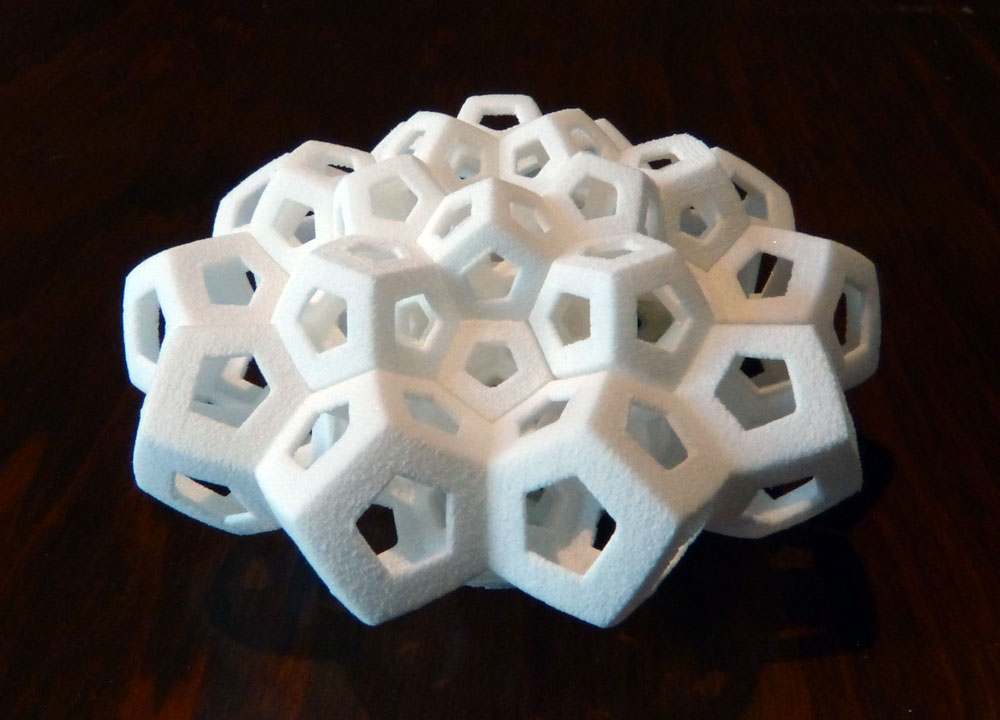}
\includegraphics[height=82pt]{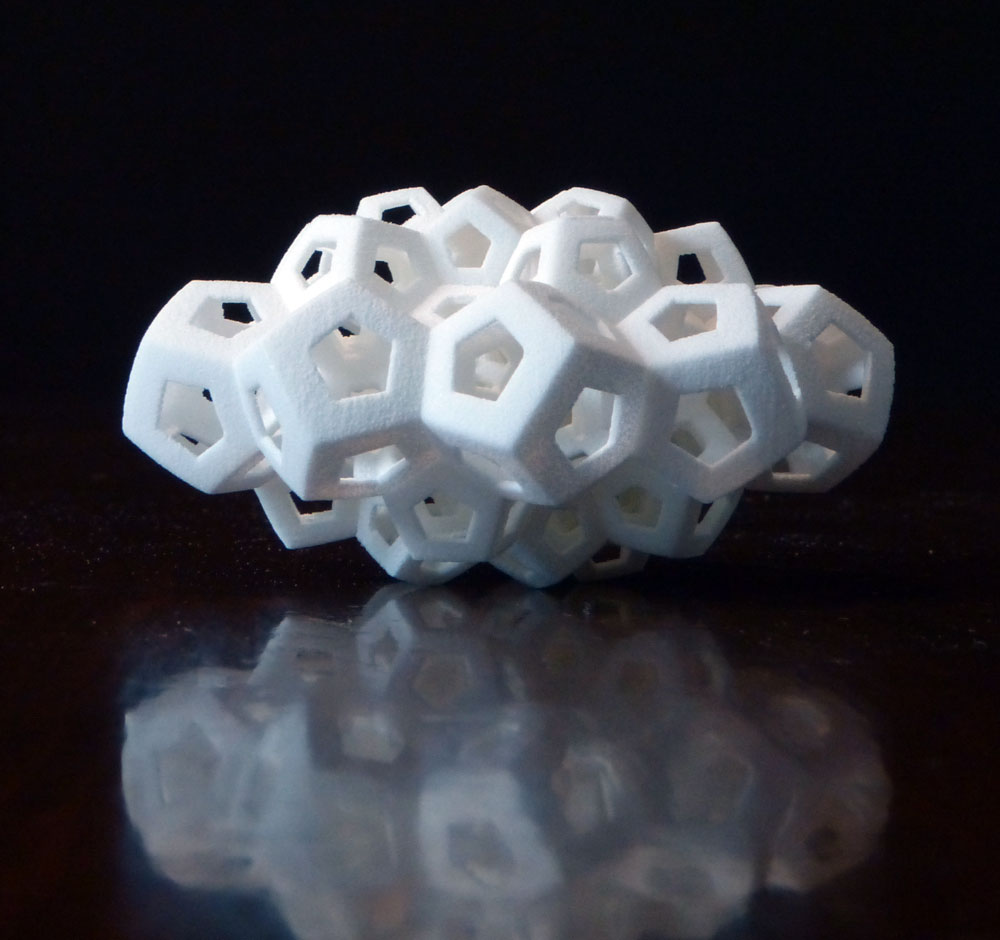}
\includegraphics[height=82pt]{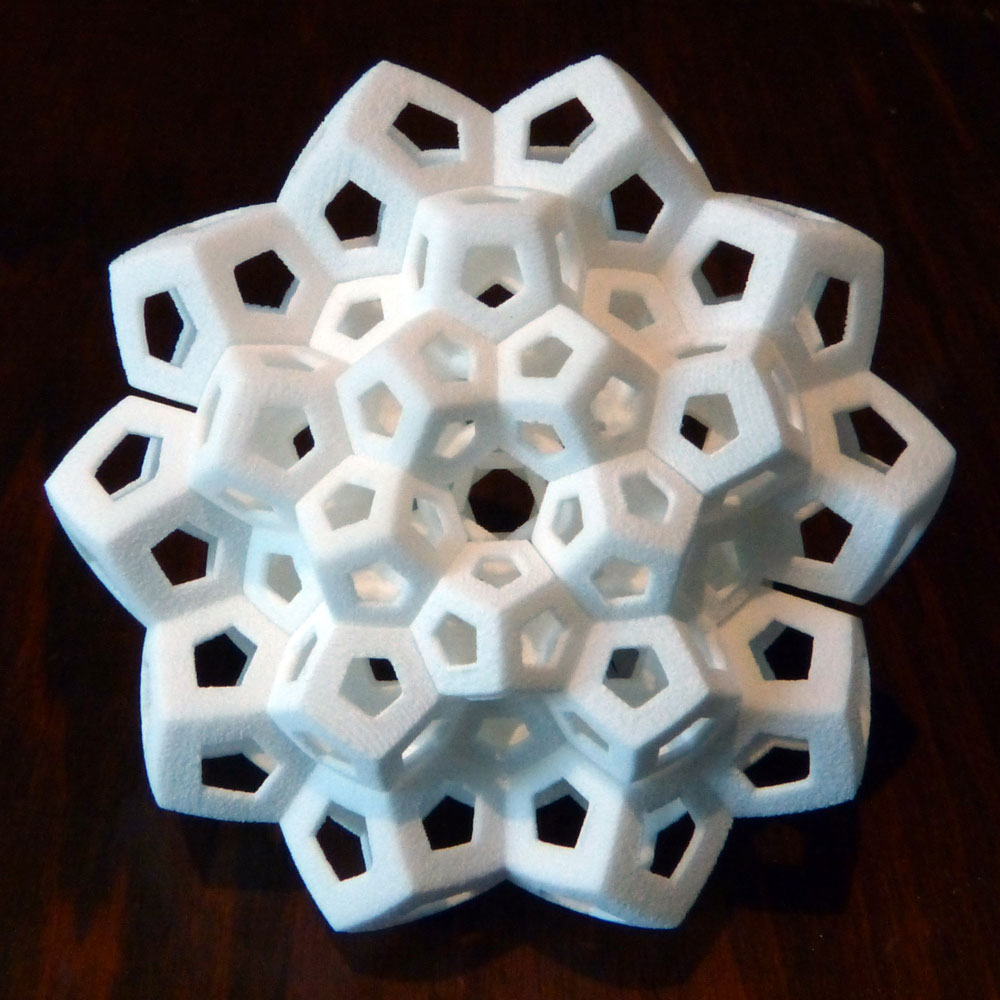}
}\\
 \cline{1-1}
  $5\times \text{inner four}$ &\\
  $5\times \text{outer four}$ &\\
  $2\times \text{equator}$&\\
  &\\
  &\\
  &\\

  Dc75 Meteor & \multirow{6}{*}{
\includegraphics[height=82pt]{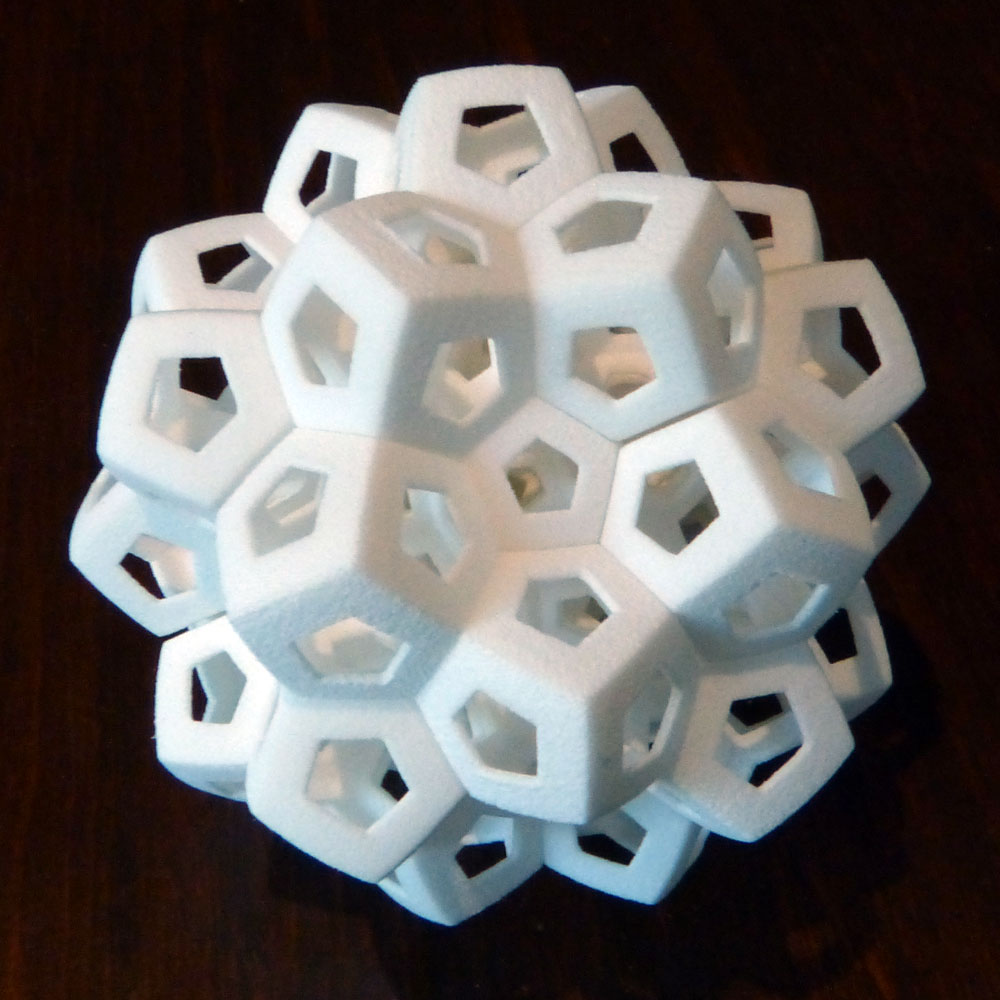}
\includegraphics[height=82pt]{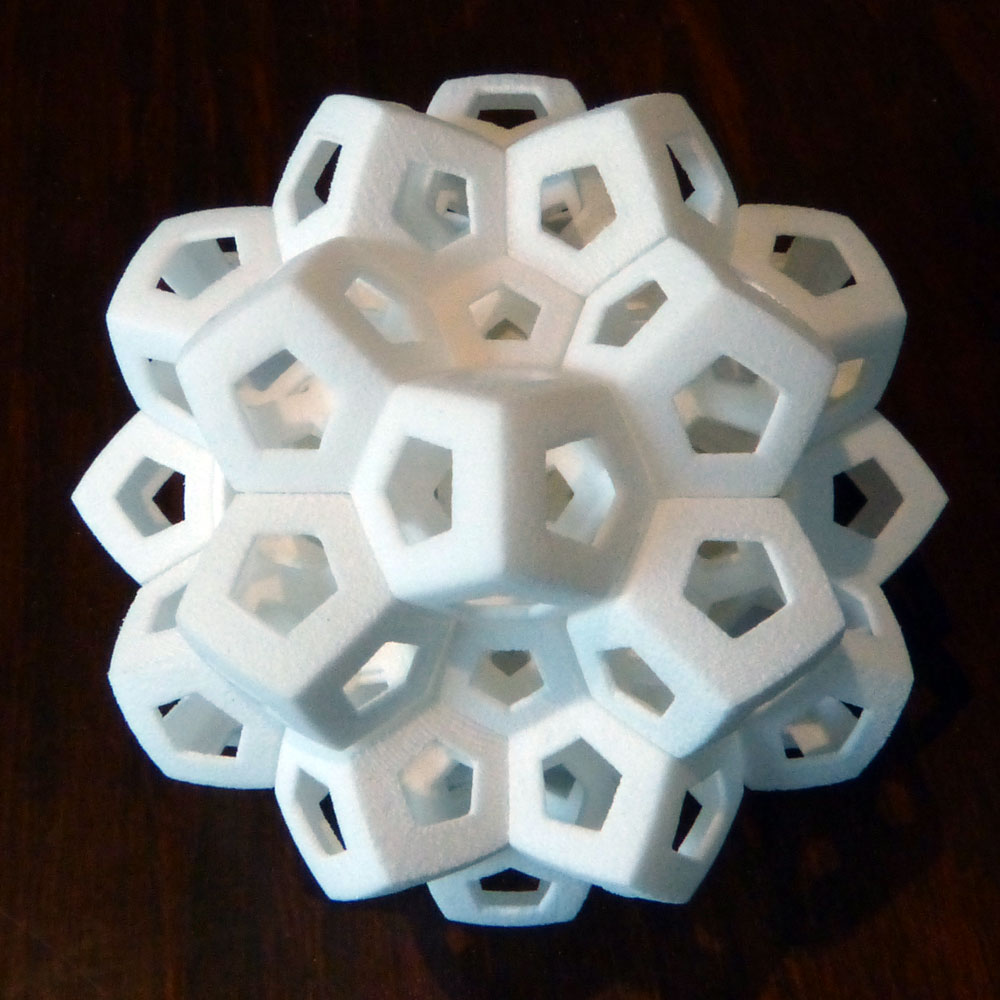}
\includegraphics[height=82pt]{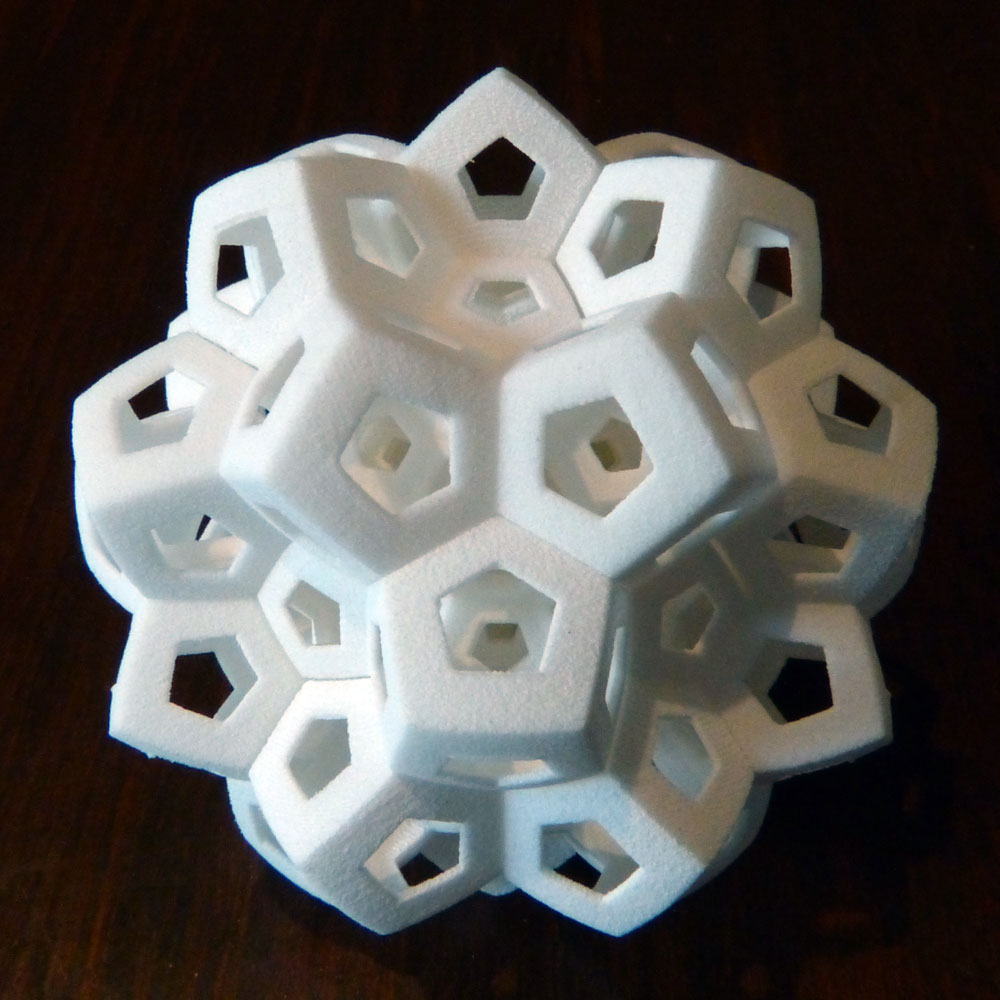}
\includegraphics[height=82pt]{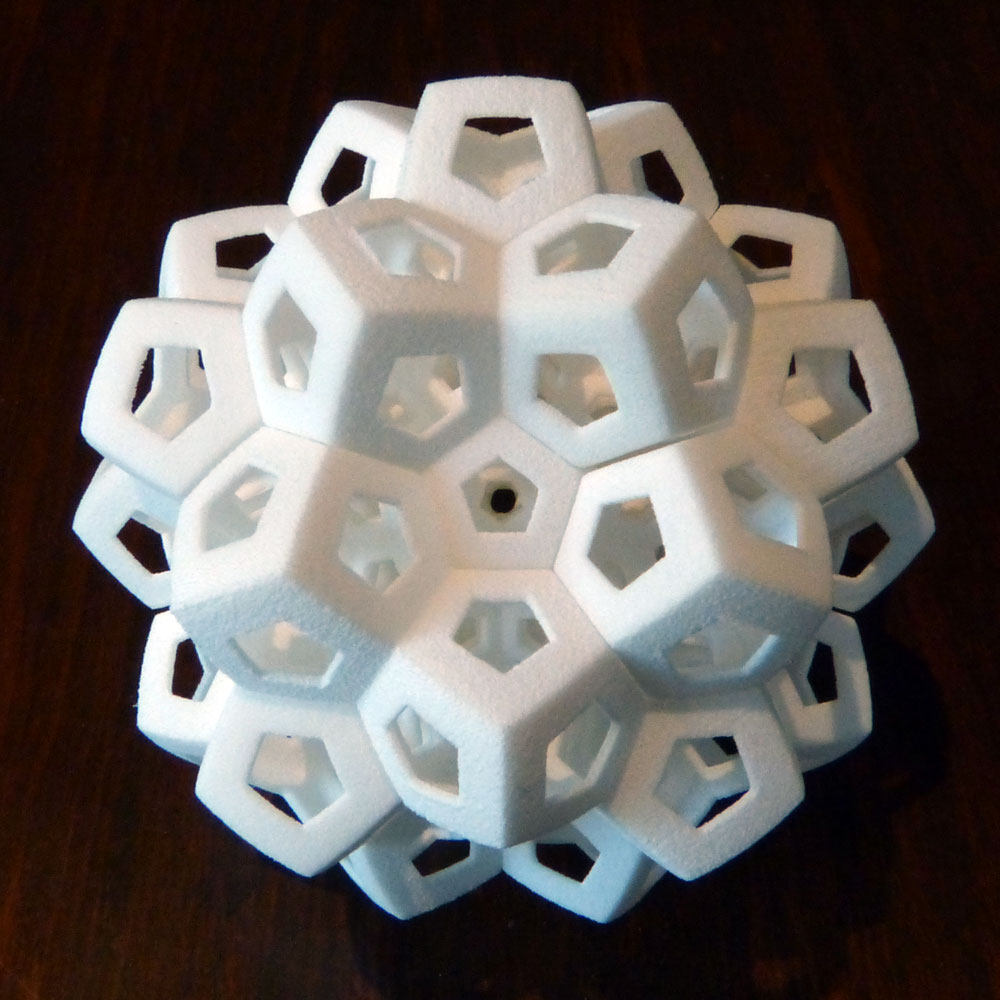}
}\\
 \cline{1-1}
  $5\times \text{inner six}$ &\\
  $5\times \text{outer six}$ &\\
  $1\times \text{spine}$ &\\
  $2\times \text{equator}$&\\
  &\\
\end{tabular}

\end{document}